\documentclass{amsart}

\usepackage{amssymb}
\usepackage{amsmath}
\usepackage{amsthm}
\usepackage{amsfonts}

\usepackage{a4wide}
\usepackage{enumerate}

\newtheorem{theorem}{Theorem}[section]

\newtheorem{corollary}{Corollary}[section]

\theoremstyle{definition}
\newtheorem{definition}{Definition}[section]
\newtheorem{remark}{Remark}[section]

\newtheorem{conjecture}{Conjecture}[section]

\numberwithin{equation}{section}

\begin{document}


\title[Periodic bifurcations in descendant trees]
{Periodic bifurcations in \\
descendant trees of finite \(p\)-groups}

\author{Daniel C. Mayer}
\address{Naglergasse 53\\8010 Graz\\Austria}
\email{algebraic.number.theory@algebra.at}
\urladdr{http://www.algebra.at}

\thanks{Research supported by the Austrian Science Fund (FWF): P 26008-N25}


\date{February 11, 2015}

\begin{abstract}
Theoretical background and an implementation
of the \(p\)-group generation algorithm by Newman and O'Brien
are used to provide computational evidence of
a new type of periodically repeating patterns
in pruned descendant trees of finite \(p\)-groups.
\end{abstract}

\maketitle



\section{Introduction}
\label{s:Intro}

In \S\S\
\ref{s:Structure}
--
\ref{s:HistoryDescTrees},
we present an exposition of facts concerning the mathematical \textit{structure}
which forms the central idea of this article:
descendant trees of finite \(p\)-groups.
Their computational \textit{construction} is recalled in \S\S\
\ref{s:Construction}
--
\ref{s:PruningStrategies}
on the \(p\)-group generation algorithm.
Recently discovered periodic patterns in descendant trees
with promising arithmetical applications
form the topic of the final \S\
\ref{s:PeriodicBifurcations}
and the coronation of the entire work.



\section{The structure: descendant trees}
\label{s:Structure}

In mathematics, specifically group theory,
a \textit{descendant tree} is a hierarchical structure
for visualizing parent-descendant relations (\S\S\
\ref{s:Terminology}
and
\ref{s:TreeDiagram})
between isomorphism classes of finite groups of prime power order \(p^n\),
for a fixed prime number \(p\) and varying integer exponents \(n\ge 0\).
Such groups are briefly called finite \(p\)-groups.
The \textit{vertices} of a descendant tree are isomorphism classes of finite \(p\)-groups.

Additionally to their order \(p^n\), finite \(p\)-groups possess two further related invariants,
the nilpotency class \(c\) and the \textit{coclass} \(r:=n-c\) (\S\S\
\ref{s:CoclassTrees}
and
\ref{s:Multifurcation}).
It turned out that descendant trees of a particular kind,
the so-called \textit{pruned coclass trees} whose infinitely many vertices share a common coclass \(r\),
reveal a \textit{repeating finite pattern} (\S\
\ref{s:VirtualPeriodicity}).
These two crucial properties of finiteness and periodicity,
which have been proved independently by M. du Sautoy
\cite{dS}
and by B. Eick and C.R. Leedham-Green
\cite{EkLg},
admit a characterization of all members of the tree
by finitely many \textit{parametrized presentations} (\S\S\
\ref{s:ConcreteExamples}
and
\ref{s:PeriodicBifurcations}).
Consequently, descendant trees play a fundamental role
in the classification of finite \(p\)-groups.
By means of kernels and targets of \textit{Artin transfer} homomorphisms
\cite{Ar2},
descendant trees can be endowed with additional structure
\cite{Ma2,Ma3,Ma4},
which recently turned out to be decisive for \textit{arithmetical applications} in class field theory,
in particular, for determining the exact length of \(p\)-class towers
\cite{BuMa}.

An important question is how the descendant tree \(\mathcal{T}(R)\) can actually be constructed
for an assigned starting group which is taken as the root \(R\) of the tree.
Sections \S\S\
\ref{s:LowerExponentP} --
\ref{s:SchurMpl}
are devoted to recall a minimum of the necessary background concerning
the \textit{\(p\)-group generation algorithm} by M.F. Newman
\cite{Nm2}
and E.A. O'Brien
\cite{Ob,HEO},
which is a recursive process for constructing the descendant tree
of a foregiven finite \(p\)-group playing the role of the tree root.
This algorithm is now implemented in the ANUPQ-package
\cite{GNO}
of the computational algebra systems GAP
\cite{GAP}
and MAGMA
\cite{MAGMA}.

As a final highlight in \S\
\ref{s:PeriodicBifurcations},
whose formulation requires an understanding of all the preceding sections,
this article concludes with brand-new discoveries of
an unknown, and up to now unproved, kind of repeating infinite patterns
called \textit{periodic bifurcations},
which appeared in extensive computational constructions of descendant trees
of certain finite \(2\)-groups, resp. \(3\)-groups,
\(G\) with abelianization \(G/G^\prime\) of type \((2,2,2)\), resp. \((3,3)\),
and have immediate applications in algebraic number theory and class field theory.



\section{Historical remarks on bifurcation}
\label{s:HistoricalRmksBifurcation}

Since computer aided classifications of finite \(p\)-groups go back to 1975,
fourty years ago, there arises the question
why periodic bifurcations did not show up in the earlier literature already.
At the first sight, this fact seems incomprehensible,
because the smallest two \(3\)-groups which reveal
the phenomenon of periodic bifurcations with modest complexity
were well known to both,
J. A. Ascione, G. Havas and C.R. Leedham-Green
\cite{AHL}
and B. Nebelung
\cite{Ne}.
Their SmallGroups identifiers are \(\langle 729,49\rangle\) and \(\langle 729,54\rangle\)
(see \S\
\ref{s:Identifiers}
and
\cite{BEO1,BEO2}).
Due to the lack of systematic identifiers in 1977,
they were called the \textit{non-CF groups} \(Q\) and \(U\) in
\cite[Tbl.1, p.265, and Tbl.2, p.266]{AHL},
since their lower central series \((\gamma_j(G))_{j\ge 1}\)
has a non-cyclic factor \(\gamma_3(G)/\gamma_4(G)\) of type \((3,3)\).
Similarly, there was no SmallGroups Database yet in 1989, whence the two groups
were designated by \(G_0^{5,6}(0,-1,0,1)\) and  \(G_0^{5,6}(0,0,0,1)\) in
\cite[Satz 6.14, p.208]{Ne}.

So Ascione and Nebelung were both standing in front of the door to a realm of uncharted waters.
The reason why they did not enter this door was the sharp definition of their project targets.
A \textit{bifurcation} is the special case of a \(2\)-fold multifurcation (\S\
\ref{s:Multifurcation}):
At a vertex \(G\) of coclass \(\mathrm{cc}(G)=r\) with nuclear rank \(\nu(G)=2\),
the descendant tree \(\mathcal{T}(G)\) forks into
a \textit{regular} component of the same coclass \(\mathcal{T}^r(G)\) and
an \textit{irregular} component of the next coclass \(\mathcal{T}^{r+1}(G)\).

Ascione's thesis subject
\cite{As1,As2}
in 1979 was to investigate
two-generated \(3\)-groups \(G\) of second maximal class, that is, of coclass \(\mathrm{cc}(G)=2\).
Consequently, she studied the regular component \(\mathcal{T}^2(G)\) for \(G\in\lbrace Q,U\rbrace\)
and did not touch the irregular tree \(\mathcal{T}^3(G)\)
whose members are not of second maximal class.

The goal of Nebelung's dissertation
\cite{Ne}
in 1989 was
the classification of metabelian \(3\)-groups \(G\) with \(G/G^\prime\) of type \((3,3)\).
Therefore she focused on the \textit{metabelian skeleton} \(\mathcal{T}_\ast^2(G)\)
of the regular coclass tree \(\mathcal{T}^2(G)\) for \(G\in\lbrace Q,U\rbrace\)
(a special case of a \textit{pruned} coclass tree, see \S\
\ref{s:VirtualPeriodicity})
and omitted the irregular component \(\mathcal{T}^3(G)\)
whose members are entirely non-metabelian of derived length \(3\).



\section{Definitions and terminology}
\label{s:Terminology}

According to M.F. Newman
\cite[\S\ 2, pp.52--53]{Nm},
there exist several distinct definitions
of the \textit{parent} \(\pi(G)\)
of a finite \(p\)-group \(G\).
The common principle is to form
the quotient \(\pi(G):=G/N\) of \(G\)
by a suitable normal subgroup \(N\unlhd G\)
which can be either

\begin{enumerate}[({P}1)]
\item
the centre \(N=\zeta_1(G)\) of \(G\),
whence \(\pi(G)=G/\zeta_1(G)\) is called \textit{central quotient} of \(G\) or
\item
the last non-trivial term \(N=\gamma_c(G)\) of the lower central series of \(G\),
where \(c\) denotes the nilpotency class of \(G\) or
\item
the last non-trivial term \(N=P_{c-1}(G)\) of the lower exponent-\(p\) central series of \(G\),
where \(c\) denotes the exponent-\(p\) class of \(G\) or
\item
the last non-trivial term \(N=G^{(d-1)}\) of the derived series of \(G\),
where \(d\) denotes the derived length of \(G\).
\end{enumerate}

In each case,
\(G\) is called an \textit{immediate descendant} of \(\pi(G)\)
and a \textit{directed edge} of the tree is defined either by \(G\to\pi(G)\)
in the direction of the canonical projection \(\pi:G\to\pi(G)\) onto the quotient \(\pi(G)=G/N\)
or by \(\pi(G)\to G\) in the opposite direction, which is more usual for descendant trees.
The former convention is adopted by Leedham-Green and Newman
\cite[\S\ 2, pp.194--195]{LgNm},
by du Sautoy and D. Segal
\cite[\S\ 7, p.280]{dSSg},
by Leedham-Green and S. McKay
\cite[Dfn.8.4.1, p.166]{LgMk},
and by Eick, Leedham-Green, Newman and O'Brien
\cite[\S\ 1]{ELNO}.
The latter definition is used by Newman
\cite[\S\ 2, pp.52--53]{Nm},
by Newman and O'Brien
\cite[\S\ 1, p.131]{NmOb},
by du Sautoy
\cite[\S\ 1, p.67]{dS},
by H. Dietrich, Eick and D. Feichtenschlager
\cite[\S\ 2, p.46]{DEF}
and by Eick and Leedham-Green
\cite[\S\ 1, p.275]{EkLg}.

In the following, the direction of the canonical projections is selected for all edges.
Then, more generally, a vertex \(R\) is a \textit{descendant} of a vertex \(P\),
and \(P\) is an \textit{ancestor} of \(R\),
if either \(R\) is equal to \(P\)
or there is a \textit{path}

\begin{equation}
\label{eqn:Path}
R=Q_0\to Q_1\to\cdots\to Q_{m-1}\to Q_m=P,\text{ with }m\ge 1,
\end{equation}

\noindent
of directed edges from \(R\) to \(P\).
The vertices forming the path necessarily coincide with
the \textit{iterated parents} \(Q_j=\pi^{j}(R)\) of \(R\), with \(0\le j\le m\):

\begin{equation}
\label{eqn:IteratedParents}
R=\pi^{0}(R)\to\pi^{1}(R)\to\cdots\to\pi^{m-1}(R)\to\pi^{m}(R)=P,\text{ with }m\ge 1.
\end{equation}

\noindent
In the most important special case (P2) of parents defined as last non-trivial lower central quotients,
they can also be viewed as the \textit{successive quotients} \(R/\gamma_{c+1-j}(R)\) \textit{of class} \(c-j\) of \(R\)
when the nilpotency class of \(R\) is given by \(c\ge m\):

\begin{equation}
\label{eqn:SuccessiveQuotients}
R\simeq R/\gamma_{c+1}(R)\to R/\gamma_{c}(R)\to\cdots\to R/\gamma_{c+2-m}(R)\to R/\gamma_{c+1-m}(R)\simeq P,
\end{equation}

\noindent
with \(c\ge m\ge 1\).

Generally, the \textit{descendant tree} \(\mathcal{T}(G)\) of a vertex \(G\) is
the subtree of all descendants of \(G\), starting at the \textit{root} \(G\).
The \textit{maximal} possible descendant tree \(\mathcal{T}(1)\) of the trivial group \(1\)
contains all finite \(p\)-groups and is somewhat exceptional, since,
for any parent definition (P1--P4),
the trivial group \(1\) has infinitely many abelian \(p\)-groups as its immediate descendants.
The parent definitions (P2--P3) have the advantage that
any non-trivial finite \(p\)-group (of order divisible by \(p\))
possesses only finitely many immediate descendants.



\section{Pro-\(p\) groups and coclass trees}
\label{s:CoclassTrees}

For a sound understanding of coclass trees as a particular instance of descendant trees,
it is necessary to summarize some facts concerning \textit{infinite topological pro-\(p\) groups}.
The members \(\gamma_j(S)\), with \(j\ge 1\), of the lower central series of a pro-\(p\) group \(S\)
are open and closed subgroups of finite index,
and therefore the corresponding quotients \(S/\gamma_j(S)\) are finite \(p\)-groups.
The pro-\(p\) group \(S\) is said to be of \textit{coclass} \(\mathrm{cc}(S):=r\)
when the limit \(r=\lim_{j\to\infty}\,\mathrm{cc}(S/\gamma_j(S))\)
of the coclass of the successive quotients exists and is finite.
An infinite pro-\(p\) group \(S\) of coclass \(r\) is a \textit{\(p\)-adic pre-space group}
\cite[Dfn.7.4.11, p.147]{LgMk},
since it has a normal subgroup \(T\), the \textit{translation group},
which is a free module over the ring \(\mathbb{Z}_p\) of \(p\)-adic integers
of uniquely determined rank \(d\), the \textit{dimension},
such that the quotient \(P=S/T\) is a finite \(p\)-group, the \textit{point group},
which \textit{acts on \(T\) uniserially}.
The dimension is given by

\begin{equation}
\label{eqn:Dimension}
d=(p-1)p^{s},\text{ with some }0\le s<r.
\end{equation}

A central finiteness result for infinite pro-\(p\) groups of coclass \(r\)
is provided by the so-called \textit{Theorem D},
which is one of the five \textit{Coclass Theorems} proved in 1994 independently by A. Shalev
\cite{Sv}
and by C.R. Leedham-Green
\cite[Thm.7.7, p.66]{Lg},
and conjectured in 1980 already by Leedham-Green and Newman
\cite[\S\ 2, pp.194--196]{LgNm}.
Theorem D asserts that there are only finitely many isomorphism classes
of infinite pro-\(p\) groups of coclass \(r\),
for any fixed prime \(p\) and any fixed non-negative integer \(r\).
As a consequence, if \(S\) is an infinite pro-\(p\) group of coclass \(r\), then
there exists a minimal integer \(i\ge 1\) such that
the following three conditions are satisfied for any integer \(j\ge i\).

\begin{itemize}
\item
\(\mathrm{cc}(S/\gamma_j(S))=r\),
\item
\(S/\gamma_j(S)\) is not a lower central quotient of any infinite pro-\(p\) group of coclass \(r\)
which is not isomorphic to \(S\),
\item
\(\gamma_j/\gamma_{j+1}(S)\) is cyclic of order \(p\).
\end{itemize}

The descendant tree \(\mathcal{T}(R)\), with respect to the parent definition (P2),
of the root \(R=S/\gamma_i(S)\) with minimal \(i\)
is called the \textit{coclass tree} \(\mathcal{T}(S)\) of \(S\)
and its unique maximal infinite (reverse-directed) path

\begin{equation}
\label{eqn:MainLine}
R=S/\gamma_i(S)\leftarrow S/\gamma_{i+1}(S)\leftarrow S/\gamma_{i+2}(S)\leftarrow\cdots
\end{equation}

\noindent
is called the \textit{mainline} (or trunk) of the tree.



\section{Tree diagram}
\label{s:TreeDiagram}

Further terminology, used in diagrams visualizing finite parts of descendant trees,
is explained in Figure
\ref{fig:TreeNotation}
by means of an artificial abstract tree.
On the left hand side, a \textit{level} indicates the basic top-down design of a descendant tree.
For concrete trees, such as those in Figures
\ref{fig:2GroupsCoclass1},
resp.
\ref{fig:3GroupsCoclass1},
etc.,
the level is usually replaced by a scale of orders increasing from the top to the bottom.
A vertex is \textit{capable} (or \textit{extendable}) if it has at least one immediate descendant,
otherwise it is \textit{terminal} (or a \textit{leaf}).
Vertices sharing a common parent are called \textit{siblings}.

{\tiny

\begin{figure}[ht]
\caption{Terminology for descendant trees}
\label{fig:TreeNotation}


\setlength{\unitlength}{1cm}
\begin{picture}(10,12)(-4,-11)

\put(-4,0.5){\makebox(0,0)[cb]{Level \(n\)}}
\put(-4,0){\line(0,-1){10}}
\multiput(-4.1,0)(0,-1){11}{\line(1,0){0.2}}
\put(-4.2,0){\makebox(0,0)[rc]{\(0\)}}
\put(-4.2,-1){\makebox(0,0)[rc]{\(1\)}}
\put(-4.2,-2){\makebox(0,0)[rc]{\(2\)}}
\put(-4.2,-3){\makebox(0,0)[rc]{\(3\)}}
\put(-4.2,-4){\makebox(0,0)[rc]{\(4\)}}
\put(-4.2,-5){\makebox(0,0)[rc]{\(5\)}}
\put(-4.2,-6){\makebox(0,0)[rc]{\(6\)}}
\put(-4.2,-7){\makebox(0,0)[rc]{\(7\)}}
\put(-4.2,-8){\makebox(0,0)[rc]{\(8\)}}
\put(-4.2,-9){\makebox(0,0)[rc]{\(9\)}}
\put(-4.2,-10){\makebox(0,0)[rc]{\(10\)}}
\put(-4,-10){\vector(0,-1){1}}

\put(0,0){\circle*{0.2}}
\multiput(0,-1)(0,-1){9}{\circle*{0.1}}

\put(-1,-1){\circle*{0.1}}

\multiput(-3,-2)(5,0){2}{\circle*{0.1}}
\put(-3,-3){\circle*{0.1}}
\multiput(1,-3)(1,0){3}{\circle*{0.1}}

\put(-1,-4){\circle*{0.1}}
\put(2,-5){\circle*{0.1}}

\multiput(-1,-6)(0,-1){4}{\circle*{0.1}}
\multiput(-2,-6)(0,-1){4}{\circle*{0.1}}
\multiput(-3,-6)(0,-2){2}{\circle*{0.1}}
\multiput(1,-6)(0,-1){4}{\circle*{0.1}}
\multiput(2,-6)(0,-1){4}{\circle*{0.1}}
\multiput(3,-6)(0,-2){2}{\circle*{0.1}}
\multiput(4,-6)(0,-2){2}{\circle*{0.1}}

\multiput(0,0)(0,-1){9}{\line(0,-1){1}}
\put(0,0){\line(-1,-1){1}}

\put(0,-1){\line(-3,-1){3}}
\put(-3,-3){\vector(0,1){0.9}}
\put(0,-1){\line(2,-1){2}}
\put(2,-2){\line(-1,-1){1}}
\put(2,-2){\line(0,-1){1}}
\put(2,-2){\line(1,-1){1}}

\put(0,-3){\line(-1,-1){1}}
\put(0,-3){\line(1,-1){2}}

\multiput(0,-5)(0,-1){4}{\line(-1,-1){1}}
\multiput(0,-5)(0,-1){4}{\line(-2,-1){2}}
\multiput(0,-5)(0,-2){2}{\line(-3,-1){3}}
\multiput(0,-5)(0,-1){4}{\line(1,-1){1}}
\multiput(0,-5)(0,-1){4}{\line(2,-1){2}}
\multiput(0,-5)(0,-2){2}{\line(3,-1){3}}
\multiput(0,-5)(0,-2){2}{\line(4,-1){4}}

\put(-0.2,0){\makebox(0,0)[rc]{root \(R\)}}

\put(-1.2,-1){\makebox(0,0)[rc]{leaf}}
\put(0.2,-1){\makebox(0,0)[lc]{capable vertex}}

\put(-2.7,-2){\makebox(0,0)[lc]{parent \(\pi(V)\)}}
\put(-3.1,-2.5){\makebox(0,0)[rc]{\(\pi\)}}
\put(-2.9,-2.5){\makebox(0,0)[lc]{directed edge}}
\put(-2.7,-3){\makebox(0,0)[lc]{descendant \(V\)}}

\put(1.2,-1.5){\makebox(0,0)[lc]{branch \(\mathcal{B}(1)\)}}
\put(3.5,-2){\vector(0,1){1}}
\put(3.7,-2){\makebox(0,0)[lc]{depth \(2\)}}
\put(3.5,-2){\vector(0,-1){1}}
\put(2,-3.2){\makebox(0,0)[ct]{three siblings}}

\put(-0.7,-3.5){\makebox(0,0)[rc]{bifurcation:}}
\put(-1.2,-4){\makebox(0,0)[rc]{\(\mathcal{B}_1(3)\), regular}}
\put(1.2,-4){\makebox(0,0)[lc]{edge of depth \(2\)}}
\put(2.2,-4.7){\makebox(0,0)[lc]{\(\mathcal{B}_2(3)\), irregular}}

\put(-0.2,-5){\makebox(0,0)[rc]{periodic root \(P\)}}
\put(2.2,-5.5){\makebox(0,0)[lc]{\(\mathcal{B}(5)\)}}
\put(2.2,-6.5){\makebox(0,0)[lc]{\(\mathcal{B}(6)\)}}
\put(2.2,-7.5){\makebox(0,0)[lc]{\(\mathcal{B}(7)\)}}
\put(2.2,-8.5){\makebox(0,0)[lc]{\(\mathcal{B}(8)\)}}
\put(4.5,-7){\vector(0,1){1}}
\put(4.7,-7){\makebox(0,0)[lc]{period length \(2\)}}
\put(4.5,-7){\vector(0,-1){1}}

\put(0,-9){\vector(0,-1){2}}
\put(0.2,-10){\makebox(0,0)[lc]{infinite}}
\put(0.2,-10.5){\makebox(0,0)[lc]{mainline}}
\put(-0.2,-10.5){\makebox(0,0)[rc]{tree \(\mathcal{T}(R)\)}}

\end{picture}

\end{figure}
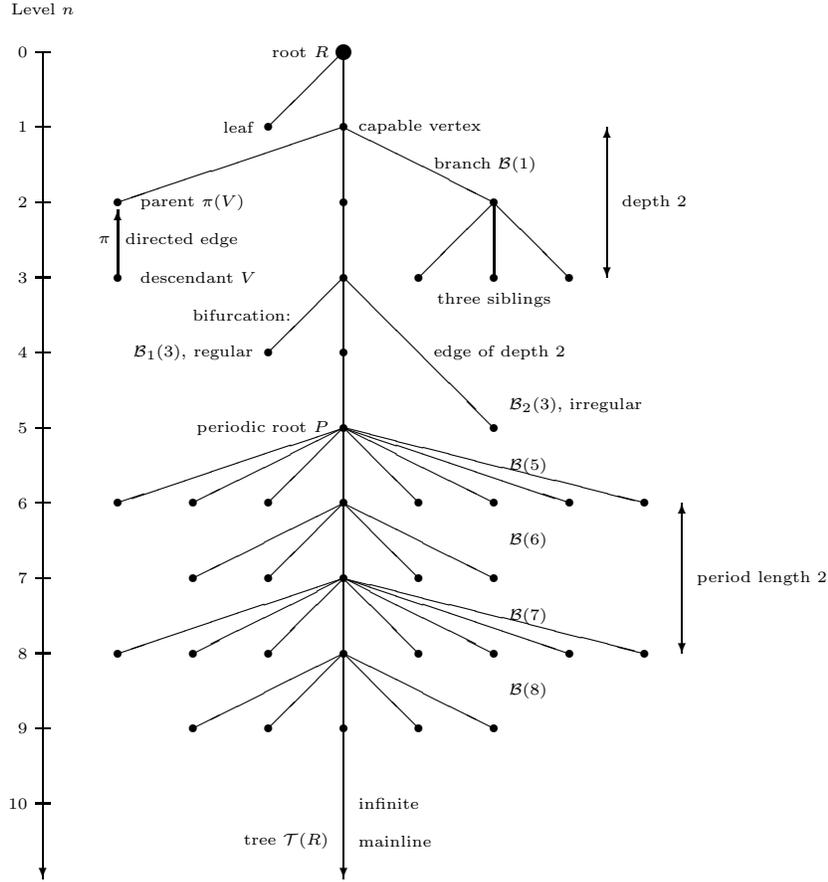

}

If the descendant tree is a coclass tree \(\mathcal{T}(R)\) with root \(R=R_0\)
and with mainline vertices \((R_n)_{n\ge 0}\) labelled according to the level \(n\),
then the finite subtree defined as the difference set

\begin{equation}
\label{eqn:Branch}
\mathcal{B}(n):=\mathcal{T}(R_n)\setminus\mathcal{T}(R_{n+1})
\end{equation}

\noindent
is called the \(n\)th \textit{branch} (or twig) of the tree
or also the branch \(\mathcal{B}(R_n)\) with root \(R_n\), for any \(n\ge 0\).
The \textit{depth} of a branch is
the maximal length of the paths connecting its vertices with its root.

Figure
\ref{fig:TreeNotation}
shows a descendant tree whose branches \(\mathcal{B}(2),\mathcal{B}(4)\) both have depth \(0\),
and \(\mathcal{B}(5)\simeq\mathcal{B}(7)\), resp. \(\mathcal{B}(6)\simeq\mathcal{B}(8)\), are isomorphic as trees.

If all vertices of depth bigger than a given integer \(k\ge 0\)
are removed from branch \(\mathcal{B}(n)\),
then we obtain the (depth-)\textit{pruned branch} \(\mathcal{B}_k(n)\).
Correspondingly, the \textit{pruned coclass tree} \(\mathcal{T}_k(R)\),
resp. the entire coclass tree \(\mathcal{T}(R)\),
consists of the infinite sequence of its pruned branches \((\mathcal{B}_k(n))_{n\ge 0}\),
resp. branches \((\mathcal{B}(n))_{n\ge 0}\),
connected by the mainline, whose vertices \(R_n\) are called \textit{infinitely capable}.



\section{Virtual periodicity}
\label{s:VirtualPeriodicity}

The periodicity of branches of depth-pruned coclass trees
has been proved with analytic methods using zeta functions
\cite[\S\ 7, Thm.15, p.280]{dSSg}
of groups by M. du Sautoy
\cite[Thm.1.11, p.68, and Thm.8.3, p.103]{dS},
and with algebraic techniques using cohomology groups by B. Eick and C.R. Leedham-Green
\cite{EkLg}.
The former methods admit the qualitative insight of ultimate virtual periodicity,
the latter techniques determine the quantitative structure.



\begin{theorem}
\label{thm:Periodicity}
For any infinite pro-\(p\) group \(S\) of coclass \(r\ge 1\) and dimension \(d\),
and for any given depth \(k\ge 1\),
there exists an effective minimal lower bound \(f(k)\ge 1\),
where \textit{periodicity of length} \(d\) of depth-\(k\) pruned branches
of the coclass tree \(\mathcal{T}(S)\) sets in,
that is, there exist graph isomorphisms

\begin{equation}
\label{eqn:Periodicity}
\mathcal{B}_k(n+d)\simeq\mathcal{B}_k(n),\text{ for all }n\ge f(k).
\end{equation}

\end{theorem}

\begin{proof}
The graph isomorphisms of depth-\(k\) pruned banches with roots of sufficiently large order
\(n\ge f(k)\)
are derived with cohomological methods in
\cite[Thm.6, p.277, Thm.9, p.278]{EkLg}
and the effective lower bound \(f(k)\) for the branch root orders is established in
\cite[Thm.29, p.287]{EkLg}.
\end{proof}

This central result can be expressed ostensively:
When we look at a coclass tree through a pair of blinkers
and ignore a finite number of pre-periodic branches at the top,
then we shall see a repeating finite pattern (\textit{ultimate} periodicity).
However, if we take wider blinkers
the pre-periodic initial section may become longer (\textit{virtual} periodicity).

The vertex \(P=R_{f(k)}\) is called the \textit{periodic root} of the pruned coclass tree,
for a fixed value of the depth \(k\). See Figure
\ref{fig:TreeNotation}.



\section{Multifurcation and coclass graphs}
\label{s:Multifurcation}

Assume that parents of finite \(p\)-groups are defined as last non-trivial lower central quotients (P2).
For a \(p\)-group \(G\) of coclass \(\mathrm{cc}(G)=r\),
we can distinguish its (entire) descendant tree \(\mathcal{T}(G)\)
and its \textit{coclass-\(r\) descendant tree} \(\mathcal{T}^r(G)\),
the subtree consisting of descendants of coclass \(r\) only.
The group \(G\) is \textit{coclass-settled} if \(\mathcal{T}(G)=\mathcal{T}^r(G)\).

The \textit{nuclear rank} \(\nu(G)\) of \(G\) (see \S\
\ref{s:CoveringGroup})
in the theory of the \(p\)-group generation algorithm by M.F. Newman
\cite{Nm2}
and E.A. O'Brien
\cite{Ob}
provides the following criteria.

\begin{itemize}
\item
\(G\) is terminal, and thus trivially coclass-settled, if and only if \(\nu(G)=0\).
\item
If \(\nu(G)=1\), then \(G\) is capable, but it remains unknown whether \(G\) is coclass-settled.
\item
If \(\nu(G)=m\ge 2\), then \(G\) is capable and certainly not coclass-settled.
\end{itemize}

In the last case, a more precise assertion is possible:
If \(G\) has coclass \(r\) and nuclear rank \(\nu(G)=m\ge 2\), then it gives rise to
an \(m\)-fold \textit{multifurcation}
into a \textit{regular} coclass-\(r\) descendant tree \(\mathcal{T}^r(G)\)
and \(m-1\) \textit{irregular} descendant trees \(\mathcal{T}^{r+j}(G)\) of coclass \(r+j\),
for \(1\le j\le m-1\).
Consequently, the descendant tree of \(G\) is the disjoint union

\begin{equation}
\label{eqn:Components}
\mathcal{T}(G)=\dot{\cup}_{j=0}^{m-1}\,\mathcal{T}^{r+j}(G).
\end{equation}

Multifurcation is correlated with different orders of the last non-trivial lower central of immediate descendants.
Since the nilpotency class increases exactly by a unit,
\(c=\mathrm{cl}(Q)=\mathrm{cl}(P)+1\),
from a parent \(P=\pi(Q)\) to any immediate descendant \(Q\),
the coclass remains stable, \(r=\mathrm{cc}(Q)=\mathrm{cc}(P)\), if \(\vert\gamma_c(Q)\vert=p\).
In this case,
\(Q\) is a \textit{regular} immediate descendant with directed edge \(P\leftarrow Q\) of depth 1,
as usual.
However, the coclass increases by \(m-1\), if \(\vert\gamma_c(Q)\vert=p^m\) with \(m\ge 2\).
Then \(Q\) is called an \textit{irregular} immediate descendant with directed edge of \textit{depth} \(m\).

If the condition of depth (or \textit{step size}) 1 is imposed on all directed edges,
then the maximal descendant tree \(\mathcal{T}(1)\) of the trivial group \(1\)
splits into a countably infinite disjoint union

\begin{equation}
\label{eqn:CoclassGraphs}
\mathcal{T}(1)=\dot{\cup}_{r=0}^\infty\,\mathcal{G}(p,r)
\end{equation}

\noindent
of directed \textit{coclass graphs} \(\mathcal{G}(p,r)\),
which are rather \textit{forests} than trees.
More precisely, the above mentioned Coclass Theorems imply that

\begin{equation}
\label{eqn:SporadicPart}
\mathcal{G}(p,r)=\left(\dot{\cup}_i\,\mathcal{T}(S_i)\right)\dot{\cup}\mathcal{G}_0(p,r)
\end{equation}

\noindent
is the disjoint union of
finitely many coclass trees \(\mathcal{T}(S_i)\)
of pairwise non-isomorphic infinite pro-\(p\) groups \(S_i\) of coclass \(r\) (Theorem D)
and a finite subgraph \(\mathcal{G}_0(p,r)\) of \textit{sporadic groups} lying outside of any coclass tree.



\section{Identifiers}
\label{s:Identifiers}

The SmallGroups Library \textit{identifiers} of finite groups, in particular \(p\)-groups, given in the form
\[\langle\text{order},\text{counting number}\rangle\]
in the following concrete examples of descendant trees,
are due to H.U. Besche, B. Eick and E.A. O'Brien
\cite{BEO1,BEO2}.
When the group orders are given in a scale on the left hand side as in Figure
\ref{fig:2GroupsCoclass1}
and Figure
\ref{fig:3GroupsCoclass1},
the identifiers are briefly denoted by
\[\langle\text{counting number}\rangle.\]

Depending on the prime \(p\), there is an upper bound on the order of groups for which a SmallGroup identifier exists,
e.g. \(512=2^9\) for \(p=2\), and \(2187=3^7\) for \(p=3\).
For groups of bigger orders, a notation with \textit{generalized identifiers}
resembling the descendant structure is employed:
A regular immediate descendant, connected by an edge of depth \(1\) with its parent \(P\),
is denoted by
\[P-\#1;\text{counting number},\]
and
an irregular immediate descendant, connected by an edge of depth \(d\ge 2\) with its parent \(P\),
is denoted by
\[P-\#d;\text{counting number}.\]
The ANUPQ package
\cite{GNO}
containing the implementation of the \(p\)-group generation algorithm
uses this notation, which goes back to J.A. Ascione in 1979
\cite{As1}.



\section{Concrete examples of trees}
\label{s:ConcreteExamples}

In all examples, the underlying parent definition (P2) corresponds to the usual lower central series.
Occasional differences to the parent definition (P3)
with respect to the lower exponent-\(p\) central series are pointed out.



\subsection{Coclass \(0\)}
\label{ss:CoclassZero}

The coclass graph

\begin{equation}
\label{eqn:CoclassZero}
\mathcal{G}(p,0)=\mathcal{G}_0(p,0)
\end{equation}

\noindent
of finite \(p\)-groups of coclass \(0\)
does not contain a coclass tree and
consists of the \textit{trivial group} \(1\) and the \textit{cyclic group} \(C_p\) of order \(p\),
which is a leaf (however, it is capable with respect to the lower exponent-\(p\) central series).
For \(p=2\) the SmallGroup identifier of \(C_p\) is \(\langle 2,1\rangle\),
for \(p=3\) it is \(\langle 3,1\rangle\).



\subsection{Coclass \(1\)}
\label{ss:CoclassOne}

The coclass graph

\begin{equation}
\label{eqn:CoclassOne}
\mathcal{G}(p,1)=\mathcal{T}^1(R)\dot{\cup}\mathcal{G}_0(p,1)
\end{equation}

\noindent
of finite \(p\)-groups of coclass \(1\)
consists of the unique coclass tree with root \(R=C_p\times C_p\),
the \textit{elementary abelian \(p\)-group of rank \(2\)},
and a single \textit{isolated vertex}
(a terminal orphan without proper parent in the same coclass graph,
since the directed edge to the trivial group \(1\) has depth \(2\)),
the \textit{cyclic group} \(C_{p^2}\) of order \(p^2\) in the sporadic part \(\mathcal{G}_0(p,1)\)
(however, this group is capable with respect to the lower exponent-\(p\) central series).
The tree \(\mathcal{T}^1(R)=\mathcal{T}^1(S_1)\) is
the coclass tree of the unique infinite pro-\(p\) group \(S_1\) of coclass \(1\).

For \(p=2\), resp. \(p=3\), the SmallGroup identifier of the root \(R\) is
\(\langle 4,2\rangle\), resp. \(\langle 9,2\rangle\),
and a tree diagram of the coclass graph from branch \(\mathcal{B}(2)\) up to branch \(\mathcal{B}(7)\)
(counted with respect to the \(p\)-logarithm of the order of the branch root)
is drawn in Figure
\ref{fig:2GroupsCoclass1},
resp. Figure
\ref{fig:3GroupsCoclass1},
where all groups of order at least \(p^3\) are \textit{metabelian},
that is non-abelian with derived length \(2\)
(vertices represented by black discs in contrast to contour squares indicating abelian groups).
In Figure
\ref{fig:3GroupsCoclass1},
smaller black discs denote metabelian 3-groups where even the maximal subgroups are non-abelian,
a feature which does not occur for the metabelian 2-groups in Figure
\ref{fig:2GroupsCoclass1},
since they all possess an abelian subgroup of index \(p\) (usually exactly one).
The coclass tree of \(\mathcal{G}(2,1)\), resp. \(\mathcal{G}(3,1)\),
has periodic root \(\langle 8,3\rangle\) and period of length \(1\) starting with branch \(\mathcal{B}(3)\),
resp. periodic root \(\langle 81,9\rangle\) and period of length \(2\) starting with branch \(\mathcal{B}(4)\).
Both trees have branches of bounded depth \(1\),
so their virtual periodicity is in fact a \textit{strict periodicity}.
The \(\Gamma_s\), resp. \(\Phi_s\), denote isoclinism families
\cite{HaSn,Hl}.

However, the coclass tree of \(\mathcal{G}(p,1)\) with \(p\ge 5\)
has \textit{unbounded depth} and contains non-metabelian groups,
and the coclass tree of \(\mathcal{G}(p,1)\) with \(p\ge 7\) has even \textit{unbounded width},
that is the number of descendants of a fixed order increases indefinitely with growing order
\cite{DEF}.

With the aid of kernels and targets of Artin transfer homomorphisms
\cite{Ar2},
the diagrams in Figures 
\ref{fig:2GroupsCoclass1}
and
\ref{fig:3GroupsCoclass1}
can be endowed with additional information
and redrawn as \textit{structured descendant trees}
\cite[Fig.3.1, p.419, and Fig.3.2, p.422]{Ma4}.



{\tiny

\begin{figure}[ht]
\caption{\(2\)-Groups of Coclass \(1\)}
\label{fig:2GroupsCoclass1}


\setlength{\unitlength}{1cm}
\begin{picture}(12,16)(-8,-15)

\put(-8,0.5){\makebox(0,0)[cb]{Order \(2^n\)}}
\put(-8,0){\line(0,-1){12}}
\multiput(-8.1,0)(0,-2){7}{\line(1,0){0.2}}
\put(-8.2,0){\makebox(0,0)[rc]{\(4\)}}
\put(-7.8,0){\makebox(0,0)[lc]{\(2^2\)}}
\put(-8.2,-2){\makebox(0,0)[rc]{\(8\)}}
\put(-7.8,-2){\makebox(0,0)[lc]{\(2^3\)}}
\put(-8.2,-4){\makebox(0,0)[rc]{\(16\)}}
\put(-7.8,-4){\makebox(0,0)[lc]{\(2^4\)}}
\put(-8.2,-6){\makebox(0,0)[rc]{\(32\)}}
\put(-7.8,-6){\makebox(0,0)[lc]{\(2^5\)}}
\put(-8.2,-8){\makebox(0,0)[rc]{\(64\)}}
\put(-7.8,-8){\makebox(0,0)[lc]{\(2^6\)}}
\put(-8.2,-10){\makebox(0,0)[rc]{\(128\)}}
\put(-7.8,-10){\makebox(0,0)[lc]{\(2^7\)}}
\put(-8.2,-12){\makebox(0,0)[rc]{\(256\)}}
\put(-7.8,-12){\makebox(0,0)[lc]{\(2^8\)}}
\put(-8,-12){\vector(0,-1){2}}

\put(-0.1,-0.1){\framebox(0.2,0.2){}}
\put(-2.1,-0.1){\framebox(0.2,0.2){}}
\multiput(0,-2)(0,-2){6}{\circle*{0.2}}
\multiput(-2,-2)(0,-2){6}{\circle*{0.2}}
\multiput(-4,-4)(0,-2){5}{\circle*{0.2}}

\multiput(0,0)(0,-2){6}{\line(0,-1){2}}
\multiput(0,0)(0,-2){6}{\line(-1,-1){2}}
\multiput(0,-2)(0,-2){5}{\line(-2,-1){4}}

\put(0,-12){\vector(0,-1){2}}
\put(-0.2,-13.5){\makebox(0,0)[rc]{infinite}}
\put(-0.2,-14){\makebox(0,0)[rc]{mainline}}
\put(0.2,-14){\makebox(0,0)[lc]{\(\mathcal{T}^1(C_2\times C_2)\)}}

\put(0.2,0){\makebox(0,0)[lb]{\(=C_2\times C_2=V_4\)}}
\put(-2.6,0){\makebox(0,0)[rb]{\(C_4=\)}}
\put(-0.7,0.3){\makebox(0,0)[rb]{abelian}}
\put(-2.6,-2){\makebox(0,0)[rb]{\(Q(8)=\)}}

\put(0.2,-12.2){\makebox(0,0)[lt]{\(D(2^n)\)}}
\put(-2,-12.2){\makebox(0,0)[ct]{\(Q(2^n)\)}}
\put(-4,-12.2){\makebox(0,0)[ct]{\(S(2^n)\)}}
\put(0.2,-12.7){\makebox(0,0)[lt]{dihedral}}
\put(-2,-12.7){\makebox(0,0)[ct]{quaternion}}
\put(-4,-12.7){\makebox(0,0)[ct]{semidihedral}}

\put(3,0){\makebox(0,0)[lb]{\(\Gamma_1\)}}
\put(-2.1,0){\makebox(0,0)[rb]{\(\langle 1\rangle\)}}
\put(-0.1,0){\makebox(0,0)[rb]{\(\langle 2\rangle\)}}

\put(-1.2,-1){\makebox(0,0)[rc]{branch \(\mathcal{B}(2)\)}}
\put(2,-1.7){\makebox(0,0)[lb]{\(\Gamma_2\)}}
\put(-2.1,-2){\makebox(0,0)[rb]{\(\langle 4\rangle\)}}
\put(-0.1,-2){\makebox(0,0)[rb]{\(\langle 3\rangle\)}}

\put(0.5,-3){\vector(0,1){1}}
\put(0.7,-3){\makebox(0,0)[lc]{depth \(1\)}}
\put(0.5,-3){\vector(0,-1){1}}

\put(-3.5,-3){\makebox(0,0)[cc]{\(\mathcal{B}(3)\)}}
\put(2,-3.7){\makebox(0,0)[lb]{\(\Gamma_3\)}}
\put(-4.1,-4){\makebox(0,0)[rb]{\(\langle 8\rangle\)}}
\put(-2.1,-4){\makebox(0,0)[rb]{\(\langle 9\rangle\)}}
\put(-0.1,-4){\makebox(0,0)[rb]{\(\langle 7\rangle\)}}

\put(-4.8,-5){\vector(0,1){1}}
\put(-5,-5){\makebox(0,0)[rc]{period length \(1\)}}
\put(-4.8,-5){\vector(0,-1){1}}

\put(-3.5,-5){\makebox(0,0)[cc]{\(\mathcal{B}(4)\)}}
\put(2,-5.7){\makebox(0,0)[lb]{\(\Gamma_8\)}}
\put(-4.1,-6){\makebox(0,0)[rb]{\(\langle 19\rangle\)}}
\put(-2.1,-6){\makebox(0,0)[rb]{\(\langle 20\rangle\)}}
\put(-0.1,-6){\makebox(0,0)[rb]{\(\langle 18\rangle\)}}

\put(-3.5,-7){\makebox(0,0)[cc]{\(\mathcal{B}(5)\)}}
\put(-4.1,-8){\makebox(0,0)[rb]{\(\langle 53\rangle\)}}
\put(-2.1,-8){\makebox(0,0)[rb]{\(\langle 54\rangle\)}}
\put(-0.1,-8){\makebox(0,0)[rb]{\(\langle 52\rangle\)}}

\put(-3.5,-9){\makebox(0,0)[cc]{\(\mathcal{B}(6)\)}}
\put(-4.1,-10){\makebox(0,0)[rb]{\(\langle 162\rangle\)}}
\put(-2.1,-10){\makebox(0,0)[rb]{\(\langle 163\rangle\)}}
\put(-0.1,-10){\makebox(0,0)[rb]{\(\langle 161\rangle\)}}

\put(-3.5,-11){\makebox(0,0)[cc]{\(\mathcal{B}(7)\)}}
\put(-4.1,-12){\makebox(0,0)[rb]{\(\langle 540\rangle\)}}
\put(-2.1,-12){\makebox(0,0)[rb]{\(\langle 541\rangle\)}}
\put(-0.1,-12){\makebox(0,0)[rb]{\(\langle 539\rangle\)}}

\end{picture}

\end{figure}
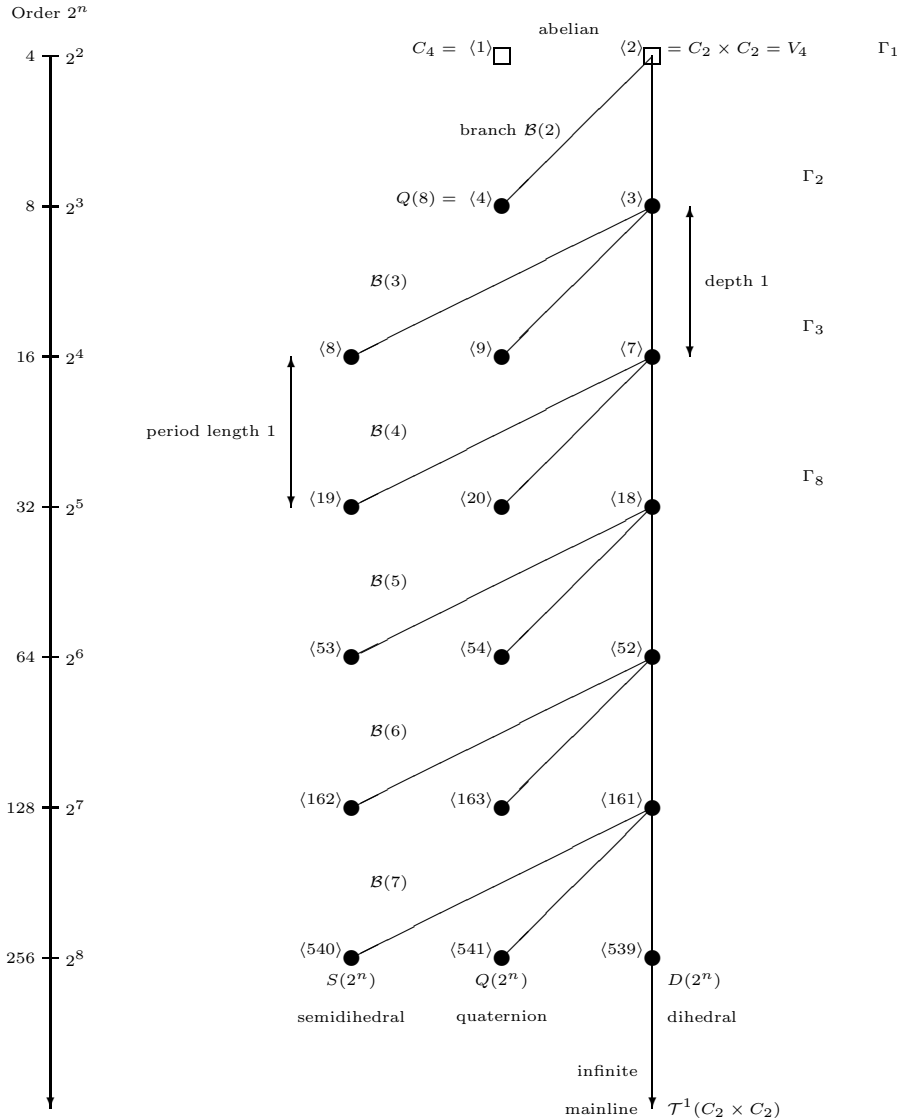

}



The concrete examples \(\mathcal{G}(2,1)\) and \(\mathcal{G}(3,1)\) provide an opportunity
to give a \textit{parametrized polycyclic power-commutator presentation}
\cite[pp.82--84]{Bl}
for the complete coclass tree,
mentioned in \S\
\ref{s:Structure}
as a benefit of the descendant tree concept
and as a consequence of the periodicity of the pruned coclass tree.
In both cases, the group \(G\) is generated by two elements \(x,y\)
but the presentation contains the series of
\textit{higher commutators} \(s_j\), \(2\le j\le n-1=\mathrm{cl}(G)\),
starting with the \textit{main commutator} \(s_2=\lbrack y,x\rbrack\).
The nilpotency is formally expressed by \(s_n=1\),
when the group is of order \(\vert G\vert=p^n\).



{\tiny

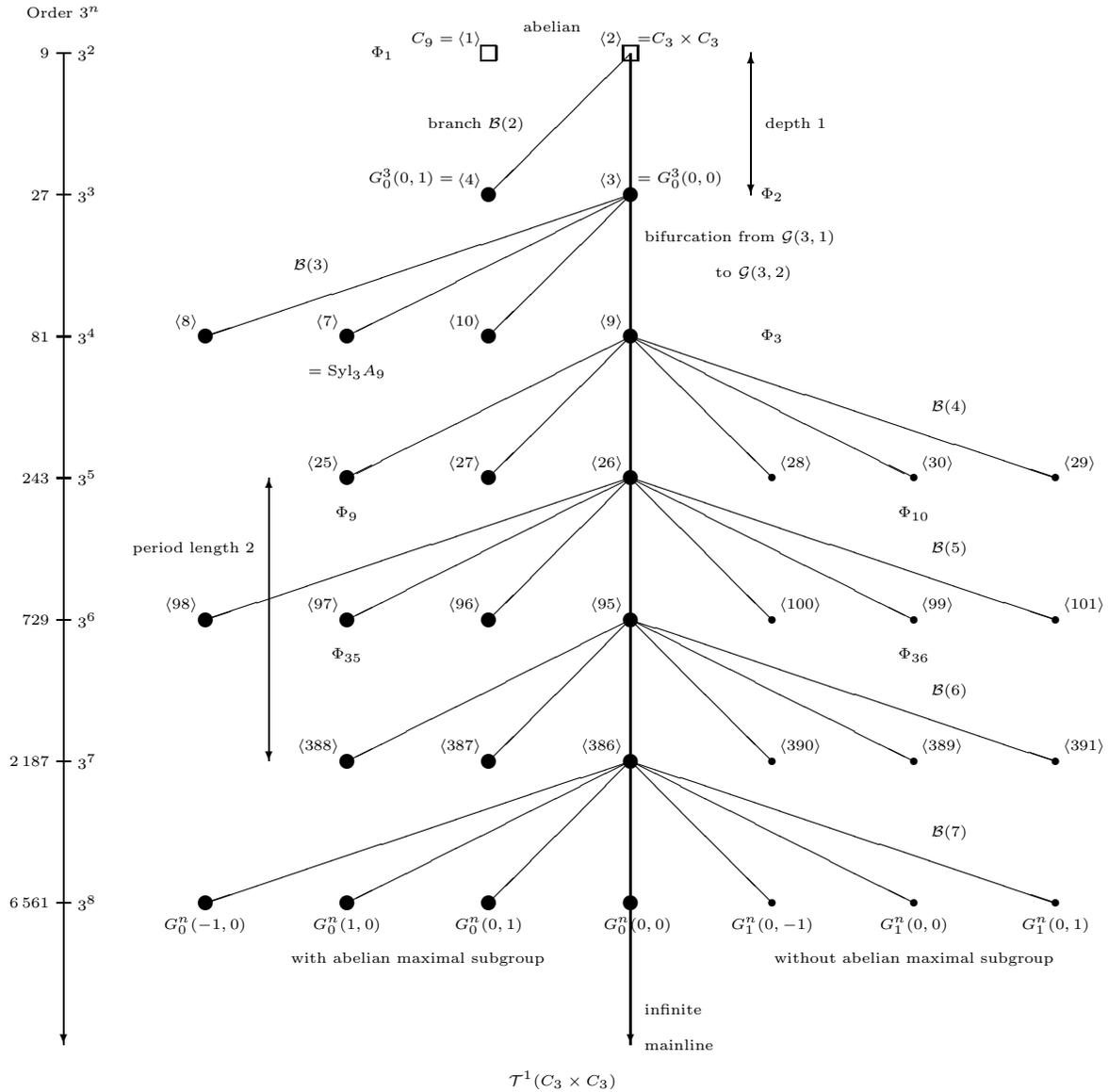
\begin{figure}[hb]
\caption{\(3\)-Groups of Coclass \(1\)}
\label{fig:3GroupsCoclass1}


\setlength{\unitlength}{1cm}
\begin{picture}(16,15)(-8.5,-14)

\put(-8,0.5){\makebox(0,0)[cb]{Order \(3^n\)}}
\put(-8,0){\line(0,-1){12}}
\multiput(-8.1,0)(0,-2){7}{\line(1,0){0.2}}
\put(-8.2,0){\makebox(0,0)[rc]{\(9\)}}
\put(-7.8,0){\makebox(0,0)[lc]{\(3^2\)}}
\put(-8.2,-2){\makebox(0,0)[rc]{\(27\)}}
\put(-7.8,-2){\makebox(0,0)[lc]{\(3^3\)}}
\put(-8.2,-4){\makebox(0,0)[rc]{\(81\)}}
\put(-7.8,-4){\makebox(0,0)[lc]{\(3^4\)}}
\put(-8.2,-6){\makebox(0,0)[rc]{\(243\)}}
\put(-7.8,-6){\makebox(0,0)[lc]{\(3^5\)}}
\put(-8.2,-8){\makebox(0,0)[rc]{\(729\)}}
\put(-7.8,-8){\makebox(0,0)[lc]{\(3^6\)}}
\put(-8.2,-10){\makebox(0,0)[rc]{\(2\,187\)}}
\put(-7.8,-10){\makebox(0,0)[lc]{\(3^7\)}}
\put(-8.2,-12){\makebox(0,0)[rc]{\(6\,561\)}}
\put(-7.8,-12){\makebox(0,0)[lc]{\(3^8\)}}
\put(-8,-12){\vector(0,-1){2}}

\put(-0.1,-0.1){\framebox(0.2,0.2){}}
\put(-2.1,-0.1){\framebox(0.2,0.2){}}
\multiput(0,-2)(0,-2){6}{\circle*{0.2}}
\multiput(-2,-2)(0,-2){6}{\circle*{0.2}}
\multiput(-4,-4)(0,-2){5}{\circle*{0.2}}
\multiput(-6,-4)(0,-4){3}{\circle*{0.2}}
\multiput(2,-6)(0,-2){4}{\circle*{0.1}}
\multiput(4,-6)(0,-2){4}{\circle*{0.1}}
\multiput(6,-6)(0,-2){4}{\circle*{0.1}}

\multiput(0,0)(0,-2){6}{\line(0,-1){2}}
\multiput(0,0)(0,-2){6}{\line(-1,-1){2}}
\multiput(0,-2)(0,-2){5}{\line(-2,-1){4}}
\multiput(0,-2)(0,-4){3}{\line(-3,-1){6}}
\multiput(0,-4)(0,-2){4}{\line(1,-1){2}}
\multiput(0,-4)(0,-2){4}{\line(2,-1){4}}
\multiput(0,-4)(0,-2){4}{\line(3,-1){6}}

\put(0,-12){\vector(0,-1){2}}
\put(0.2,-13.5){\makebox(0,0)[lc]{infinite}}
\put(0.2,-14){\makebox(0,0)[lc]{mainline}}
\put(-0.2,-14.5){\makebox(0,0)[rc]{\(\mathcal{T}^1(C_3\times C_3)\)}}

\put(0.1,0.1){\makebox(0,0)[lb]{=\(C_3\times C_3\)}}
\put(-2.5,0.1){\makebox(0,0)[rb]{\(C_9=\)}}
\put(-0.7,0.3){\makebox(0,0)[rb]{abelian}}
\put(0.1,-1.9){\makebox(0,0)[lb]{\(=G^3_0(0,0)\)}}
\put(-2.5,-1.9){\makebox(0,0)[rb]{\(G^3_0(0,1)=\)}}
\put(-4,-4.5){\makebox(0,0)[cc]{\(=\mathrm{Syl}_3A_9\)}}

\put(0.2,-2.5){\makebox(0,0)[lt]{bifurcation from \(\mathcal{G}(3,1)\)}}
\put(1.2,-3){\makebox(0,0)[lt]{to \(\mathcal{G}(3,2)\)}}

\put(-3.5,0){\makebox(0,0)[cc]{\(\Phi_1\)}}
\put(-2.1,0.1){\makebox(0,0)[rb]{\(\langle 1\rangle\)}}
\put(-0.1,0.1){\makebox(0,0)[rb]{\(\langle 2\rangle\)}}

\put(1.7,-1){\vector(0,1){1}}
\put(1.9,-1){\makebox(0,0)[lc]{depth \(1\)}}
\put(1.7,-1){\vector(0,-1){1}}

\put(-1.5,-1){\makebox(0,0)[rc]{branch \(\mathcal{B}(2)\)}}
\put(2,-2){\makebox(0,0)[cc]{\(\Phi_2\)}}
\put(-2.1,-1.9){\makebox(0,0)[rb]{\(\langle 4\rangle\)}}
\put(-0.1,-1.9){\makebox(0,0)[rb]{\(\langle 3\rangle\)}}

\put(-4.5,-3){\makebox(0,0)[cc]{\(\mathcal{B}(3)\)}}
\put(2,-4){\makebox(0,0)[cc]{\(\Phi_3\)}}
\put(-6.1,-3.9){\makebox(0,0)[rb]{\(\langle 8\rangle\)}}
\put(-4.1,-3.9){\makebox(0,0)[rb]{\(\langle 7\rangle\)}}
\put(-2.1,-3.9){\makebox(0,0)[rb]{\(\langle 10\rangle\)}}
\put(-0.1,-3.9){\makebox(0,0)[rb]{\(\langle 9\rangle\)}}

\put(-4,-6.5){\makebox(0,0)[cc]{\(\Phi_9\)}}
\put(-4.1,-5.9){\makebox(0,0)[rb]{\(\langle 25\rangle\)}}
\put(-2.1,-5.9){\makebox(0,0)[rb]{\(\langle 27\rangle\)}}
\put(-0.1,-5.9){\makebox(0,0)[rb]{\(\langle 26\rangle\)}}

\put(4.5,-5){\makebox(0,0)[cc]{\(\mathcal{B}(4)\)}}
\put(4,-6.5){\makebox(0,0)[cc]{\(\Phi_{10}\)}}
\put(2.1,-5.9){\makebox(0,0)[lb]{\(\langle 28\rangle\)}}
\put(4.1,-5.9){\makebox(0,0)[lb]{\(\langle 30\rangle\)}}
\put(6.1,-5.9){\makebox(0,0)[lb]{\(\langle 29\rangle\)}}

\put(-4,-8.5){\makebox(0,0)[cc]{\(\Phi_{35}\)}}
\put(-6.1,-7.9){\makebox(0,0)[rb]{\(\langle 98\rangle\)}}
\put(-4.1,-7.9){\makebox(0,0)[rb]{\(\langle 97\rangle\)}}
\put(-2.1,-7.9){\makebox(0,0)[rb]{\(\langle 96\rangle\)}}
\put(-0.1,-7.9){\makebox(0,0)[rb]{\(\langle 95\rangle\)}}

\put(4.5,-7){\makebox(0,0)[cc]{\(\mathcal{B}(5)\)}}
\put(4,-8.5){\makebox(0,0)[cc]{\(\Phi_{36}\)}}
\put(2.1,-7.9){\makebox(0,0)[lb]{\(\langle 100\rangle\)}}
\put(4.1,-7.9){\makebox(0,0)[lb]{\(\langle 99\rangle\)}}
\put(6.1,-7.9){\makebox(0,0)[lb]{\(\langle 101\rangle\)}}

\put(-5.1,-8){\vector(0,1){2}}
\put(-5.3,-7){\makebox(0,0)[rc]{period length \(2\)}}
\put(-5.1,-8){\vector(0,-1){2}}

\put(-4.1,-9.9){\makebox(0,0)[rb]{\(\langle 388\rangle\)}}
\put(-2.1,-9.9){\makebox(0,0)[rb]{\(\langle 387\rangle\)}}
\put(-0.1,-9.9){\makebox(0,0)[rb]{\(\langle 386\rangle\)}}

\put(4.5,-9){\makebox(0,0)[cc]{\(\mathcal{B}(6)\)}}
\put(2.1,-9.9){\makebox(0,0)[lb]{\(\langle 390\rangle\)}}
\put(4.1,-9.9){\makebox(0,0)[lb]{\(\langle 389\rangle\)}}
\put(6.1,-9.9){\makebox(0,0)[lb]{\(\langle 391\rangle\)}}

\put(4.5,-11){\makebox(0,0)[cc]{\(\mathcal{B}(7)\)}}

\put(0.1,-12.2){\makebox(0,0)[ct]{\(G^n_0(0,0)\)}}
\put(-2,-12.2){\makebox(0,0)[ct]{\(G^n_0(0,1)\)}}
\put(-4,-12.2){\makebox(0,0)[ct]{\(G^n_0(1,0)\)}}
\put(-6,-12.2){\makebox(0,0)[ct]{\(G^n_0(-1,0)\)}}
\put(2,-12.2){\makebox(0,0)[ct]{\(G^n_1(0,-1)\)}}
\put(4,-12.2){\makebox(0,0)[ct]{\(G^n_1(0,0)\)}}
\put(6,-12.2){\makebox(0,0)[ct]{\(G^n_1(0,1)\)}}

\put(-3,-12.7){\makebox(0,0)[ct]{with abelian maximal subgroup}}
\put(4,-12.7){\makebox(0,0)[ct]{without abelian maximal subgroup}}

\end{picture}

\end{figure}

}



For \(p=2\), there are two parameters \(0\le w,z\le 1\) and the pc-presentation
is given by

\begin{equation}
\label{eqn:ParamPres2Cocl1}
\begin{aligned}
G^n(z,w)= & \langle x,y,s_2,\ldots,s_{n-1}\mid\\
& x^2=s_{n-1}^w,\ y^2=s_2^{-1}s_{n-1}^z,\ \lbrack s_2,y\rbrack=1,\\
& s_2=\lbrack y,x\rbrack,\ s_j=\lbrack s_{j-1},x\rbrack\text{ for }3\le j\le n-1\rangle
\end{aligned}
\end{equation}

The 2-groups of maximal class, that is of coclass \(1\),
form three \textit{periodic infinite sequences},

\begin{itemize}
\item
the \textit{dihedral} groups, \(D(2^n)=G^n(0,0)\), \(n\ge 3\),
forming the mainline (with infinitely capable vertices),
\item
the \textit{generalized quaternion} groups, \(Q(2^n)=G^n(0,1)\), \(n\ge 3\),
which are all terminal vertices,
\item
the \textit{semidihedral} groups, \(S(2^n)=G^n(1,0)\), \(n\ge 4\),
which are also leaves.
\end{itemize}

For \(p=3\), there are three parameters \(0\le a\le 1\) and \(-1\le w,z\le 1\)
and the pc-presentation is given by

\begin{equation}
\label{eqn:ParamPres3Cocl1}
\begin{aligned}
G^n_a(z,w)= & \langle x,y,s_2,\ldots,s_{n-1}\mid\\
& x^3=s_{n-1}^w,\ y^3=s_2^{-3}s_3^{-1}s_{n-1}^z,\ \lbrack y,s_2\rbrack=s_{n-1}^a,\\
& s_2=\lbrack y,x\rbrack,\ s_j=\lbrack s_{j-1},x\rbrack\text{ for }3\le j\le n-1\rangle
\end{aligned}
\end{equation}

\(3\)-groups with parameter \(a=0\) possess an abelian maximal subgroup,
those with parameter \(a=1\) do not.
More precisely, an existing abelian maximal subgroup is unique,
except for the two \textit{extra special} groups \(G^3_0(0,0)\) and \(G^3_0(0,1)\),
where all four maximal subgroups are abelian.

In contrast to any bigger coclass \(r\ge 2\),
the coclass graph \(\mathcal{G}(p,1)\) exclusively contains
\(p\)-groups \(G\) with abelianization \(G/G^\prime\) of type \((p,p)\),
except for its unique isolated vertex.
The case \(p=2\) is distinguished by the truth of the reverse statement:
Any \(2\)-group with abelianization of type \((2,2)\) is of coclass \(1\)
(O. Taussky's Theorem
\cite[p.83]{Ta}).



{\tiny

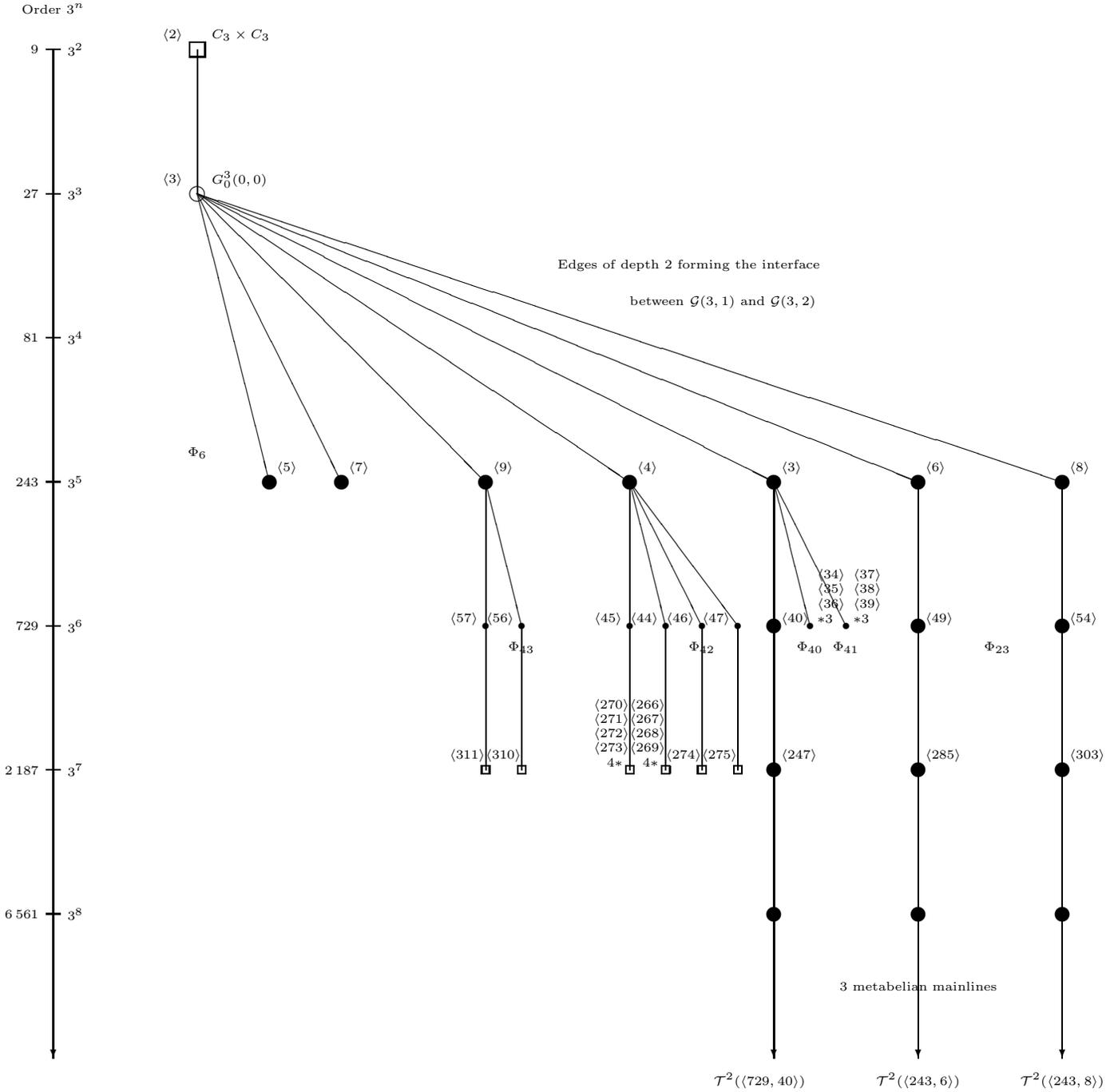
\begin{figure}[ht]
\caption{\(3\)-Groups of Coclass \(2\) with Abelianization \((3,3)\)}
\label{fig:3GrpTyp33Cocl2}


\setlength{\unitlength}{1.2cm}
\begin{picture}(16,15)(1,-12)

\put(0,2.5){\makebox(0,0)[cb]{Order \(3^n\)}}
\put(0,2){\line(0,-1){12}}
\multiput(-0.1,2)(0,-2){7}{\line(1,0){0.2}}
\put(-0.2,2){\makebox(0,0)[rc]{\(9\)}}
\put(0.2,2){\makebox(0,0)[lc]{\(3^2\)}}
\put(-0.2,0){\makebox(0,0)[rc]{\(27\)}}
\put(0.2,0){\makebox(0,0)[lc]{\(3^3\)}}
\put(-0.2,-2){\makebox(0,0)[rc]{\(81\)}}
\put(0.2,-2){\makebox(0,0)[lc]{\(3^4\)}}
\put(-0.2,-4){\makebox(0,0)[rc]{\(243\)}}
\put(0.2,-4){\makebox(0,0)[lc]{\(3^5\)}}
\put(-0.2,-6){\makebox(0,0)[rc]{\(729\)}}
\put(0.2,-6){\makebox(0,0)[lc]{\(3^6\)}}
\put(-0.2,-8){\makebox(0,0)[rc]{\(2\,187\)}}
\put(0.2,-8){\makebox(0,0)[lc]{\(3^7\)}}
\put(-0.2,-10){\makebox(0,0)[rc]{\(6\,561\)}}
\put(0.2,-10){\makebox(0,0)[lc]{\(3^8\)}}
\put(0,-10){\vector(0,-1){2}}

\put(2.2,2.2){\makebox(0,0)[lc]{\(C_3\times C_3\)}}
\put(1.8,2.2){\makebox(0,0)[rc]{\(\langle 2\rangle\)}}
\put(1.9,1.9){\framebox(0.2,0.2){}}
\put(2,2){\line(0,-1){2}}
\put(2,0){\circle{0.2}}
\put(2.2,0.2){\makebox(0,0)[lc]{\(G^3_0(0,0)\)}}
\put(1.8,0.2){\makebox(0,0)[rc]{\(\langle 3\rangle\)}}

\put(2,0){\line(1,-4){1}}
\put(2,0){\line(1,-2){2}}
\put(2,0){\line(1,-1){4}}
\put(2,0){\line(3,-2){6}}
\put(2,0){\line(2,-1){8}}
\put(2,0){\line(5,-2){10}}
\put(2,0){\line(3,-1){12}}
\put(7,-1){\makebox(0,0)[lc]{Edges of depth \(2\) forming the interface}}
\put(8,-1.5){\makebox(0,0)[lc]{between \(\mathcal{G}(3,1)\) and \(\mathcal{G}(3,2)\)}}

\put(2,-3.6){\makebox(0,0)[cc]{\(\Phi_6\)}}
\multiput(3,-4)(1,0){2}{\circle*{0.2}}
\put(3.1,-3.9){\makebox(0,0)[lb]{\(\langle 5\rangle\)}}
\put(4.1,-3.9){\makebox(0,0)[lb]{\(\langle 7\rangle\)}}
\multiput(6,-4)(2,0){5}{\circle*{0.2}}
\put(6.1,-3.9){\makebox(0,0)[lb]{\(\langle 9\rangle\)}}
\put(8.1,-3.9){\makebox(0,0)[lb]{\(\langle 4\rangle\)}}
\put(10.1,-3.9){\makebox(0,0)[lb]{\(\langle 3\rangle\)}}
\put(12.1,-3.9){\makebox(0,0)[lb]{\(\langle 6\rangle\)}}
\put(14.1,-3.9){\makebox(0,0)[lb]{\(\langle 8\rangle\)}}



\put(6.5,-6.3){\makebox(0,0)[cc]{\(\Phi_{43}\)}}
\put(6,-4){\line(0,-1){2}}
\put(5.9,-5.9){\makebox(0,0)[rc]{\(\langle 57\rangle\)}}
\put(6.4,-5.9){\makebox(0,0)[rc]{\(\langle 56\rangle\)}}
\put(6,-7.8){\makebox(0,0)[rc]{\(\langle 311\rangle\)}}
\put(6.5,-7.8){\makebox(0,0)[rc]{\(\langle 310\rangle\)}}
\put(6,-4){\line(1,-4){0.5}}
\multiput(6,-6)(0.5,0){2}{\circle*{0.1}}
\multiput(6,-6)(0.5,0){2}{\line(0,-1){2}}
\multiput(5.95,-8.05)(0.5,0){2}{\framebox(0.1,0.1){}}

\put(9,-6.3){\makebox(0,0)[cc]{\(\Phi_{42}\)}}
\put(8,-4){\line(0,-1){2}}
\put(7.9,-5.9){\makebox(0,0)[rc]{\(\langle 45\rangle\)}}
\put(8.4,-5.9){\makebox(0,0)[rc]{\(\langle 44\rangle\)}}
\put(8.9,-5.9){\makebox(0,0)[rc]{\(\langle 46\rangle\)}}
\put(9.4,-5.9){\makebox(0,0)[rc]{\(\langle 47\rangle\)}}
\put(8,-7.1){\makebox(0,0)[rc]{\(\langle 270\rangle\)}}
\put(8,-7.3){\makebox(0,0)[rc]{\(\langle 271\rangle\)}}
\put(8,-7.5){\makebox(0,0)[rc]{\(\langle 272\rangle\)}}
\put(8,-7.7){\makebox(0,0)[rc]{\(\langle 273\rangle\)}}
\put(8.5,-7.1){\makebox(0,0)[rc]{\(\langle 266\rangle\)}}
\put(8.5,-7.3){\makebox(0,0)[rc]{\(\langle 267\rangle\)}}
\put(8.5,-7.5){\makebox(0,0)[rc]{\(\langle 268\rangle\)}}
\put(8.5,-7.7){\makebox(0,0)[rc]{\(\langle 269\rangle\)}}
\put(9,-7.8){\makebox(0,0)[rc]{\(\langle 274\rangle\)}}
\put(9.5,-7.8){\makebox(0,0)[rc]{\(\langle 275\rangle\)}}
\put(8,-4){\line(1,-4){0.5}}
\put(8,-4){\line(1,-2){1}}
\put(8,-4){\line(3,-4){1.5}}
\multiput(8,-6)(0.5,0){4}{\circle*{0.1}}
\multiput(8,-6)(0.5,0){4}{\line(0,-1){2}}
\multiput(7.95,-8.05)(0.5,0){4}{\framebox(0.1,0.1){}}
\multiput(7.9,-7.9)(0.5,0){2}{\makebox(0,0)[rc]{\(4*\)}}




\put(10.5,-6.3){\makebox(0,0)[cc]{\(\Phi_{40}\)}}
\put(11,-6.3){\makebox(0,0)[cc]{\(\Phi_{41}\)}}
\multiput(10,-4)(0,-2){3}{\line(0,-1){2}}
\multiput(10,-6)(0,-2){3}{\circle*{0.2}}
\put(10.1,-5.9){\makebox(0,0)[lc]{\(\langle 40\rangle\)}}
\put(10.6,-5.3){\makebox(0,0)[lc]{\(\langle 34\rangle\)}}
\put(10.6,-5.5){\makebox(0,0)[lc]{\(\langle 35\rangle\)}}
\put(10.6,-5.7){\makebox(0,0)[lc]{\(\langle 36\rangle\)}}
\put(11.1,-5.3){\makebox(0,0)[lc]{\(\langle 37\rangle\)}}
\put(11.1,-5.5){\makebox(0,0)[lc]{\(\langle 38\rangle\)}}
\put(11.1,-5.7){\makebox(0,0)[lc]{\(\langle 39\rangle\)}}
\put(10.1,-7.8){\makebox(0,0)[lc]{\(\langle 247\rangle\)}}
\put(10,-4){\line(1,-4){0.5}}
\put(10,-4){\line(1,-2){1}}
\multiput(10.5,-6)(0.5,0){2}{\circle*{0.1}}
\multiput(10.6,-5.9)(0.5,0){2}{\makebox(0,0)[lc]{\(*3\)}}
\put(10,-10){\vector(0,-1){2}}
\put(9.8,-12.2){\makebox(0,0)[ct]{\(\mathcal{T}^2(\langle 729,40\rangle)\)}}


\multiput(12,-4)(0,-2){3}{\line(0,-1){2}}
\multiput(12,-6)(0,-2){3}{\circle*{0.2}}
\put(12.1,-5.9){\makebox(0,0)[lc]{\(\langle 49\rangle\)}}
\put(12.1,-7.8){\makebox(0,0)[lc]{\(\langle 285\rangle\)}}
\put(12,-10){\vector(0,-1){2}}
\put(12,-12.2){\makebox(0,0)[ct]{\(\mathcal{T}^2(\langle 243,6\rangle)\)}}

\put(13.1,-6.3){\makebox(0,0)[cc]{\(\Phi_{23}\)}}
\put(12,-11){\makebox(0,0)[cc]{\(3\) metabelian mainlines}}

\multiput(14,-4)(0,-2){3}{\line(0,-1){2}}
\multiput(14,-6)(0,-2){3}{\circle*{0.2}}
\put(14.1,-5.9){\makebox(0,0)[lc]{\(\langle 54\rangle\)}}
\put(14.1,-7.8){\makebox(0,0)[lc]{\(\langle 303\rangle\)}}
\put(14,-10){\vector(0,-1){2}}
\put(14,-12.2){\makebox(0,0)[ct]{\(\mathcal{T}^2(\langle 243,8\rangle)\)}}


\end{picture}

\end{figure}

}



Figure
\ref{fig:3GrpTyp33Cocl2}
shows the interface between finite \(3\)-groups of coclass \(1\) and \(2\) of type \((3,3)\).

\subsection{Coclass \(2\)}
\label{ss:CoclassTwo}

The genesis of the coclass graph \(\mathcal{G}(p,r)\) with \(r\ge 2\) is not uniform.
\(p\)-groups with several distinct abelianizations contribute to its constitution.
For coclass \(r=2\),
there are essential contributions from groups \(G\) with abelianizations
\(G/G^\prime\) of the types \((p,p)\), \((p^2,p)\), \((p,p,p)\),
and an isolated contribution by the cyclic group of order \(p^3\):

\begin{equation}
\label{eqn:Coclass2}
\mathcal{G}(p,2)=\mathcal{G}_{(p,p)}(p,2)\dot{\cup}\mathcal{G}_{(p^2,p)}(p,2)\dot{\cup}\mathcal{G}_{(p,p,p)}(p,2)\dot{\cup}\mathcal{G}_{(p^3)}(p,2).
\end{equation}



\subsubsection{Abelianization of type \((p,p)\)}
\label{sss:TypePP}

As opposed to \(p\)-groups of coclass \(2\) with abelianization of type \((p^2,p)\) or \((p,p,p)\),
which arise as regular descendants of abelian \(p\)-groups of the same types,
\(p\)-groups of coclass \(2\) with abelianization of type \((p,p)\)
arise from irregular descendants of a non-abelian \(p\)-group with coclass \(1\) and nuclear rank \(2\).

For the prime \(p=2\), such groups do not exist at all,
since the dihedral group \(\langle 8,3\rangle\) is coclass-settled,
which is the deeper reason for Taussky's Theorem.
This remarkable fact has been observed by G. Bagnera
\cite[Part 2, \S\ 4, p.182]{Bg}
in 1898 already.

For odd primes \(p\ge 3\),
the existence of \(p\)-groups of coclass \(2\) with abelianization of type \((p,p)\)
is due to the fact that the extra special group \(G^3_0(0,0)\) is not coclass-settled.
Its nuclear rank equals \(2\),
which gives rise to a \textit{bifurcation} of the descendant tree \(\mathcal{T}(G^3_0(0,0))\)
into two coclass graphs.
The regular component \(\mathcal{T}^1(G^3_0(0,0))\)
is a subtree of the unique tree \(\mathcal{T}^1(C_p\times C_p)\) in the coclass graph \(\mathcal{G}(p,1)\).
The irregular component \(\mathcal{T}^2(G^3_0(0,0))\)
becomes a subgraph \(\mathcal{G}=\mathcal{G}_{(p,p)}(p,2)\) of the coclass graph \(\mathcal{G}(p,2)\)
when the connecting edges of depth \(2\) of the irregular immediate descendants of \(G^3_0(0,0)\) are removed.

For \(p=3\), this subgraph \(\mathcal{G}\) is drawn in Figure 4.
It has seven top level vertices of three important kinds, all having order \(243=3^5\),
which have been discovered by G. Bagnera
\cite[Part 2, \S\ 4, pp.182--183]{Bg}.

\begin{itemize}
\item
Firstly, there are two terminal \textit{Schur \(\sigma\)-groups}
\cite{Ag}
\(\langle 243,5\rangle\) and \(\langle 243,7\rangle\)
in the sporadic part \(\mathcal{G}_0(3,2)\) of the coclass graph \(\mathcal{G}(3,2)\).
\item
Secondly, the two groups
\(G=\langle 243,4\rangle\) and \(G=\langle 243,9\rangle\)
are roots of finite trees \(\mathcal{T}^2(G)\) in the sporadic part \(\mathcal{G}_0(3,2)\).
(However, since they are not coclass-settled, the complete trees \(\mathcal{T}(G)\) are infinite.)
\item
And, finally, the three groups
\(\langle 243,3\rangle\), \(\langle 243,6\rangle\) and \(\langle 243,8\rangle\)
give rise to (infinite) coclass trees, e.g.,
\(\mathcal{T}^2(\langle 729,40\rangle)\),
\(\mathcal{T}^2(\langle 243,6\rangle)\),
\(\mathcal{T}^2(\langle 243,8\rangle)\),
each having a metabelian mainline,
in the coclass graph \(\mathcal{G}(3,2)\).
None of these three groups is coclass-settled. See \S\
\ref{s:PeriodicBifurcations}.
\end{itemize}

Displaying additional information on kernels and targets of Artin transfers
\cite{Ar2},
we can draw these trees as \textit{structured descendant trees}
\cite[Fig.3.5, p.439, Fig.3.6, p.442, and Fig.3.7, p.443]{Ma4}.



\begin{definition}
\label{dfn:SchurSigmaGroup}
Generally, a \textit{Schur group}
(called a \textit{closed} group by I. Schur, who coined the concept)
is a pro-\(p\) group \(G\) whose relation rank
\(r(G)=\mathrm{dim}_{\mathbb{F}_p}(\mathrm{H}^2(G,\mathbb{F}_p))\)
coincides with its generator rank
\(d(G)=\mathrm{dim}_{\mathbb{F}_p}(\mathrm{H}^1(G,\mathbb{F}_p))\).
A \textit{\(\sigma\)-group} is a pro-\(p\) group \(G\)
which possesses an automorphism \(\sigma\in\mathrm{Aut}(G)\)
inducing the inversion \(x\mapsto x^{-1}\) on its abelianization \(G/G^\prime\).
A \textit{Schur \(\sigma\)-group}
\cite{Ag,BBH,BuMa,KoVe}
is a Schur group \(G\) which is also a \(\sigma\)-group and has a finite abelianization \(G/G^\prime\).
\end{definition}

It should be pointed out that \(\langle 243,3\rangle\) is not root of a coclass tree,
since its immediate descendant \(\langle 729,40\rangle=B\),
which is root of a coclass tree with metabelian mainline vertices,
has two siblings \(\langle 729,35\rangle=I\), resp. \(\langle 729,34\rangle=H\),
which give rise to a single, resp. three, coclass tree(s)
with non-metabelian mainline vertices having cyclic centres of order \(3\)
and branches of considerable complexity but nevertheless of bounded depth \(5\).



\renewcommand{\arraystretch}{1.0}

\begin{table}[ht]
\caption{Quotients of the Groups \(G=G(f,g,h)\)}
\label{tbl:QuotientsBQUGA}
\begin{center}
\begin{tabular}{|c||c|c|c|c|}
\hline
 parameters  & abelianization & class-\(2\) quotient     & class-\(3\) quotient      & class-\(4\) quotient       \\
 \((f,g,h)\) & \(G/G^\prime\) & \(G/\gamma_3(G)\)        & \(G/\gamma_4(G)\)         &    \(G/\gamma_5(G)\)       \\
\hline
 \((0,1,0)\) & \((3,3)\)      & \(\langle 27,3\rangle\)  & \(\langle 243,3\rangle\)  & \(\langle 729,40\rangle\)  \\
 \((0,1,2)\) & \((3,3)\)      & \(\langle 27,3\rangle\)  & \(\langle 243,6\rangle\)  & \(\langle 729,49\rangle\)  \\
 \((1,1,2)\) & \((3,3)\)      & \(\langle 27,3\rangle\)  & \(\langle 243,8\rangle\)  & \(\langle 729,54\rangle\)  \\
\hline
 \((1,0,0)\) & \((9,3)\)      & \(\langle 81,3\rangle\)  & \(\langle 243,15\rangle\) & \(\langle 729,79\rangle\)  \\
 \((0,0,1)\) & \((9,3)\)      & \(\langle 81,3\rangle\)  & \(\langle 243,17\rangle\) & \(\langle 729,84\rangle\)  \\
\hline
 \((0,0,0)\) & \((3,3,3)\)    & \(\langle 81,12\rangle\) & \(\langle 243,53\rangle\) & \(\langle 729,395\rangle\) \\
\hline
\end{tabular}
\end{center}
\end{table}



\subsubsection{Pro-\(3\) groups of coclass \(2\) with non-trivial centre}
\label{sss:Pro3Cocl2}

B. Eick, C.R. Leedham-Green, M.F. Newman and E.A. O'Brien
\cite[\S\ 4, Thm.4.1]{ELNO}
have constructed a family of infinite pro-3 groups with coclass \(2\)
having a non-trivial centre of order \(3\).
The members are characterized by three parameters \((f,g,h)\).
Their finite quotients generate all mainline vertices with bicyclic centres of type \((3,3)\)
of six coclass trees in the coclass graph \(\mathcal{G}(3,2)\).
The association of parameters to the roots of these six trees is given in Table
\ref{tbl:QuotientsBQUGA},
the tree diagrams are indicated in Figures
\ref{fig:3GrpTyp33Cocl2}
and
\ref{fig:3GrpTyp93Cocl2}, 
 and
the parametrized pro-\(3\) presentation is given by

\begin{equation}
\label{eqn:ParamPres3Cocl2}
\begin{aligned}G(f,g,h)= & \langle a,t,z\mid\\
& a^3=z^f,\ \lbrack t,t^a\rbrack=z^g,\ t^{1+a+a^2}=z^h,\\
& z^3=1,\ \lbrack z,a\rbrack=1,\ \lbrack z,t\rbrack=1\rangle
\end{aligned}
\end{equation}



{\tiny

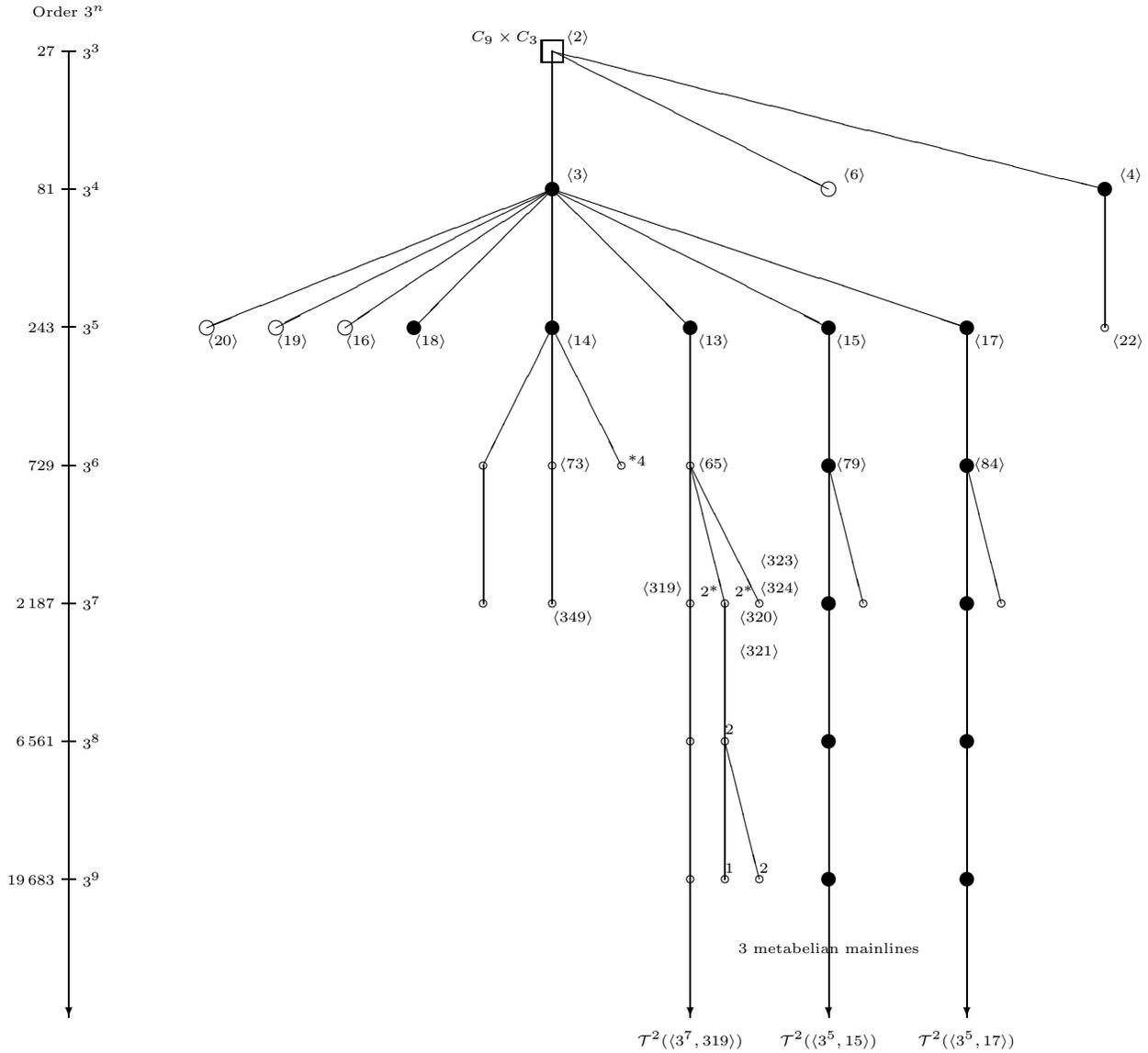
\begin{figure}[ht]
\caption{\(3\)-Groups of Coclass \(2\) with Abelianization \((9,3)\)}
\label{fig:3GrpTyp93Cocl2}


\setlength{\unitlength}{1cm}
\begin{picture}(17,17)(1,-16)

\put(1,0.5){\makebox(0,0)[cb]{Order \(3^n\)}}
\put(1,0){\line(0,-1){12}}
\multiput(0.9,0)(0,-2){7}{\line(1,0){0.2}}
\put(0.8,0){\makebox(0,0)[rc]{\(27\)}}
\put(1.2,0){\makebox(0,0)[lc]{\(3^3\)}}
\put(0.8,-2){\makebox(0,0)[rc]{\(81\)}}
\put(1.2,-2){\makebox(0,0)[lc]{\(3^4\)}}
\put(0.8,-4){\makebox(0,0)[rc]{\(243\)}}
\put(1.2,-4){\makebox(0,0)[lc]{\(3^5\)}}
\put(0.8,-6){\makebox(0,0)[rc]{\(729\)}}
\put(1.2,-6){\makebox(0,0)[lc]{\(3^6\)}}
\put(0.8,-8){\makebox(0,0)[rc]{\(2\,187\)}}
\put(1.2,-8){\makebox(0,0)[lc]{\(3^7\)}}
\put(0.8,-10){\makebox(0,0)[rc]{\(6\,561\)}}
\put(1.2,-10){\makebox(0,0)[lc]{\(3^8\)}}
\put(0.8,-12){\makebox(0,0)[rc]{\(19\,683\)}}
\put(1.2,-12){\makebox(0,0)[lc]{\(3^9\)}}
\put(1,-12){\vector(0,-1){2}}

\put(7.8,0.2){\makebox(0,0)[rc]{\(C_9\times C_3\)}}
\put(8.2,0.2){\makebox(0,0)[lc]{\(\langle 2\rangle\)}}
\put(7.85,-0.15){\framebox(0.3,0.3){}}
\put(8,0){\line(0,-1){2}}
\put(8,-2){\circle*{0.2}}
\put(8.2,-1.8){\makebox(0,0)[lc]{\(\langle 3\rangle\)}}
\put(8,0){\line(2,-1){4}}
\put(12,-2){\circle{0.2}}
\put(12.2,-1.8){\makebox(0,0)[lc]{\(\langle 6\rangle\)}}
\put(8,0){\line(4,-1){8}}
\put(16,-2){\circle*{0.2}}
\put(16.2,-1.8){\makebox(0,0)[lc]{\(\langle 4\rangle\)}}
\put(16,-2){\line(0,-1){2}}
\put(16,-4){\circle{0.1}}
\put(16.1,-4.1){\makebox(0,0)[lt]{\(\langle 22\rangle\)}}

\put(8,-2){\line(-5,-2){5}}
\put(8,-2){\line(-2,-1){4}}
\put(8,-2){\line(-3,-2){3}}
\put(8,-2){\line(-1,-1){2}}
\put(8,-2){\line(0,-1){2}}
\put(8,-2){\line(1,-1){2}}
\put(8,-2){\line(2,-1){4}}
\put(8,-2){\line(3,-1){6}}

\multiput(3,-4)(1,0){3}{\circle{0.2}}
\put(3,-4.1){\makebox(0,0)[lt]{\(\langle 20\rangle\)}}
\put(4,-4.1){\makebox(0,0)[lt]{\(\langle 19\rangle\)}}
\put(5,-4.1){\makebox(0,0)[lt]{\(\langle 16\rangle\)}}
\multiput(6,-4)(2,0){5}{\circle*{0.2}}
\put(6,-4.1){\makebox(0,0)[lt]{\(\langle 18\rangle\)}}
\put(8.2,-4.1){\makebox(0,0)[lt]{\(\langle 14\rangle\)}}
\put(10.1,-4.1){\makebox(0,0)[lt]{\(\langle 13\rangle\)}}
\put(12.1,-4.1){\makebox(0,0)[lt]{\(\langle 15\rangle\)}}
\put(14.1,-4.1){\makebox(0,0)[lt]{\(\langle 17\rangle\)}}

\put(8,-4){\line(-1,-2){1}}
\put(8,-4){\line(0,-1){2}}
\put(8,-4){\line(1,-2){1}}
\multiput(7,-6)(1,0){3}{\circle{0.1}}
\put(8.1,-6.1){\makebox(0,0)[lb]{\(\langle 73\rangle\)}}
\put(9.1,-6){\makebox(0,0)[lb]{*4}}
\multiput(7,-6)(1,0){2}{\line(0,-1){2}}
\multiput(7,-8)(1,0){2}{\circle{0.1}}
\put(8,-8.1){\makebox(0,0)[lt]{\(\langle 349\rangle\)}}


\multiput(10,-4)(0,-2){4}{\line(0,-1){2}}
\multiput(10,-6)(0,-2){4}{\circle{0.1}}
\put(10,-6){\line(1,-4){0.5}}
\put(10,-6){\line(1,-2){1}}
\multiput(10.5,-8)(0,-2){2}{\line(0,-1){2}}
\put(10.5,-10){\line(1,-4){0.5}}
\multiput(10.5,-8)(0,-2){3}{\circle{0.1}}
\multiput(11,-8)(0,-4){2}{\circle{0.1}}
\put(10.1,-6.1){\makebox(0,0)[lb]{\(\langle 65\rangle\)}}
\put(9.9,-7.9){\makebox(0,0)[rb]{\(\langle 319\rangle\)}}
\put(11.3,-8.1){\makebox(0,0)[rt]{\(\langle 320\rangle\)}}
\put(11.3,-8.6){\makebox(0,0)[rt]{\(\langle 321\rangle\)}}
\put(11,-7.9){\makebox(0,0)[lb]{\(\langle 324\rangle\)}}
\put(11,-7.5){\makebox(0,0)[lb]{\(\langle 323\rangle\)}}
\put(10.4,-7.9){\makebox(0,0)[rb]{2*}}
\put(10.9,-7.9){\makebox(0,0)[rb]{2*}}
\put(10.5,-9.9){\makebox(0,0)[lb]{\(2\)}}
\put(10.5,-11.9){\makebox(0,0)[lb]{\(1\)}}
\put(11,-11.9){\makebox(0,0)[lb]{\(2\)}}
\put(10,-12){\vector(0,-1){2}}
\put(10,-14.2){\makebox(0,0)[ct]{\(\mathcal{T}^2(\langle 3^7,319\rangle)\)}}

\multiput(12,-4)(0,-2){4}{\line(0,-1){2}}
\multiput(12,-6)(0,-2){4}{\circle*{0.2}}
\put(12.1,-6.1){\makebox(0,0)[lb]{\(\langle 79\rangle\)}}
\put(12,-6){\line(1,-4){0.5}}
\put(12.5,-8){\circle{0.1}}
\put(12,-12){\vector(0,-1){2}}
\put(12,-14.2){\makebox(0,0)[ct]{\(\mathcal{T}^2(\langle 3^5,15\rangle)\)}}

\put(12,-13){\makebox(0,0)[cc]{\(3\) metabelian mainlines}}

\multiput(14,-4)(0,-2){4}{\line(0,-1){2}}
\multiput(14,-6)(0,-2){4}{\circle*{0.2}}
\put(14.1,-6.1){\makebox(0,0)[lb]{\(\langle 84\rangle\)}}
\put(14,-6){\line(1,-4){0.5}}
\put(14.5,-8){\circle{0.1}}
\put(14,-12){\vector(0,-1){2}}
\put(14,-14.2){\makebox(0,0)[ct]{\(\mathcal{T}^2(\langle 3^5,17\rangle)\)}}

\end{picture}

\end{figure}

}

Figure
\ref{fig:3GrpTyp93Cocl2}
shows some finite \(3\)-groups with coclass \(2\) and type \((9,3)\).



\subsubsection{Abelianization of type \((p^2,p)\)}
\label{sss:TypeP2P}

For \(p=3\), the top levels of the subtree \(\mathcal{T}^2(\langle 27,2\rangle)\)
of the coclass graph \(\mathcal{G}(3,2)\) are drawn in Figure
\ref{fig:3GrpTyp93Cocl2}.
The most important vertices of this tree are
the eight siblings sharing the common parent \(\langle 81,3\rangle\),
which are of three important kinds.

\begin{itemize}
\item
Firstly, there are three leaves
\(\langle 243,20\rangle\),
\(\langle 243,19\rangle\),
\(\langle 243,16\rangle\)
having cyclic centre of order \(9\),
and a single leaf \(\langle 243,18\rangle\) with bicyclic centre of type \((3,3)\). 
\item
Secondly, the group \(G=\langle 243,14\rangle\)
is root of a finite tree \(\mathcal{T}(G)=\mathcal{T}^2(G)\).
\item
And, finally, the three groups
\(\langle 243,13\rangle\),
\(\langle 243,15\rangle\) and
\(\langle 243,17\rangle\)
give rise to infinite coclass trees, e. g.,
\(\mathcal{T}^2(\langle 2187,319\rangle)\),
\(\mathcal{T}^2(\langle 243,15\rangle)\),
\(\mathcal{T}^2(\langle 243,17\rangle)\),
each having a metabelian mainline,
the first with cyclic centres of order \(3\),
the second and third with bicyclic centres of type \((3,3)\).
\end{itemize}

Here, it should be emphasized that \(\langle 243,13\rangle\) is not root of a coclass tree,
since aside from its descendant \(\langle 2187,319\rangle\),
which is root of a coclass tree with metabelian mainline vertices,
it possesses five further descendants
which give rise to coclass trees with non-metabelian mainline vertices having cyclic centres of order \(3\)
and branches of considerable complexity,
here partially even with \textit{unbounded depth}
\cite[Thm.4.2(a--b)]{ELNO}.



{\tiny

\begin{figure}[ht]
\caption{\(2\)-Groups of Coclass \(3\) with Abelianization \((2,2,2)\)}
\label{fig:2GrpTyp222Cocl3}


\setlength{\unitlength}{1cm}
\begin{picture}(11,15)(-8,-14)

\put(-8,0.5){\makebox(0,0)[cb]{Order \(2^n\)}}
\put(-8,0){\line(0,-1){12}}
\put(-8,-12){\vector(0,-1){2}}
\multiput(-8.1,0)(0,-2){7}{\line(1,0){0.2}}
\put(-8.2,0){\makebox(0,0)[rc]{\(8\)}}
\put(-7.8,0){\makebox(0,0)[lc]{\(2^3\)}}
\put(-8.2,-2){\makebox(0,0)[rc]{\(16\)}}
\put(-7.8,-2){\makebox(0,0)[lc]{\(2^4\)}}
\put(-8.2,-4){\makebox(0,0)[rc]{\(32\)}}
\put(-7.8,-4){\makebox(0,0)[lc]{\(2^5\)}}
\put(-8.2,-6){\makebox(0,0)[rc]{\(64\)}}
\put(-7.8,-6){\makebox(0,0)[lc]{\(2^6\)}}
\put(-8.2,-8){\makebox(0,0)[rc]{\(128\)}}
\put(-7.8,-8){\makebox(0,0)[lc]{\(2^7\)}}
\put(-8.2,-10){\makebox(0,0)[rc]{\(256\)}}
\put(-7.8,-10){\makebox(0,0)[lc]{\(2^8\)}}
\put(-8.2,-12){\makebox(0,0)[rc]{\(512\)}}
\put(-7.8,-12){\makebox(0,0)[lc]{\(2^9\)}}

\put(-5,0){\line(1,-4){1}}
\put(-5,0){\line(1,-2){2}}
\put(-5,0){\line(3,-4){3}}
\put(-5,0){\line(1,-1){4}}
\put(-5,0){\line(5,-4){5}}
\put(-5,0){\line(3,-2){6}}
\put(1.2,-1.8){\makebox(0,0)[cc]{Edges of depth \(2\) forming the interface}}
\put(1.2,-2.3){\makebox(0,0)[cc]{between \(\mathcal{G}(2,2)\) and \(\mathcal{G}(2,3)\)}}

\put(0,-4){\line(-1,-4){1}}
\put(0,-4){\line(-1,-2){2}}
\put(-4.7,-6.7){\makebox(0,0)[cc]{Edges of depth \(2\) forming the interface}}
\put(-4.7,-7.2){\makebox(0,0)[cc]{between \(\mathcal{G}(2,3)\) and \(\mathcal{G}(2,4)\)}}

\put(-5,-2){\vector(0,-1){2}}
\put(-5,-3.5){\makebox(0,0)[rc]{main}}
\put(-5,-3.5){\makebox(0,0)[lc]{line}}
\put(-5.2,-4.3){\makebox(0,0)[rc]{\(\mathcal{T}^2(\langle 16,11\rangle)\)}}

\put(-5.1,-0.1){\framebox(0.2,0.2){}}

\put(-5,-2){\circle*{0.2}}
\put(-6,-2){\circle*{0.2}}

\put(-7,-2){\circle{0.2}}

\put(-5,0){\line(0,-1){2}}
\put(-5,0){\line(-1,-2){1}}
\put(-5,0){\line(-1,-1){2}}

\put(-4.9,0.3){\makebox(0,0)[lc]{\(\langle 5\rangle\simeq C_2\times C_2\times C_2\)}}

\put(-5.1,-1.8){\makebox(0,0)[rc]{\(\langle 11\rangle\)}}
\put(-6,-2.2){\makebox(0,0)[ct]{\(\langle 12\rangle\)}}
\put(-7,-2.2){\makebox(0,0)[ct]{\(\langle 13\rangle\)}}

\put(-2,-12){\vector(0,-1){2}}
\put(-1.8,-13.8){\makebox(0,0)[lc]{mainline}}
\put(-2.2,-14.3){\makebox(0,0)[rc]{\(\mathcal{T}^4(\langle 128,438\rangle)\)}}

\multiput(-4,-4)(1,0){5}{\circle*{0.2}}

\multiput(-2,-8)(0,-2){3}{\circle*{0.2}}
\multiput(-1,-8)(0,-2){3}{\circle*{0.2}}

\multiput(-2,-8)(0,-2){2}{\line(0,-1){2}}
\multiput(-2,-8)(0,-2){2}{\line(1,-2){1}}

\put(-4.1,-4.2){\makebox(0,0)[rc]{\(\langle 27\rangle\)}}
\put(-3.5,-4.2){\makebox(0,0)[cc]{\(\ldots\)}}
\put(-2.9,-4.2){\makebox(0,0)[lc]{\(\langle 31\rangle\)}}

\put(-2,-4.2){\makebox(0,0)[ct]{\(\langle 32\rangle\)}}
\put(-1,-4.2){\makebox(0,0)[ct]{\(\langle 33\rangle\)}}

\put(-0.1,-4.2){\makebox(0,0)[rc]{\(\langle 34\rangle\)}}

\put(-2.1,-7.8){\makebox(0,0)[rc]{\(\langle 438\rangle\)}}
\put(-1,-8.2){\makebox(0,0)[ct]{\(\langle 439\rangle\)}}

\put(-2.1,-9.8){\makebox(0,0)[rc]{\(\langle 5491\rangle\)}}
\put(-1,-10.2){\makebox(0,0)[ct]{\(\langle 5492\rangle\)}}

\put(-2.1,-11.8){\makebox(0,0)[rc]{\(\langle 58908\rangle\)}}
\put(-1,-12.2){\makebox(0,0)[ct]{\(\langle 58909\rangle\)}}

\put(-4,-4){\vector(0,-1){2}}
\put(-3.9,-4.8){\makebox(0,0)[lc]{five}}
\put(-3.9,-5.3){\makebox(0,0)[lc]{main}}
\put(-3.9,-5.8){\makebox(0,0)[lc]{lines}}
\put(-3,-4){\vector(0,-1){2}}

\put(0,-4){\vector(0,-1){2}}
\put(-0.5,-6.4){\makebox(0,0)[lc]{mainline}}

\put(1,-12){\vector(0,-1){2}}
\put(1.2,-13.8){\makebox(0,0)[lc]{mainline}}
\put(0.8,-14.3){\makebox(0,0)[rc]{\(\mathcal{T}^3(\langle 32,35\rangle)\)}}

\put(1,-4){\line(3,-4){1.5}}
\multiput(1,-4)(0,-2){4}{\line(0,-1){2}}
\multiput(1,-6)(0,-2){3}{\line(1,-2){1}}
\multiput(1,-6)(0,-2){3}{\line(1,-1){2}}

\multiput(1,-4)(0,-2){5}{\circle*{0.2}}
\put(2.5,-6){\circle*{0.2}}
\multiput(2,-8)(0,-2){3}{\circle*{0.2}}
\multiput(3,-8)(0,-2){3}{\circle*{0.2}}

\put(0.9,-4.2){\makebox(0,0)[rc]{\(\langle 35\rangle\)}}

\put(0.9,-5.8){\makebox(0,0)[rc]{\(\langle 181\rangle\)}}
\put(2.5,-6.2){\makebox(0,0)[ct]{\(\langle 180\rangle\)}}

\put(0.9,-7.8){\makebox(0,0)[rc]{\(\langle 984\rangle\)}}
\put(2,-8.2){\makebox(0,0)[ct]{\(\langle 985\rangle\)}}
\put(3.1,-8.2){\makebox(0,0)[ct]{\(\langle 986\rangle\)}}

\put(0.9,-9.8){\makebox(0,0)[rc]{\(\langle 6719\rangle\)}}
\put(2,-10.2){\makebox(0,0)[ct]{\(\langle 6720\rangle\)}}
\put(3.1,-10.2){\makebox(0,0)[ct]{\(\langle 6721\rangle\)}}

\put(0.9,-11.8){\makebox(0,0)[rc]{\(\langle 60891\rangle\)}}
\put(1.9,-12.2){\makebox(0,0)[ct]{\(\langle 60892\rangle\)}}
\put(3.2,-12.2){\makebox(0,0)[ct]{\(\langle 60893\rangle\)}}

\end{picture}

\end{figure}
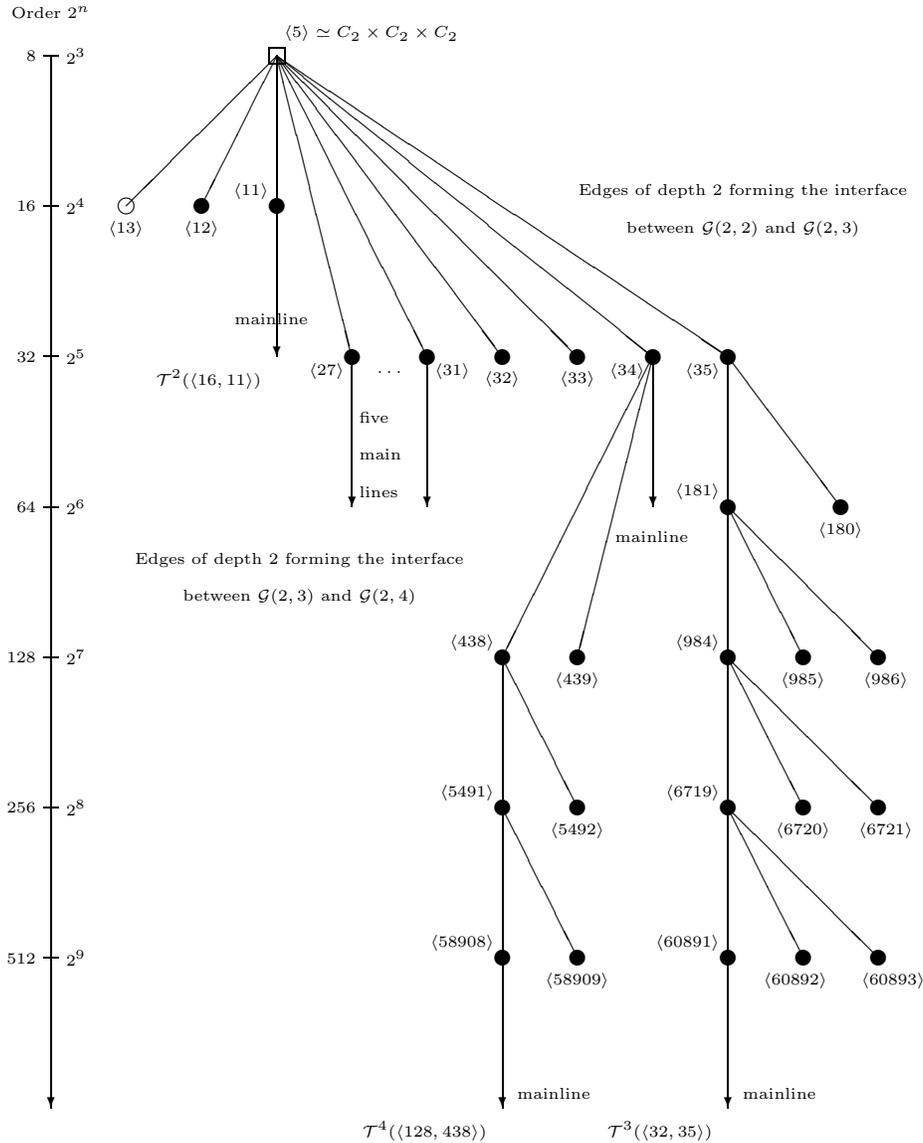

}



\subsubsection{Abelianization of type \((p,p,p)\)}
\label{sss:TypePPP}

For \(p=2\), resp. \(p=3\),
there exists a unique coclass tree with \(p\)-groups of type \((p,p,p)\)
in the coclass graph \(\mathcal{G}(p,2)\).
Its root is the elementary abelian \(p\)-group of type \((p,p,p)\),
that is, \(\langle 8,5\rangle\), resp. \(\langle 27,5\rangle\).
This unique tree corresponds to
the pro-2 group of the family \(\#59\) by M.F. Newman and E.A. O'Brien
\cite[Appendix A, no.59, p.153, Appendix B, Tbl.59, p.165]{NmOb},
resp. the pro-3 group given by the parameters \((f,g,h)=(0,0,0)\) in Table 1.
For \(p=2\), the tree is indicated in Figure 6.

Figure
\ref{fig:2GrpTyp222Cocl3}
shows some finite \(2\)-groups of coclass \(2,3,4\) and type \((2,2,2)\).



\subsection{Coclass \(3\)}
\label{ss:CoclassThree}

Here again, \(p\)-groups with several distinct abelianizations contribute
to the constitution of the coclass graph \(\mathcal{G}(p,3)\) .
There are regular, resp. irregular, essential contributions from groups \(G\)
with abelianizations \(G/G^\prime\) of the types
\((p^3,p)\), \((p^2,p^2)\), \((p^2,p,p)\), \((p,p,p,p)\),
resp. \((p,p)\),  \((p^2,p)\), \((p,p,p)\),
and an isolated contribution by the cyclic group of order \(p^4\).



\subsubsection{Abelianization of type \((p,p,p)\)}
\label{sss:TypePxPxP}

Since the elementary abelian \(p\)-group \(C_p\times C_p\times C_p\) of rank \(3\), that is,
\(\langle 8,5\rangle\), resp. \(\langle 27,5\rangle\), for \(p=2\), resp. \(p=3\),
is not coclass-settled, it gives rise to a multifurcation.
The regular component \(\mathcal{T}^2(C_p\times C_p\times C_p)\)
has been described in the section about coclass \(2\).
The irregular component \(\mathcal{T}^3(C_p\times C_p\times C_p)\)
becomes a subgraph \(\mathcal{G}=\mathcal{G}_{(p,p,p)}(p,3)\)
of the coclass graph \(\mathcal{G}(p,3)\)
when the connecting edges of depth \(2\) of the
irregular immediate descendants of \(C_p\times C_p\times C_p\) are removed.

For \(p=2\), this subgraph \(\mathcal{G}\) is contained in Figure
\ref{fig:2GrpTyp222Cocl3}.
It has nine top level vertices of order \(32=2^5\) which can be divided into terminal and capable vertices:

\begin{itemize}
\item
the groups \(\langle 32,32\rangle\) and \(\langle 32,33\rangle\) are leaves,
\item
the five groups \(\langle 32,27..31\rangle\) and
the two groups \(\langle 32,34..35\rangle\) are infinitely capable.
\end{itemize}

The trees arising from the capable vertices are associated
with infinite pro-\(2\) groups by M.F. Newman and E.A. O'Brien
\cite[Appendix A, no.73--79, pp.154--155, and Appendix B, Tbl.73--79, pp.167--168]{NmOb}
in the following manner:\\
\(\langle 32,28\rangle\) gives rise to\\
\(\mathcal{T}^3(\langle 64,140\rangle)\) associated with family \(\#73\),
and
\(\mathcal{T}^3(\langle 64,147\rangle)\) associated with family \(\#74\),\\
\(\mathcal{T}^3(\langle 32,29\rangle)\) is associated with family \(\#75\),\\
\(\mathcal{T}^3(\langle 32,30\rangle)\) is associated with family \(\#76\),\\
\(\mathcal{T}^3(\langle 32,31\rangle)\) is associated with family \(\#77\),\\
\(\langle 32,34\rangle\) gives rise to\\
\(\mathcal{T}^3(\langle 64,174\rangle)\) associated with family \(\#78\) (see \S\
\ref{s:PeriodicBifurcations}),
and finally\\
\(\mathcal{T}^3(\langle 32,35\rangle)\) is associated with family \(\#79\) (see Figure
\ref{fig:2GrpTyp222Cocl3}).\\
The roots of the coclass trees \(\mathcal{T}^4(\langle 128,438\rangle)\) in Figure
\ref{fig:2GrpTyp222Cocl3}
and \(\mathcal{T}^4(\langle 128,444\vert 445\rangle)\) in Figure
\ref{fig:PeriodicBifurcation222}
are siblings.



\renewcommand{\arraystretch}{1.0}

\begin{table}[ht]
\caption{Class-\(2\) Quotients \(Q\) of certain metabelian \(2\)-Groups \(G\) of Type \((2,2,2)\)}
\label{tbl:Quotients222}
\begin{center}
\begin{tabular}{|c||c|c|c|c|c|}
\hline
 SmallGroups              & Hall Senior             & Schur multiplier   & \(2\)-rank of \(G^\prime\) & \(4\)-rank of \(G^\prime\) & maximum of              \\
 identifier of \(Q\)      & classification of \(Q\) & \(\mathcal{M}(Q)\) & \(r_2(G^\prime)\)          & \(r_4(G^\prime)\)          & \(r_2(H_i/H_i^\prime)\) \\
\hline
 \(\langle 32,32\rangle\) & \(32.040\)              & \((2)\)            & \(2\)                      & \(0\)                      & \(2\)                   \\
 \(\langle 32,33\rangle\) & \(32.041\)              & \((2)\)            & \(2\)                      & \(0\)                      & \(2\)                   \\
\hline
 \(\langle 32,29\rangle\) & \(32.037\)              & \((2,2)\)          & \(2\)                      & \(1\)                      & \(3\)                   \\
 \(\langle 32,30\rangle\) & \(32.038\)              & \((2,2)\)          & \(2\)                      & \(1\)                      & \(3\)                   \\
 \(\langle 32,35\rangle\) & \(32.035\)              & \((2,2)\)          & \(2\)                      & \(1\)                      & \(3\)                   \\
\hline
 \(\langle 32,28\rangle\) & \(32.036\)              & \((2,2,2)\)        & \(2\)                      & \(2\)                      & \(3\)                   \\
 \(\langle 32,27\rangle\) & \(32.033\)              & \((2,2,2,2)\)      & \(3\)                      & \(2\) or \(3\)             & \(4\)                   \\
\hline
\end{tabular}
\end{center}
\end{table}



\subsubsection{Hall-Senior classification}
\label{sss:HallSenior}

Seven of these nine top level vertices have been investigated by E. Benjamin, F. Lemmermeyer and C. Snyder
\cite[\S\ 2, Tbl.1]{BLS}
with respect to their occurrence as class-2 quotients \(Q=G/\gamma_3(G)\)
of bigger metabelian 2-groups \(G\) of type \((2,2,2)\) and with coclass \(3\),
which are exactly the members of the descendant trees of the seven vertices.
These authors use the classification of 2-groups by M. Hall and J.K. Senior
\cite{HaSn}
which is put in correspondence with the SmallGroups Library
\cite{BEO1,BEO2}
in Table
\ref{tbl:Quotients222}.
The complexity of the descendant trees of these seven vertices
increases with the 2-ranks and 4-ranks indicated in Table
\ref{tbl:Quotients222},
where the maximal subgroups of index \(2\) in \(G\) are denoted by \(H_i\), for \(1\le i\le 7\).



\section{History of descendant trees}
\label{s:HistoryDescTrees}

Descendant trees with central quotients as parents (P1) are implicit in P. Hall's 1940 paper
\cite{Hl}
about isoclinism of groups.
Trees with last non-trivial lower central quotients as parents (P2) were first presented by C. R. Leedham-Green
at the International Congress of Mathematicians in Vancouver, 1974
\cite{Nm}.
The first extensive tree diagrams have been drawn manually
by J. A. Ascione, G. Havas and Leedham-Green (1977)
\cite{AHL},
by Ascione (1979)
\cite{As1}
and by B. Nebelung (1989)
\cite{Ne}.
In the former two cases, the parent definition by means of the lower exponent-\(p\) central series (P3)
was adopted in view of computational advantages,
in the latter case, where theoretical aspects were focused,
the parents were taken with respect to the usual lower central series (P2).

The kernels and targets of Artin transfer homomorphisms have recently turned out
to be compatible with parent-descendant relations between finite \(p\)-groups and
can favourably be used to endow descendant trees with additional structure
\cite{Ma4}.



\section{The construction: \(p\)-group generation algorithm}
\label{s:Construction}

The \textit{\(p\)-group generation algorithm} by M.F. Newman
\cite{Nm2}
and E.A. O'Brien
\cite{Ob,HEO}
is a recursive process for constructing the descendant tree
of an assigned finite \(p\)-group which is taken as the root of the tree.
It is discussed in some detail in \S\S\
\ref{s:LowerExponentP}--\ref{s:SchurMpl}.



\section{Lower exponent-\(p\) central series}
\label{s:LowerExponentP}

For a finite \(p\)-group \(G\), the \textit{lower exponent-\(p\) central series}
(briefly lower \(p\)-central series) of \(G\)
is a descending series \((P_j(G))_{j\ge 0}\) of characteristic subgroups of \(G\),
defined recursively by

\begin{equation}
\label{eqn:RecursiveLowerExpCS}
P_0(G):=G\text{ and }P_j(G):=\lbrack P_{j-1}(G),G\rbrack\cdot P_{j-1}(G)^p,\text{ for }j\ge 1.
\end{equation}

\noindent
Since any non-trivial finite \(p\)-group \(G>1\) is nilpotent,
there exists an integer \(c\ge 1\) such that \(P_{c-1}(G)>P_c(G)=1\)
and \(\mathrm{cl}_p(G):=c\) is called the \textit{exponent-\(p\) class} (briefly \(p\)-class) of \(G\).
Only the trivial group \(1\) has \(\mathrm{cl}_p(1)=0\).
Generally, for any finite \(p\)-group \(G\),
its \(p\)-class can be defined as
\(\mathrm{cl}_p(G):=\min\lbrace c\ge 0\mid P_c(G)=1\rbrace\).

The complete lower \(p\)-central series of \(G\) is therefore given by

\begin{equation}
\label{eqn:CompleteLowerExpCS}
G=P_0(G)>\Phi(G)=P_1(G)>P_2(G)>\cdots>P_{c-1}(G)>P_c(G)=1,
\end{equation}

\noindent
since \(P_1(G)=\lbrack P_0(G),G\rbrack\cdot P_0(G)^p=\lbrack G,G\rbrack\cdot G^p=\Phi(G)\)
is the \textit{Frattini subgroup} of \(G\).

For the convenience of the reader and for pointing out the shifted numeration, we recall that
the (usual) \textit{lower central series} of \(G\) is also a descending series \((\gamma_j(G))_{j\ge 1}\)
of characteristic subgroups of \(G\),
defined recursively by

\begin{equation}
\label{eqn:RecursiveLowerCS}
\gamma_1(G):=G\text{ and }\gamma_j(G):=\lbrack\gamma_{j-1}(G),G\rbrack,\text{ for }j\ge 2.
\end{equation}

\noindent
As above, for any non-trivial finite \(p\)-group \(G>1\),
there exists an integer \(c\ge 1\) such that \(\gamma_c(G)>\gamma_{c+1}(G)=1\)
and \(\mathrm{cl}(G):=c\) is called the \textit{nilpotency class} of \(G\),
whereas \(c+1\) is called the \textit{index of nilpotency} of \(G\).
Only the trivial group \(1\) has \(\mathrm{cl}(1)=0\).

Thus, the complete lower central series of \(G\) is given by

\begin{equation}
\label{eqn:CompleteLowerCS}
G=\gamma_1(G)>G^{\prime}=\gamma_2(G)>\gamma_3(G)>\cdots>\gamma_c(G)>\gamma_{c+1}(G)=1,
\end{equation}

\noindent
since \(\gamma_2(G)=\lbrack\gamma_1(G),G\rbrack=\lbrack G,G\rbrack=G^{\prime}\)
is the \textit{commutator subgroup} or \textit{derived subgroup} of \(G\).

The following \textit{rules} should be remembered for the exponent-\(p\) class:

\noindent
Let \(G\) be a finite \(p\)-group.

\begin{enumerate}[({R}1)]
\item
\(\mathrm{cl}(G)\le\mathrm{cl}_p(G)\),
since the \(\gamma_j(G)\) descend more quickly than the \(P_j(G)\).
\item
\(\vartheta\in\mathrm{Hom}(G,\tilde{G})\), for some group \(\tilde{G}\)
\(\Rightarrow\) \(\vartheta(P_j(G))=P_j(\vartheta(G))\), for any \(j\ge 0\).
\item
For any \(c\ge 0\),
the conditions \(N\unlhd G\) and \(\mathrm{cl}_p(G/N)=c\) imply \(P_c(G)\le N\).
\item
For any \(c\ge 0\),
\(\mathrm{cl}_p(G)=c\) \(\Rightarrow\) \(\mathrm{cl}_p(G/P_k(G))=\min(k,c)\), for all \(k\ge 0\),
in particular, \(\mathrm{cl}_p(G/P_k(G))=k\), for all \(0\le k\le c\).
\end{enumerate}

We point out that every non-trivial finite \(p\)-group \(G>1\) defines a \textit{maximal path}
with respect to the parent definition (P3), consisting of \(c\) edges,

\begin{equation}
\label{eqn:MaxPath}
\begin{aligned}
G=G/1=G/P_c(G)\to\pi(G)=G/P_{c-1}(G)\to\pi^2(G)=G/P_{c-2}(G)\to\cdots\\
\cdots\to\pi^{c-1}(G)=G/P_1(G)\to\pi^c(G)=G/P_0(G)=G/G=1
\end{aligned}
\end{equation}

\noindent
and ending in the trivial group \(\pi^c(G)=1\).
The last but one quotient of the maximal path of \(G\) is the
elementary abelian \(p\)-group \(\pi^{c-1}(G)=G/P_1(G)\simeq C_p^d\) of rank \(d=d(G)\),
where \(d(G)=\dim_{\mathbb{F}_p}(H^1(G,\mathbb{F}_p))\) denotes the generator rank of \(G\).



\section{\(p\)-covering group, \(p\)-multiplicator and nucleus}
\label{s:CoveringGroup}

Let \(G\) be a finite \(p\)-group with \(d\) \textit{generators}.
Our goal is to compile a complete list of pairwise non-isomorphic immediate descendants of \(G\).
It turned out that all immediate descendants can be obtained as quotients of a certain extension \(G^{\ast}\) of \(G\)
which is called the \textit{\(p\)-covering group} of \(G\) and can be constructed in the following manner.

We can certainly find a \textit{presentation} of \(G\) in the form of an exact sequence

\begin{equation}
\label{eqn:FreePresentation}
1\longrightarrow R\longrightarrow F\longrightarrow G\longrightarrow 1,
\end{equation}

\noindent
where \(F\) denotes the free group with \(d\) generators and
\(\vartheta:\ F\longrightarrow G\) is an epimorphism with kernel \(R:=\ker(\vartheta)\).
Then \(R\triangleleft F\) is a normal subgroup of \(F\) consisting of the defining \textit{relations} for \(G\simeq F/R\).
For elements \(r\in R\) and \(f\in F\),
the conjugate \(f^{-1}rf\in R\) and thus also the commutator \(\lbrack r,f\rbrack=r^{-1}f^{-1}rf\in R\) are contained in \(R\).
Consequently, \(R^{\ast}:=\lbrack R,F\rbrack\cdot R^p\) is a characteristic subgroup of \(R\),
and the \textit{\(p\)-multiplicator} \(R/R^{\ast}\) of \(G\) is an elementary abelian \(p\)-group,
since

\begin{equation}
\label{eqn:Multiplicator}
\lbrack R,R\rbrack\cdot R^p\le\lbrack R,F\rbrack\cdot R^p=R^{\ast}\unlhd R.
\end{equation}

\noindent
Now we can define the \(p\)-covering group of \(G\) by

\begin{equation}
\label{eqn:CoveringGroup}
G^{\ast}:=F/R^{\ast},
\end{equation}

\noindent
and the exact sequence

\begin{equation}
\label{eqn:Extension}
1\longrightarrow R/R^{\ast}\longrightarrow F/R^{\ast}\longrightarrow F/R\longrightarrow 1
\end{equation}

\noindent
shows that \(G^{\ast}\) is an extension of \(G\) by the elementary abelian \(p\)-multiplicator.
We call

\begin{equation}
\label{eqn:MultiplicatorRank}
\mu(G):=\dim_{\mathbb{F}_p}(R/R^{\ast})
\end{equation}

\noindent
the \textit{\(p\)-multiplicator rank} of \(G\).

Let us assume now that the assigned finite \(p\)-group \(G=F/R\) is of \(p\)-class \(\mathrm{cl}_p(G)=c\).
Then the  conditions \(R\unlhd F\) and \(\mathrm{cl}_p(F/R)=c\) imply \(P_c(F)\le R\),
according to the rule (R3),
and we can define the \textit{nucleus} of \(G\) by

\begin{equation}
\label{eqn:Nucleus}
P_c(G^{\ast})=P_c(F)\cdot R^{\ast}/R^{\ast}\le R/R^{\ast}
\end{equation}

\noindent
as a subgroup of the \(p\)-multiplicator.
Consequently, the \textit{nuclear rank}

\begin{equation}
\label{eqn:NuclearRank}
\nu(G):=\dim_{\mathbb{F}_p}(P_c(G^{\ast}))\le\mu(G)
\end{equation}

\noindent
of \(G\) is bounded from above by the \(p\)-multiplicator rank.



\section{Allowable subgroups of the \(p\)-multiplicator}
\label{s:Allowable}

As before, let \(G\) be a finite \(p\)-group with \(d\) generators.
Any \(p\)-elementary abelian central extension \(1\to Z\to H\to G\to 1\) of \(G\)
by a \(p\)-elementary abelian subgroup \(Z\le\zeta_1(H)\) such that \(d(H)=d(G)=d\)
is a quotient of the \(p\)-covering group \(G^{\ast}\) of \(G\).

The reason is that there exists an epimorphism \(\psi:\ F\to H\) such that \(\vartheta=\omega\circ\psi\),
where \(\omega:\ H\to G=H/Z\) denotes the canonical projection.
Consequently, we have \(R=\ker(\vartheta)=\ker(\omega\circ\psi)=\psi^{-1}(Z)\)
and thus \(\psi(R)=\psi(\psi^{-1}(Z))=Z\).
Further, \(\psi(R^p)\le Z^p=1\), since \(Z\) is \(p\)-elementary,
and \(\psi(\lbrack R,F\rbrack)\le\lbrack Z,Z\rbrack=1\), since \(Z\) is central.
Together this shows that \(\psi(R^{\ast})=\psi(\lbrack R,F\rbrack\cdot R^p)=1\)
and thus \(\psi\) induces the desired epimorphism \(\psi^\ast:\ G^{\ast}\to H\)
such that \(H\simeq G^{\ast}/\ker(\psi^\ast)\).

In particular, an immediate descendant \(H\) of \(G\) is
a \(p\)-elementary abelian central extension 

\begin{equation}
\label{eqn:CentralExtension}
1\to P_{c-1}(H)\to H\to G\to 1
\end{equation}

\noindent
of \(G\), since
\[1=P_c(H)=\lbrack P_{c-1}(H),H\rbrack\cdot P_{c-1}(H)^p\text{ implies }P_{c-1}(H)^p=1\text{ and }P_{c-1}(H)\le\zeta_1(H),\]
where \(c=\mathrm{cl}_p(H)\).

A subgroup \(M/R^{\ast}\le R/R^{\ast}\) of the \(p\)-multiplicator of \(G\) is called \textit{allowable}
if it is given by the kernel \(M/R^{\ast}=\ker(\psi^\ast)\)
of an epimorphism \(\psi^\ast:\ G^{\ast}\to H\)
onto an immediate descendant \(H\) of \(G\).
An equivalent characterization is that \(1<M/R^{\ast}<R/R^{\ast}\) is
a proper subgroup which \textit{supplements the nucleus}

\begin{equation}
\label{eqn:SupplementNucleus}
(M/R^{\ast})\cdot(P_c(F)\cdot R^{\ast}/R^{\ast})=R/R^{\ast}.
\end{equation}

Therefore, the first part of our goal to compile a list of all immediate descendants of \(G\) is done,
when we have constructed all allowable subgroups of \(R/R^{\ast}\)
which supplement the nucleus \(P_c(G^{\ast})=P_c(F)\cdot R^{\ast}/R^{\ast}\),
where \(c=\mathrm{cl}_p(G)\).
However, in general the list

\begin{equation}
\label{eqn:DescList}
\lbrace F/M\quad\mid\quad M/R^{\ast}\le R/R^{\ast}\text{ is allowable }\rbrace,
\end{equation}

\noindent
where \(G^{\ast}/(M/R^{\ast})=(F/R^{\ast})/(M/R^{\ast})\simeq F/M\)
will be redundant,
due to isomorphisms \(F/M_1\simeq F/M_2\) among the immediate descendants.



\section{Orbits under extended automorphisms}
\label{s:Orbits}

Two allowable subgroups \(M_1/R^{\ast}\) and \(M_2/R^{\ast}\) are called \textit{equivalent}
if the quotients \(F/M_1\simeq F/M_2\),
that are the corresponding immediate descendants of \(G\), are isomorphic.

Such an isomorphism \(\varphi:\ F/M_1\to F/M_2\) between immediate descendants of \(G=F/R\)
with \(c=\mathrm{cl}_p(G)\) has the property that
\[\varphi(R/M_1)=\varphi(P_c(F/M_1))=P_c(\varphi(F/M_1))=P_c(F/M_2)=R/M_2\]
and thus induces an automorphism \(\alpha\in\mathrm{Aut}(G)\) of \(G\)
which can be extended to an automorphism \(\alpha^\ast\in\mathrm{Aut}(G^\ast)\)
of the \(p\)-covering group \(G^\ast=F/R^\ast\) of \(G\).
The restriction of this  \textit{extended automorphism} \(\alpha^\ast\)
to the \(p\)-multiplicator \(R/R^\ast\) of \(G\) is determined uniquely by \(\alpha\).
Since
\[\alpha^\ast(M/R^{\ast})\cdot P_c(F/R^{\ast})=\alpha^\ast\lbrack M/R^{\ast}\cdot P_c(F/R^{\ast})\rbrack=\alpha^\ast(R/R^\ast)=R/R^\ast,\]
according to the rule (R2),
each extended automorphism \(\alpha^\ast\in\mathrm{Aut}(G^\ast)\)
induces a \textit{permutation} \(\tilde{\alpha}\) of the allowable subgroups \(M/R^{\ast}\le R/R^{\ast}\).
We define

\begin{equation}
\label{eqn:PermGroup}
P:=\langle\tilde{\alpha}\mid\alpha\in\mathrm{Aut}(G)\rangle
\end{equation}

\noindent
to be the \textit{permutation group} generated by all permutations induced by automorphisms of \(G\).
Then the map \(\mathrm{Aut}(G)\to P\), \(\alpha\mapsto\tilde{\alpha}\) is an epimorphism
and the equivalence classes of allowable subgroups \(M/R^{\ast}\le R/R^{\ast}\)
are precisely the \textit{orbits} of allowable subgroups \textit{under the action of} the permutation group \(P\).

Eventually, our goal to compile a list \(\lbrace F/M_i\mid 1\le i\le N\rbrace\)
of all immediate descendants of \(G\) will be done,
when we select a representative \(M_i/R^{\ast}\)
for each of the \(N\) orbits of allowable subgroups of \(R/R^{\ast}\) under the action of \(P\).
This is precisely what the \textit{\(p\)-group generation algorithm} does
in a single step of the recursive procedure for constructing the descendant tree of an assigned root.



\section{Capable \(p\)-groups and step sizes}
\label{s:StepSizes}

We recall from \S\
\ref{s:TreeDiagram}
that a finite \(p\)-group \(G\) is called \textit{capable} (or \textit{extendable})
if it possesses at least one immediate descendant,
otherwise it is called \textit{terminal} (or a \textit{leaf}).
As mentioned in \S\
\ref{s:Multifurcation}
already, the nuclear rank \(\nu(G)\) of \(G\) admits a decision about the capability of \(G\):

\begin{itemize}
\item
\(G\) is terminal if and only if \(\nu(G)=0\).
\item
\(G\) is capable if and only if \(\nu(G)\ge 1\).
\end{itemize}

In the case of capability, \(G=F/R\) has immediate descendants
of \(\nu=\nu(G)\) different \textit{step sizes} \(1\le s\le\nu\),
in dependence on the index

\begin{equation}
\label{eqn:IndexOfAllowable}
(R/R^\ast:M/R^\ast)=p^s
\end{equation}

\noindent
of the corresponding allowable subgroup \(M/R^\ast\)
in the \(p\)-multiplicator \(R/R^\ast\).
When \(G\) is of order \(\vert G\vert=p^n\),
then an immediate descendant of step size \(s\) is of order
\[\#(F/M)=(F/R^\ast:M/R^\ast)=(F/R^\ast:R/R^\ast)\cdot (R/R^\ast:M/R^\ast)=\vert G\vert\cdot p^s=p^n\cdot p^s=p^{n+s}.\]
For the related phenomenon of \textit{multifurcation} of a descendant tree
at a vertex \(G\) with nuclear rank \(\nu(G)\ge 2\)
see \S\
\ref{s:Multifurcation}
on multifurcation and coclass graphs.

The \(p\)-group generation algorithm provides the flexibility
to restrict the construction of immediate descendants to those of a single fixed step size \(1\le s\le\nu\),
which is very convenient in the case of huge descendant numbers (see the next section).



\section{Numbers of immediate descendants}
\label{s:DescNumbers}

We denote the \textit{number of all immediate descendants},
resp. \textit{immediate descendants of step size \(s\)}, of \(G\)
by \(N\), resp. \(N_s\).
Then we have \(N=\sum_{s=1}^\nu\,N_s\).
As concrete examples, we present some interesting finite metabelian \(p\)-groups
with extensive sets of immediate descendants,
using the SmallGroups identifiers and
additionally pointing out the \textit{numbers \(0\le C_s\le N_s\) of capable immediate descendants}
in the usual format

\begin{equation}
\label{eqn:NumbersOfDesc}
(N_1/C_1;\ldots;N_\nu/C_\nu)
\end{equation}

\noindent
as given by actual implementations of the \(p\)-group generation algorithm
in the computational algebra systems GAP and MAGMA.
These invariants completely determine the local structure of the descendant tree
\(\mathcal{T}(G)\).

First, let \(p=2\).
We begin with groups having abelianization of type \((2,2,2)\).
See Figure
\ref{fig:2GrpTyp222Cocl3}.

\begin{itemize}
\item
The group \(\langle 32,35\rangle\) of coclass \(3\) has ranks \(\nu=1\), \(\mu=5\)
and descendant numbers \((4/1)\), \(N=4\). See \S\
\ref{s:PeriodicBifurcations}.
\item
The group \(\langle 32,34\rangle\) of coclass \(3\) has ranks \(\nu=2\), \(\mu=6\)
and descendant numbers \((6/1;19/6)\), \(N=25\). See \S\
\ref{s:PeriodicBifurcations}.
\item
The group \(\langle 32,27\rangle\) of coclass \(3\) has ranks \(\nu=3\), \(\mu=7\)
and\\
descendant numbers \((12/2;70/25;104/85)\), \(N=186\).
\end{itemize}

Next, let \(p=3\).
We consider groups having abelianization of type \((3,3)\).
See Figure
\ref{fig:3GrpTyp33Cocl2}.

\begin{itemize}
\item
The group \(\langle 27,3\rangle\) of coclass \(1\) has ranks \(\nu=2\), \(\mu=4\)
and descendant numbers \((4/1;7/5)\), \(N=11\).
\item
The group \(\langle 243,3\rangle=\langle 27,3\rangle-\#2;1\) of coclass \(2\) has ranks \(\nu=2\), \(\mu=4\)
and descendant numbers \((10/6;15/15)\), \(N=25\).
\item
One of its immediate descendants,
the group \(B=\langle 729,40\rangle=\langle 243,3\rangle-\#1;7\), has ranks \(\nu=2\), \(\mu=5\)
and descendant numbers \((16/2;27/4)\), \(N=43\).
\end{itemize}

In contrast, groups with abelianization of type \((3,3,3)\)
are partially located beyond the limit of actual computability.

\begin{itemize}
\item
The group \(\langle 81,12\rangle\) of coclass \(2\) has ranks \(\nu=2\), \(\mu=7\)
and\\
descendant numbers \((10/2;100/50)\), \(N=110\).
\item
The group \(\langle 243,37\rangle\) of coclass \(3\) has ranks \(\nu=5\), \(\mu=9\)
and descendant numbers \((35/3;2\,783/186;81\,711/10\,202;350\,652/202\,266;\ldots)\), \(N>4\cdot 10^5\) unknown.
\item
The group \(\langle 729,122\rangle\) of coclass \(4\) has ranks \(\nu=8\), \(\mu=11\)
and descendant numbers \((45/3;117\,919/1\,377;\ldots)\), \(N>10^5\) unknown.
\end{itemize}






\section{Schur multiplier}
\label{s:SchurMpl}

Via the isomorphism
\[\mathbb{Q}/\mathbb{Z}\to\mu_{\infty},\ \frac{n}{d}\mapsto\exp(\frac{n}{d}\cdot 2\pi i)\]
the quotient group

\begin{equation}
\label{eqn:AdditiveRootsOfUnity}
\mathbb{Q}/\mathbb{Z}=\lbrace\frac{n}{d}\cdot\mathbb{Z}\mid d\ge 1,\ 0\le n\le d-1\rbrace
\end{equation}

\noindent
can be viewed as the additive analogue of the multiplicative group

\begin{equation}
\label{eqn:RootsOfUnity}
\mu_{\infty}=\lbrace z\in\mathbb{C}\mid z^d=1 \text{ for some integer } d\ge 1\rbrace
\end{equation}

\noindent
of all roots of unity.

Let \(p\) be a prime number and \(G\) be a finite \(p\)-group with presentation \(G=F/R\) as in the previous section.
Then the second cohomology group

\begin{equation}
\label{eqn:SchurMpl}
M(G):=H^2(G,\mathbb{Q}/\mathbb{Z})
\end{equation}

\noindent
of the \(G\)-module \(\mathbb{Q}/\mathbb{Z}\)
is called the \textit{Schur multiplier} of \(G\).
It can also be interpreted as the quotient group

\begin{equation}
\label{eqn:FreeSchurMpl}
M(G)=(R\cap\lbrack F,F\rbrack)/\lbrack F,R\rbrack.
\end{equation}

I.R. Shafarevich
\cite[\S\ 6, p.146]{Sh}
has proved that the difference between the \textit{relation rank}
\(r(G)=\dim_{\mathbb{F}_p}(H^2(G,\mathbb{F}_p))\) of \(G\)
and the \textit{generator rank}
\(d(G)=\dim_{\mathbb{F}_p}(H^1(G,\mathbb{F}_p))\) of \(G\)
is given by the minimal number of generators of the Schur multiplier of \(G\),
that is

\begin{equation}
\label{eqn:Shafarevich}
r(G)-d(G)=d(M(G)).
\end{equation}

N. Boston and H. Nover
\cite[\S\ 3.2, Prop.2]{BoNo}
have shown that

\begin{equation}
\label{eqn:BostonNover}
\mu(G_j)-\nu(G_j)\le r(G),
\end{equation}

\noindent
for all quotients \(G_j:=G/P_j(G)\) of \(p\)-class \(\mathrm{cl}_p(G_j)=j\), \(j\ge 0\),
of a pro-\(p\) group \(G\) with finite abelianization \(G/G^\prime\).

Furthermore, J. Blackhurst
(in the appendix \textit{On the nucleus of certain p-groups} of a paper by N. Boston, M.R. Bush and F. Hajir
\cite{BBH})
has proved that a non-cyclic finite \(p\)-group \(G\) with trivial Schur multiplier \(M(G)\)
is a terminal vertex in the descendant tree \(\mathcal{T}(1)\) of the trivial group \(1\),
that is,

\begin{equation}
\label{eqn:Blackhurst}
M(G)=1 \Rightarrow \nu(G)=0.
\end{equation}

We conclude this section by giving two examples.

\begin{itemize}
\item
A finite \(p\)-group \(G\) has a balanced presentation \(r(G)=d(G)\) if and only if \(r(G)-d(G)=0=d(M(G))\),
that is, if and only if its Schur multiplier \(M(G)=1\) is trivial. 
Such a group is called a \textit{Schur group}
\cite{Ag,BBH,BuMa,KoVe}
and it must be a leaf in the descendant tree \(\mathcal{T}(1)\).
\item
A finite \(p\)-group \(G\) satisfies \(r(G)=d(G)+1\) if and only if \(r(G)-d(G)=1=d(M(G))\),
that is, if and only if it has a non-trivial cyclic Schur multiplier \(M(G)\). 
Such a group is called a \textit{Schur\(+1\) group}.
\end{itemize}



\section{Pruning strategies}
\label{s:PruningStrategies}

For \textit{searching} a particular group in a descendant tree \(\mathcal{T}(R)\) 
by looking for \textit{patterns} defined by the kernels and targets of Artin transfer homomorphisms
\cite{Ma4},
it is frequently adequate to reduce the number of vertices
in the branches of a dense tree with high complexity
by sifting groups with desired special properties, for example

\begin{enumerate}[({F}1)]
\item  filtering the \(\sigma\)-groups (see Definition
\ref{dfn:SchurSigmaGroup}),
\item  eliminating a set of certain transfer kernel types (TKTs, see
\cite[pp.403--404]{Ma4}),
\item  cancelling all non-metabelian groups (thus restricting to the \textit{metabelian skeleton}),
\item  removing metabelian groups with cyclic centre (usually of higher complexity),
\item  cutting off vertices whose distance from the mainline (depth) exceeds some lower bound,
\item  combining several different sifting criteria.
\end{enumerate}

The result of such a sieving procedure is called a \textit{pruned descendant tree}
\(\mathcal{T}_\ast(R)<\mathcal{T}(R)\) with respect to the desired set of properties.

However, in any case,
it should be avoided that the mainline of a coclass tree is eliminated,
since the result would be a disconnected infinite set of finite graphs instead of a tree.
We expand this idea further in the following detailed discussion of new phenomena.



\section{Striking news: periodic bifurcations in trees}
\label{s:PeriodicBifurcations}

We begin this section about brand-new discoveries
with the most recent example of periodic bifurcations in trees of \(2\)-groups.
It has been found on the 17th of January, 2015, motivated by
a search for metabelian \(2\)-class tower groups
\cite{Ma1}
of complex quadratic fields
\cite{No}
and complex bicyclic biquadratic Dirichlet fields
\cite{AZT}.



\subsection{Finite \(2\)-groups \(G\) with \(G/G^\prime\simeq (2,2,2)\)}
\label{ss:2Groups222}

The \(2\)-groups under investigation are three-generator groups
with elementary abelian commutator factor group of type \((2,2,2)\).
As shown in Figure
\ref{fig:2GrpTyp222Cocl3}
of \S\
\ref{s:ConcreteExamples},
all such groups are descendants of the abelian root \(\langle 8,5\rangle\).
Among its immediate descendants of step size \(2\),
there are three groups which reveal multifurcation.
\(\langle 32,27\rangle\) has nuclear rank \(\nu=3\),
giving rise to \(3\)-fold multifurcation.
The two groups \(\langle 32,28\rangle\) and \(\langle 32,34\rangle\)
possess the required nuclear rank \(\nu=2\) for \textit{bifurcation}.
Due to the arithmetical origin of the problem,
we focused on the latter, \(G:=\langle 32,34\rangle\), and
constructed an extensive finite part of its pruned descendant tree
\(\mathcal{T}_\ast(G)\),
using the \(p\)-group generation algorithm
\cite{Nm2,Ob,HEO}
as implemented in the computational algebra system Magma
\cite{BCP,BCFS,MAGMA}.
All groups turned out to be \textit{metabelian}.



\begin{remark}
\label{rmk:Sifting222}
Since our primary intention is
to provide a sound group theoretic background for several phenomena
discovered in class field theory and algebraic number theory,
we eliminated superfluous brushwood in the descendant trees
to avoid unnecessary complexity.

The selected sifting process for reducing the entire descendant tree
\(\mathcal{T}(G)\)
to the pruned descendant tree
\(\mathcal{T}_\ast(G)\)
filters all vertices which satisfy one of the conditions in Equations
(\ref{eqn:LayeredTKTMainline})
or
(\ref{eqn:LayeredTKTPrdSequence}),
and essentially consists of pruning strategy (F2),
more precisely, of

\begin{enumerate}
\item
omitting all the \(13\) terminal step size-\(2\) descendants, and
\(5\), resp. \(4\), of the \(6\) capable step size-\(2\) descendants,
together with their complete descendant trees, 
in Theorem
\ref{thm:TwoBifurcation},
resp. Corollary
\ref{cor:TwoBifurcation},
and
\item
eliminating
all, resp. \(4\), of the \(5\) terminal step size-\(1\) descendants
in Theorem
\ref{thm:TwoBifurcation},
resp. Corollary
\ref{cor:TwoBifurcation}.
\end{enumerate}

\end{remark}



Denote by \(x,y,z\) the generators of a finite \(2\)-group \(G=\langle x,y,z\rangle\)
with abelian type invariants \((2,2,2)\).
We fix an ordering of the seven maximal normal subgroups by putting

\begin{equation}
\label{eqn:MaxSbgr222}
\begin{aligned}
S_1=\langle y,z,G^\prime\rangle,
S_2=\langle z,x,G^\prime\rangle,
S_3=\langle x,y,G^\prime\rangle,\\
S_4=\langle x,yz,G^\prime\rangle,
S_5=\langle y,zx,G^\prime\rangle,
S_6=\langle z,xy,G^\prime\rangle,\\
S_7=\langle xy,yz,G^\prime\rangle.
\end{aligned}
\end{equation}

Just within this subsection,
we select a special designation for a TKT
\cite[pp.403--404]{Ma4}
whose first layer consists exactly of
all these seven planes in the \(3\)-dimensional
\(\mathbb{F}_2\)-vector space \(G/G^\prime\),
in any ordering.

\begin{definition}
\label{dfn:Permutation}
The transfer kernel type (TKT) \(\varkappa=\lbrack\varkappa_0;\varkappa_1;\varkappa_2;\varkappa_3\rbrack\)
is called a \textit{permutation} if all seven members of the first layer \(\varkappa_1\)
are maximal subgroups of \(G\) and there exists a permutation \(\sigma\in S_7\) such that
\(\varkappa_1=(S_{\sigma(j)})_{1\le j\le 7}\).
\end{definition}



For brevity, we give
\(2\)-logarithms of abelian type invariants in the following theorem
and we denote iteration by formal exponents,
for instance,
\(1^3:=(1,1,1)\hat{=}(2,2,2)\),
\((1^3)^4:=(1^3,1^3,1^3,1^3)\),
\(0^7:=(0,0,0,0,0,0,0)\)  and
\((j+2,j+1)\hat{=}(2^{j+2},2^{j+1})\).
Further, we eliminate an initial anomaly of generalized identifiers
by putting \(G-\#2;1:=G-\#2;8\) and \(G-\#2;2:=G-\#2;9\), formally.

\begin{theorem}
\label{thm:TwoBifurcation}
Let \(\ell\) be a positive integer bounded from above by \(10\).

\begin{enumerate}
\item
In the descendant tree \(\mathcal{T}(G)\) of \(G=\langle 32,34\rangle\),
there exists a unique path of length \(\ell\),
\[G=\delta^0(G)\leftarrow\delta^1(G)\leftarrow\cdots\leftarrow\delta^\ell(G),\]
of (reverse) directed edges with uniform step size \(2\) such that
\(\delta^j(G)=\pi(\delta^{j+1}(G))\), for all \(0\le j\le\ell-1\)
(along the path, \(\delta\) is a section of the surjection \(\pi\)),
and all the vertices

\begin{equation}
\label{eqn:StepSizeTwoDesc}
\delta^j(G)=G(-\#2;1)^j
\end{equation}

\noindent
of this path share the following common invariants:

\begin{itemize}
\item
the transfer kernel type with layer \(\varkappa_1\) containing three \(2\)-cycles
(and nearly a permutation, except for the first component which is total, \(0\hat{=}\delta^j(G)\)),

\begin{equation}
\label{eqn:LayeredTKTMainline}
\varkappa(\delta^j(G))=\lbrack 1;(0,S_6,S_5,S_7,S_3,S_2,S_4);0^7;0\rbrack,
\end{equation}

\item
the \(2\)-multiplicator rank and
the nuclear rank, giving rise to the bifurcation,

\begin{equation}
\label{eqn:Ranks}
\mu(\delta^j(G))=6,\quad
\nu(\delta^j(G))=2,
\end{equation}

\item
and the counters of immediate descendants,

\begin{equation}
\label{eqn:Counters}
N_1(\delta^j(G))=6,\ C_1(\delta^j(G))=1,\quad N_2(\delta^j(G))=19,\ C_2(\delta^j(G))=6,
\end{equation}

\noindent
determining the local structure of the descendant tree.

\end{itemize}

\item
A few other invariants of the vertices \(\delta^j(G)\) depend on the superscript \(j\),

\begin{itemize}
\item
the \(2\)-logarithm of the order,
the nilpotency class and
the coclass,

\begin{equation}
\label{eqn:LogOrdClCc}
\log_2(\mathrm{ord}(\delta^j(G)))=2j+5,\quad
\mathrm{cl}(\delta^j(G))=j+2,\quad
\mathrm{cc}(\delta^j(G))=j+3,
\end{equation}

\item
a single component of layer \(\tau_1\), three components of layer \(\tau_2\),
and layer \(\tau_3\) of the transfer target type

\begin{equation}
\label{eqn:LayeredTTT}
\tau(\delta^j(G))=\lbrack 1^3;((j+2,j+2),(1^3)^6);((j+2,j+1)^3,(1^3)^4);(j+1,j+1)\rbrack.
\end{equation}

\end{itemize}

\end{enumerate}

\end{theorem}



In view of forthcoming number theoretic applications, we add the following

\begin{corollary}
\label{cor:TwoBifurcation}
Let \(0\le j\le\ell\) be a non-negative integer.

\begin{enumerate}
\item
The regular component \(\mathcal{T}^{j+3}(\delta^j(G))\)
of the descendant tree \(\mathcal{T}(\delta^j(G))\)
is a coclass tree which
contains a unique periodic sequence
whose vertices \(V_{j,k}:=\delta^j(G)(-\#1;1)^k-\#1;2\) with \(k\ge 0\)
are characterized by a permutation TKT

\begin{equation}
\label{eqn:LayeredTKTPrdSequence}
\varkappa(V_{j,k})=\lbrack 1;(S_1,S_6,S_5,S_7,S_3,S_2,S_4);0^7;0\rbrack,
\end{equation}

\noindent
with a single fixed point \(S_1\) and the same three \(2\)-cycles
\((S_2,S_6)\), \((S_3,S_5)\), \((S_4,S_7)\)
as in the mainline TKT of Equation
(\ref{eqn:LayeredTKTMainline}).
\item
The irregular component \(\mathcal{T}^{j+4}(\delta^j(G))\)
of the descendant tree \(\mathcal{T}(\delta^j(G))\)
is a forest which
contains a unique second coclass tree
\(\mathcal{T}^{j+4}(\delta^j(G)-\#2;2)\)
whose mainline vertices
\(M_{j+1,k}:=\delta^j(G)-\#2;2(-\#1;1)^k\) with \(k\ge 0\)
possess the same permutation TKT as in Equation
(\ref{eqn:LayeredTKTPrdSequence}),
apart from the first coclass tree
\(\mathcal{T}^{j+4}(\delta^j(G)-\#2;1)\),
where \(\delta^j(G)-\#2;1=\delta^{j+1}(G)\),
whose mainline vertices \(\delta^{j+1}(G)(-\#1;1)^k\) with \(k\ge 0\)
share the TKT in Equation
(\ref{eqn:LayeredTKTMainline}).
\end{enumerate}

\end{corollary}



\begin{proof}
(of Theorem
\ref{thm:TwoBifurcation},
Corollary
\ref{cor:TwoBifurcation} and
Theorem
\ref{thm:2ParamPres})\\
The \(p\)-group generation algorithm
\cite{Nm2,Ob,HEO}
as implemented in the Magma computational algebra system
\cite{BCP,BCFS,MAGMA}
was employed to construct the pruned descendant tree
\(\mathcal{T}_\ast(G)\) with root \(G=\langle 32,34\rangle\)
which we defined as the disjoint union of all pruned coclass trees
\(\mathcal{T}_\ast^{j+3}(\delta^j(G))\) with the
successive descendants \(\delta^j(G)=G(-\#2;1)^j\), \(0\le j\le 10\),
of step size \(2\) of \(G\) as roots.
Using the well-known virtual periodicity
\cite{dS,EkLg}
of each coclass tree \(\mathcal{T}^{j+3}(\delta^j(G))\),
which turned out to be strict and of the smallest possible length \(1\),
the vertical construction was terminated at nilpotency class \(12\),
considerably deeper than the point where periodicity sets in.
The horizontal construction was extended up to coclass \(13\),
where the amount of CPU time started to become annoying.
\end{proof}

Within the frame of our computations, the periodicity was not
restriced to bifurcations only:
It seems that the pruned (or maybe even the entire) descendant trees
\(\mathcal{T}_\ast(\delta^j(G))\) are all isomorphic to
\(\mathcal{T}_\ast(G)\) as graphs.
This is visualized impressively by Figure
\ref{fig:PeriodicBifurcation222}.



{\tiny

\begin{figure}[hb]
\caption{Periodic Bifurcations in \(\mathcal{T}_\ast(\langle 32,34\rangle)\)}
\label{fig:PeriodicBifurcation222}


\setlength{\unitlength}{0.8cm}
\begin{picture}(18,22)(-4,-21)

\put(-5,0.5){\makebox(0,0)[cb]{Order}}
\put(-5,0){\line(0,-1){18}}
\multiput(-5.1,0)(0,-2){10}{\line(1,0){0.2}}
\put(-5.2,0){\makebox(0,0)[rc]{\(8\)}}
\put(-4.8,0){\makebox(0,0)[lc]{\(2^3\)}}
\put(-5.2,-2){\makebox(0,0)[rc]{\(16\)}}
\put(-4.8,-2){\makebox(0,0)[lc]{\(2^4\)}}
\put(-5.2,-4){\makebox(0,0)[rc]{\(32\)}}
\put(-4.8,-4){\makebox(0,0)[lc]{\(2^5\)}}
\put(-5.2,-6){\makebox(0,0)[rc]{\(64\)}}
\put(-4.8,-6){\makebox(0,0)[lc]{\(2^6\)}}
\put(-5.2,-8){\makebox(0,0)[rc]{\(128\)}}
\put(-4.8,-8){\makebox(0,0)[lc]{\(2^7\)}}
\put(-5.2,-10){\makebox(0,0)[rc]{\(256\)}}
\put(-4.8,-10){\makebox(0,0)[lc]{\(2^8\)}}
\put(-5.2,-12){\makebox(0,0)[rc]{\(512\)}}
\put(-4.8,-12){\makebox(0,0)[lc]{\(2^9\)}}
\put(-5.2,-14){\makebox(0,0)[rc]{\(1\,024\)}}
\put(-4.8,-14){\makebox(0,0)[lc]{\(2^{10}\)}}
\put(-5.2,-16){\makebox(0,0)[rc]{\(2\,048\)}}
\put(-4.8,-16){\makebox(0,0)[lc]{\(2^{11}\)}}
\put(-5.2,-18){\makebox(0,0)[rc]{\(4\,096\)}}
\put(-4.8,-18){\makebox(0,0)[lc]{\(2^{12}\)}}
\put(-5,-18){\vector(0,-1){2}}

\put(-3,0){\makebox(0,0)[cc]{\(\mathcal{G}(2,2)\)}}

\put(-4.05,-0.05){\framebox(0.1,0.1){}}
\put(-3.9,0.2){\makebox(0,0)[lb]{\(\langle 5\rangle\)}}
\put(-4,0){\line(1,-1){4}}
\put(-4,0){\line(1,-2){2}}

\put(0.1,-3.8){\makebox(0,0)[lb]{\(\langle 34\rangle\) (not coclass-settled)}}
\put(0.1,-5.8){\makebox(0,0)[lb]{\(\langle 174\rangle\)}}
\put(0.1,-7.8){\makebox(0,0)[lb]{\(\langle 978\rangle\)}}
\put(0.1,-9.8){\makebox(0,0)[lb]{\(\langle 6713\rangle\)}}
\put(0.1,-11.8){\makebox(0,0)[lb]{\(\langle 60885\rangle\)}}
\put(0.1,-13.8){\makebox(0,0)[lb]{\(1;1\)}}
\put(0.1,-15.8){\makebox(0,0)[lb]{\(1;1\)}}
\put(0.1,-17.8){\makebox(0,0)[lb]{\(1;1\)}}
\multiput(0,-4)(0,-2){8}{\circle*{0.1}}
\multiput(0,-4)(0,-2){7}{\line(0,-1){2}}
\put(0,-18){\vector(0,-1){2}}
\put(0,-20.2){\makebox(0,0)[ct]{\(\mathcal{T}_\ast^3(\langle 32,34\rangle)\)}}

\put(-1,-4){\makebox(0,0)[cc]{\(\mathcal{G}(2,3)\)}}

\put(-2.1,-4.2){\makebox(0,0)[rt]{\(\langle 35\rangle\)}}
\put(-2.1,-6.2){\makebox(0,0)[rt]{\(\langle 181\rangle\)}}
\put(-2.1,-8.2){\makebox(0,0)[rt]{\(\langle 984\rangle\)}}
\put(-2.1,-10.2){\makebox(0,0)[rt]{\(\langle 6719\rangle\)}}
\put(-2.1,-12.2){\makebox(0,0)[rt]{\(\langle 60891\rangle\)}}
\put(-2.1,-14.2){\makebox(0,0)[rt]{\(1;1\)}}
\put(-2.1,-16.2){\makebox(0,0)[rt]{\(1;1\)}}
\put(-2.1,-18.2){\makebox(0,0)[rt]{\(1;1\)}}
\multiput(-2,-4)(0,-2){8}{\circle*{0.1}}
\multiput(-2,-4)(0,-2){7}{\line(0,-1){2}}
\put(-2,-18){\vector(0,-1){2}}
\put(-2,-20.5){\makebox(0,0)[ct]{\(\mathcal{T}_\ast^3(\langle 32,35\rangle)\)}}

\put(-1,-6.2){\makebox(0,0)[ct]{\(\langle 175\rangle\)}}
\put(-1,-8.2){\makebox(0,0)[ct]{\(\langle 979\rangle\)}}
\put(-1,-10.2){\makebox(0,0)[ct]{\(\langle 6714\rangle\)}}
\put(-1,-12.2){\makebox(0,0)[ct]{\(\langle 60886\rangle\)}}
\put(-1,-14.2){\makebox(0,0)[ct]{\(1;2\)}}
\put(-1,-16.2){\makebox(0,0)[ct]{\(1;2\)}}
\put(-1,-18.2){\makebox(0,0)[ct]{\(1;2\)}}
\multiput(0,-4)(0,-2){7}{\line(-1,-2){1}}
\multiput(-1,-6)(0,-2){7}{\circle*{0.1}}

\put(0,-4){\line(1,-1){4}}
\put(0,-4){\line(1,-2){2}}
\put(1.1,-4.8){\makebox(0,0)[lb]{\(1^{\text{st}}\) bifurcation}}

\put(4.1,-7.8){\makebox(0,0)[lb]{\(\langle 444\rangle\) (not coclass-settled)}}
\put(4.1,-9.8){\makebox(0,0)[lb]{\(\langle 5503\rangle\)}}
\put(4.1,-11.8){\makebox(0,0)[lb]{\(\langle 58920\rangle\)}}
\put(4.1,-13.8){\makebox(0,0)[lb]{\(1;1\)}}
\put(4.1,-15.8){\makebox(0,0)[lb]{\(1;1\)}}
\put(4.1,-17.8){\makebox(0,0)[lb]{\(1;1\)}}
\multiput(4,-8)(0,-2){6}{\circle*{0.1}}
\multiput(4,-8)(0,-2){5}{\line(0,-1){2}}
\put(4,-18){\vector(0,-1){2}}
\put(4,-20.2){\makebox(0,0)[ct]{\(\mathcal{T}_\ast^4(\langle 128,444\rangle)\)}}

\put(3,-8){\makebox(0,0)[cc]{\(\mathcal{G}(2,4)\)}}

\put(1.9,-8.2){\makebox(0,0)[rt]{\(\langle 445\rangle\)}}
\put(1.9,-10.2){\makebox(0,0)[rt]{\(\langle 5509\rangle\)}}
\put(1.9,-12.2){\makebox(0,0)[rt]{\(\langle 58926\rangle\)}}
\put(1.9,-14.2){\makebox(0,0)[rt]{\(1;1\)}}
\put(1.9,-16.2){\makebox(0,0)[rt]{\(1;1\)}}
\put(1.9,-18.2){\makebox(0,0)[rt]{\(1;1\)}}
\multiput(2,-8)(0,-2){6}{\circle*{0.1}}
\multiput(2,-8)(0,-2){5}{\line(0,-1){2}}
\put(2,-18){\vector(0,-1){2}}
\put(2,-20.5){\makebox(0,0)[ct]{\(\mathcal{T}_\ast^4(\langle 128,445\rangle)\)}}

\put(3,-10.2){\makebox(0,0)[ct]{\(\langle 5504\rangle\)}}
\put(3,-12.2){\makebox(0,0)[ct]{\(\langle 58921\rangle\)}}
\put(3,-14.2){\makebox(0,0)[ct]{\(1;2\)}}
\put(3,-16.2){\makebox(0,0)[ct]{\(1;2\)}}
\put(3,-18.2){\makebox(0,0)[ct]{\(1;2\)}}
\multiput(4,-8)(0,-2){5}{\line(-1,-2){1}}
\multiput(3,-10)(0,-2){5}{\circle*{0.1}}

\put(4,-8){\line(1,-1){4}}
\put(4,-8){\line(1,-2){2}}
\put(5.1,-8.8){\makebox(0,0)[lb]{\(2^{\text{nd}}\) bifurcation}}

\put(8.1,-11.8){\makebox(0,0)[lb]{\(\langle 30599\rangle\) (not coclass-settled)}}
\put(8.1,-13.8){\makebox(0,0)[lb]{\(1;1\)}}
\put(8.1,-15.8){\makebox(0,0)[lb]{\(1;1\)}}
\put(8.1,-17.8){\makebox(0,0)[lb]{\(1;1\)}}
\multiput(8,-12)(0,-2){4}{\circle*{0.1}}
\multiput(8,-12)(0,-2){3}{\line(0,-1){2}}
\put(8,-18){\vector(0,-1){2}}
\put(8,-20.2){\makebox(0,0)[ct]{\(\mathcal{T}_\ast^5(\langle 512,30599\rangle)\)}}

\put(7,-12){\makebox(0,0)[cc]{\(\mathcal{G}(2,5)\)}}

\put(5.9,-12.2){\makebox(0,0)[rt]{\(\langle 30600\rangle\)}}
\put(5.9,-14.2){\makebox(0,0)[rt]{\(1;1\)}}
\put(5.9,-16.2){\makebox(0,0)[rt]{\(1;1\)}}
\put(5.9,-18.2){\makebox(0,0)[rt]{\(1;1\)}}
\multiput(6,-12)(0,-2){4}{\circle*{0.1}}
\multiput(6,-12)(0,-2){3}{\line(0,-1){2}}
\put(6,-18){\vector(0,-1){2}}
\put(6,-20.5){\makebox(0,0)[ct]{\(\mathcal{T}_\ast^5(\langle 512,30600\rangle)\)}}

\put(7,-14.2){\makebox(0,0)[ct]{\(1;2\)}}
\put(7,-16.2){\makebox(0,0)[ct]{\(1;2\)}}
\put(7,-18.2){\makebox(0,0)[ct]{\(1;2\)}}
\multiput(8,-12)(0,-2){3}{\line(-1,-2){1}}
\multiput(7,-14)(0,-2){3}{\circle*{0.1}}

\put(8,-12){\line(1,-1){4}}
\put(8,-12){\line(1,-2){2}}
\put(9.1,-12.8){\makebox(0,0)[lb]{\(3^{\text{rd}}\) bifurcation}}

\put(12.1,-15.8){\makebox(0,0)[lb]{\(2;1\) (not coclass-settled)}}
\put(12.1,-17.8){\makebox(0,0)[lb]{\(1;1\)}}
\multiput(12,-16)(0,-2){2}{\circle*{0.1}}
\multiput(12,-16)(0,-2){1}{\line(0,-1){2}}
\put(12,-18){\vector(0,-1){2}}
\put(12,-20.2){\makebox(0,0)[ct]{\(\mathcal{T}_\ast^6(\langle 512,30599\rangle-\#2;1)\)}}

\put(11,-16){\makebox(0,0)[cc]{\(\mathcal{G}(2,6)\)}}

\put(9.9,-16.2){\makebox(0,0)[rt]{\(2;2\)}}
\put(9.9,-18.2){\makebox(0,0)[rt]{\(1;1\)}}
\multiput(10,-16)(0,-2){2}{\circle*{0.1}}
\multiput(10,-16)(0,-2){1}{\line(0,-1){2}}
\put(10,-18){\vector(0,-1){2}}
\put(10,-20.5){\makebox(0,0)[ct]{\(\mathcal{T}_\ast^6(\langle 512,30599\rangle-\#2;2)\)}}

\put(11,-18.2){\makebox(0,0)[ct]{\(1;2\)}}
\multiput(12,-16)(0,-2){1}{\line(-1,-2){1}}
\multiput(11,-18)(0,-2){1}{\circle*{0.1}}

\put(12,-16){\line(1,-1){4}}
\put(12,-16){\line(1,-2){2}}
\put(13.1,-16.8){\makebox(0,0)[lb]{\(4^{\text{th}}\) bifurcation}}




\end{picture}

\end{figure}

}



The extent to which we constructed the pruned descendant tree
suggests the following conjecture.

\begin{conjecture}
\label{cnj:TwoBifurcation}
Theorem
\ref{thm:TwoBifurcation},
Corollary
\ref{cor:TwoBifurcation}
and Theorem
\ref{thm:2ParamPres}
remain true for an arbitrarily large positive integer \(\ell\),
not necessarily bounded by \(10\).
\end{conjecture}

\begin{remark}
\label{rmk:Ord32Id35}
We must emphasize that the root \(\langle 8,5\rangle\) in Figure
\ref{fig:PeriodicBifurcation222}
was drawn for the sake of completeness only,
and that the mainline
of the coclass tree \(\mathcal{T}^3(\langle 32,35\rangle)\)
is exceptional, since

\begin{itemize}
\item
its root is \textit{not} a descendant of \(G\) and
\item
the TKT of its vertices \(M_{0,k}:=\langle 32,35\rangle(-\#1;1)^k\) with \(k\ge 0\),

\begin{equation}
\label{eqn:LayeredTKTOrd32Id35}
\varkappa(M_{0,k})=\lbrack 1;(S_1,S_2,S_5,S_4,S_3,S_6,S_7);0^7;0\rbrack,
\end{equation}

\noindent
is a permutation with \(5\) fixed points and
only a single \(2\)-cycle \((S_3,S_5)\).

\end{itemize}

\end{remark}



\noindent
One-parameter \textit{polycyclic pc-presentations} for all occurring groups
are given as follows.

\begin{enumerate}

\item
For the mainline vertices of the coclass tree \(\mathcal{T}^3(\langle 32,34\rangle)\) with class \(c\ge 3\),
that is, starting with \(\langle 64,174\rangle\) and excluding the root \(\langle 32,34\rangle\), by

\begin{equation}
\label{eqn:Group34}
\begin{aligned}
\delta^0(G)(-\#1;1)^{c-2}=G_3^c:=\langle x,y,z,s_2,\ldots,s_c,t_2\mid\\
s_2=\lbrack y,x\rbrack,\ t_2=\lbrack z,x\rbrack,\ s_j=\lbrack s_{j-1},x\rbrack\text{ for }3\le j\le c,\\
x^2=1,\ y^2=s_2s_3,\ z^2=t_2,\\
s_j^2=s_{j+1}s_{j+2}\text{ for }2\le j\le c-2,\ s_{c-1}^2=s_c\rangle.
\end{aligned}
\end{equation}

\item
For the mainline vertices of the coclass tree \(\mathcal{T}^4(\langle 128,444\rangle)\) with class \(c\ge 3\) by

\begin{equation}
\label{eqn:Group444}
\begin{aligned}
\delta^1(G)(-\#1;1)^{c-3}=G_4^c:=\langle x,y,z,s_2,\ldots,s_c,t_2,t_3\mid\\
s_2=\lbrack y,x\rbrack,\ t_2=\lbrack z,x\rbrack,\ s_j=\lbrack s_{j-1},x\rbrack\text{ for }3\le j\le c,\ t_3=\lbrack t_2,x\rbrack,\\
x^2=1,\ y^2=s_2s_3,\ z^2=t_2t_3,\\
s_j^2=s_{j+1}s_{j+2}\text{ for }2\le j\le c-2,\ s_{c-1}^2=s_c,\ t_2^2=t_3\rangle.
\end{aligned}
\end{equation}

\item
For the mainline vertices of the coclass tree \(\mathcal{T}^5(\langle 512,30599\rangle)\) with class \(c\ge 4\) by

\begin{equation}
\label{eqn:Group30599}
\begin{aligned}
\delta^2(G)(-\#1;1)^{c-4}=G_5^c:=\langle x,y,z,s_2,\ldots,s_c,t_2,t_3,t_4\mid\\
s_2=\lbrack y,x\rbrack,\ t_2=\lbrack z,x\rbrack,\ s_j=\lbrack s_{j-1},x\rbrack\text{ for }3\le j\le c,\ t_3=\lbrack t_2,x\rbrack,\ t_4=\lbrack t_3,x\rbrack,\\
x^2=1,\ y^2=s_2s_3,\ z^2=t_2t_3,\\
s_j^2=s_{j+1}s_{j+2}\text{ for }2\le j\le c-2,\ s_{c-1}^2=s_c,\ t_2^2=t_3t_4,\ t_3^2=t_4\rangle.
\end{aligned}
\end{equation}

\end{enumerate}



\begin{theorem}
\label{thm:2ParamPres}
For higher coclass \(4\le r\le\ell+3\) the presentations
(\ref{eqn:Group444})
and
(\ref{eqn:Group30599})
can be generalized in the shape of a two-parameter polycyclic pc-presentation
for class \(r-1\le c\le\ell+2\).

\begin{equation}
\label{eqn:TwoParamPres}
\begin{aligned}
\delta^{r-3}(G)(-\#1;1)^{c-r+1}=G_r^c:=\langle x,y,z,s_2,\ldots,s_c,t_2\ldots,t_{r-1}\mid\\
s_2=\lbrack y,x\rbrack,\ t_2=\lbrack z,x\rbrack,\ s_j=\lbrack s_{j-1},x\rbrack\text{ for }3\le j\le c,\ t_k=\lbrack t_{k-1},x\rbrack\text{ for }3\le k\le r-1,\\
x^2=1,\ y^2=s_2s_3,\ z^2=t_2t_3,\\
s_j^2=s_{j+1}s_{j+2}\text{ for }2\le j\le c-2,\ s_{c-1}^2=s_c,\ t_k^2=t_{k+1}t_{k+2}\text{ for }2\le k\le r-3,\ t_{r-2}^2=t_{r-1}\rangle.
\end{aligned}
\end{equation}

\end{theorem}



To obtain a presentation
for the vertices \(\delta^{r-3}(G)(-\#1;1)^{c-r}-\#1;2\), \(c\ge r\),
at depth \(1\) in the distinguished periodic sequence
whose vertices are characterized by the permutation TKT
(\ref{eqn:LayeredTKTPrdSequence}),
we must only add the single relation \(x^2=s_c\) to the presentation
(\ref{eqn:TwoParamPres})
of the mainline vertices of the coclass tree
\(\mathcal{T}^{r}(\delta^{r-3}(G))\)
given in Theorem
\ref{thm:2ParamPres}.



\subsection{Finite \(3\)-groups \(G\) with \(G/G^\prime\simeq (3,3)\)}
\label{ss:3Groups33}

We continue this section
with periodic bifurcations in trees of \(3\)-groups,
which have been discovered in 2012 and 2013
\cite{MaA,MaB,MaC},
inspired by a search for \(3\)-class tower groups of complex quadratic fields
\cite{SoTa,HeSm,BuMa},
which must be Schur \(\sigma\)-groups.

These \(3\)-groups are two-generator groups of coclass at least \(2\)
with elementary abelian commutator quotient of type \((3,3)\).
As shown in Figure
\ref{fig:3GrpTyp33Cocl2}
of \S\
\ref{s:ConcreteExamples},
all such groups are descendants of the extra special group \(\langle 27,3\rangle\).
Among its \(7\) immediate descendants of step size \(2\),
there are only two groups which satisfy the requirements arising
from the arithmetical background.

The two groups \(\langle 243,6\rangle\) and \(\langle 243,8\rangle\)
do not show multifurcation themselves but they are not coclass-settled either,
since their immediate mainline descendants
\(Q=\langle 729,49\rangle\) and \(U=\langle 729,54\rangle\)
possess the required nuclear rank \(\nu=2\) for \textit{bifurcation}.
We constructed an extensive finite part of their pruned descendant trees
\(\mathcal{T}_\ast(G)\), \(G\in\lbrace Q,U\rbrace\),
using the \(p\)-group generation algorithm
\cite{Nm2,Ob,HEO}
as implemented in the computational algebra system Magma
\cite{BCP,BCFS,MAGMA}.

Denote by \(x,y\) the generators of a finite \(3\)-group \(G=\langle x,y\rangle\)
with abelian type invariants \((3,3)\).
We fix an ordering of the four maximal normal subgroups by putting

\begin{equation}
\label{eqn:MaxSbgr33}
\begin{aligned}
H_1=\langle y,G^\prime\rangle,
H_2=\langle x,G^\prime\rangle,
H_3=\langle xy,G^\prime\rangle,
H_4=\langle xy^2,G^\prime\rangle.
\end{aligned}
\end{equation}

Within this subsection,
we make use of special designations for transfer kernel types (TKTs)
which were defined generally in
\cite[pp.403--404]{Ma4}
and more specifically for the present scenario in
\cite{Ma,Ma2}.

We are interested in the unavoidable mainline vertices with TKTs\\
\(\mathrm{c}.18\), \(\varkappa=(0,1,2,2)\), resp. \(\mathrm{c}.21\), \(\varkappa=(2,0,3,4)\),\\
and, above all, in most essential vertices of depth \(1\) forming periodic sequences with TKTs\\
\(\mathrm{E}.6\), \(\varkappa=(1,1,2,2)\) and \(\mathrm{E}.14\), \(\varkappa=(3,1,2,2)\sim(4,1,2,2)\),\\
resp. \(\mathrm{E}.8\), \(\varkappa=(2,2,3,4)\) and \(\mathrm{E}.9\), \(\varkappa=(2,3,3,4)\sim(2,4,3,4)\),\\
and we want to eliminate the numerous and annoying vertices with TKTs\\
\(\mathrm{H}.4\), \(\varkappa=(2,1,2,2)\), resp. \(\mathrm{G}.16\), \(\varkappa=(2,1,3,4)\).

We point out that, for instance \(\mathrm{E}.9\), \(\varkappa=(2,3,3,4)\sim(2,4,3,4)\),
is a shortcut for the layer \(\varkappa_1=(H_2,H_3,H_3,H_4)\sim(H_2,H_4,H_3,H_4)\)
of the complete (layered) TKT \(\varkappa=\lbrack\varkappa_0;\varkappa_1;\varkappa_2\rbrack\).



\begin{remark}
\label{rmk:Sifting33}
We choose the following sifting strategy for reducing the entire descendant tree
\(\mathcal{T}(G)\)
to the pruned descendant tree
\(\mathcal{T}_\ast(G)\).
We filter all vertices which,
firstly, are \(\sigma\)-groups,
and secondly
satisfy one of the conditions in Equations
(\ref{eqn:TKTMainline})
or
(\ref{eqn:TKTPrdSequences}),
whence the process is a combination (F6)\(=\)(F1)\(+\)(F2)\(+\)(F5)
and consists of

\begin{enumerate}
\item
keeping all of the
\(3\) terminal step size-\(2\) descendants,
which are exactly the Schur \(\sigma\)-groups, and
omitting
\(2\) of the \(3\) capable step size-\(2\) descendants
having TKT \(\mathrm{H}.4\), resp. \(\mathrm{G}.16\),
together with their complete descendant trees, 
and
\item
eliminating
\(2\) of the \(5\) terminal step size-\(1\) descendants
having TKT \(\mathrm{c}.18\), resp. \(\mathrm{c}.21\),
and 
\(2\) of the \(3\) capable step size-\(1\) descendants
having TKT \(\mathrm{H}.4\), resp. \(\mathrm{G}.16\),
\end{enumerate}

\noindent
in Theorem
\ref{thm:ThreeBifurcation}.

\end{remark}



For brevity, we give
\(3\)-logarithms of abelian type invariants in the following theorem
and we denote iteration by formal exponents,
for instance,
\(1^3:=(1,1,1)\hat{=}(3,3,3)\),
\((2,1)\hat{=}(9,3)\),
\((2,1)^3:=((2,1),(2,1),(2,1))\),
and
\((j+2,j+1)\hat{=}(3^{j+2},3^{j+1})\).
Further, we eliminate some initial anomalies of generalized identifiers
by putting
\(\langle 243,8\rangle-\#1;1:=\langle 243,8\rangle-\#1;3\),
\(\langle 729,54\rangle-\#1;2\vert 4:=\langle 729,54\rangle-\#1;4\vert 2\),
\(\langle 729,54\rangle-\#2;2\vert 4:=\langle 729,54\rangle-\#2;4\vert 2\),
\(\langle 729,54\rangle-\#2;1:=\langle 729,54\rangle-\#2;3\),
\(\langle 729,49\rangle(-\#2;1-\#1;1)^j-\#1;1:=\langle 729,49\rangle(-\#2;1-\#1;1)^j-\#1;2\),
\(\langle 729,49\rangle(-\#2;1-\#1;1)^j-\#1;4\vert 5\vert 6:=\langle 729,49\rangle(-\#2;1-\#1;1)^j-\#1;5\vert 6\vert 7\),
formally.



\begin{theorem}
\label{thm:ThreeBifurcation}
Let \(\ell\) be a positive integer bounded from above by \(8\).

\begin{enumerate}
\item
In the descendant tree \(\mathcal{T}(G)\) of
\(G=\langle 243,6\rangle\), resp. \(G=\langle 243,8\rangle\),
there exists a unique path of length \(2\ell\),
\[G=\delta^0(G)\leftarrow\delta^1(G)\leftarrow\cdots\leftarrow\delta^{2\ell}(G),\]
of (reverse) directed edges of alternating step sizes \(1\) and \(2\) such that
\(\delta^i(G)=\pi(\delta^{i+1}(G))\), for all \(0\le i\le 2\ell-1\), and
all the vertices with even superscript \(i=2j\), \(j\ge 0\),

\begin{equation}
\label{eqn:EvenDesc}
\delta^{2j}(G)=G(-\#1;1-\#2;1)^j,
\end{equation}

resp. all the vertices with odd superscript \(i=2j+1\) \(j\ge 0\),

\begin{equation}
\label{eqn:OddDesc}
\delta^{2j+1}(G)=G(-\#1;1-\#2;1)^j-\#1;1,
\end{equation}

\noindent
of this path share the following common invariants, respectively:

\begin{itemize}
\item
the uniform (w.r.t. \(i\)) transfer kernel type, containing a total component \(0\hat{=}\delta^i(G)\),

\begin{equation}
\label{eqn:TKTMainline}
\varkappa(\delta^i(G))=\lbrack 1;(0,1,2,2)\text{ resp. }(2,0,3,4);0\rbrack,
\end{equation}

\item
the \(2\)-multiplicator rank and
the nuclear rank,

\begin{equation}
\label{eqn:EvenRanks}
\mu(\delta^{2j}(G))=3,\quad
\nu(\delta^{2j}(G))=1,
\end{equation}

\noindent
resp., giving rise to the bifurcation for odd \(i=2j+1\),

\begin{equation}
\label{eqn:OddRanks}
\mu(\delta^{2j+1}(G))=4,\quad
\nu(\delta^{2j+1}(G))=2,
\end{equation}

\item
and the counters of immediate descendants,

\begin{equation}
\label{eqn:EvenCounters}
N_1(\delta^{2j}(G))=4,\ C_1(\delta^{2j}(G))=4,
\end{equation}

\noindent
resp.

\begin{equation}
\label{eqn:OddCounters}
N_1(\delta^{2j+1}(G))=8,\ C_1(\delta^{2j+1}(G))=3,\quad N_2(\delta^{2j+1}(G))=6,\ C_2(\delta^{2j+1}(G))=3,
\end{equation}

\noindent
determining the local structure of the descendant tree.

\end{itemize}

\item
A few other invariants of the vertices \(\delta^i(G)\) depend on the superscript \(i\),

\begin{itemize}
\item
the \(3\)-logarithm of the order,
the nilpotency class and
the coclass,

\begin{equation}
\label{eqn:EvenLogOrdClCc}
\log_3(\mathrm{ord}(\delta^{2j}(G)))=3j+5,\quad
\mathrm{cl}(\delta^{2j}(G))=2j+3,\quad
\mathrm{cc}(\delta^{2j}(G))=j+2,
\end{equation}

\noindent
resp.

\begin{equation}
\label{eqn:OddLogOrdClCc}
\log_3(\mathrm{ord}(\delta^{2j+1}(G)))=3j+6,\quad
\mathrm{cl}(\delta^{2j+1}(G))=2j+4,\quad
\mathrm{cc}(\delta^{2j+1}(G))=j+2,
\end{equation}

\item
a single component of layer \(\tau_1\)
and the layer \(\tau_2\) of the transfer target type

\begin{equation}
\label{eqn:EvenTTT}
\tau(\delta^{2j}(G))=\lbrack 1^2;((j+2,j+1),1^3,(2,1)^2)\text{ resp. }((2,1),(j+2,j+1),(2,1)^2);(j+1,j+1,1)\rbrack,
\end{equation}

\noindent
resp.

\begin{equation}
\label{eqn:OddTTT}
\tau(\delta^{2j+1}(G))=\lbrack 1^2;((j+2,j+2),1^3,(2,1)^2)\text{ resp. }((2,1),(j+2,j+2),(2,1)^2);(j+2,j+1,1)\rbrack.
\end{equation}

\end{itemize}

\end{enumerate}

\end{theorem}



Theorem
\ref{thm:ThreeBifurcation}
provided the scaffold of the pruned descendant tree \(\mathcal{T}_\ast(G)\)
of \(G=\langle 243,n\rangle\), for \(n\in\lbrace 6,8\rbrace\),
with mainlines and periodic bifurcations.

With respect to number theoretic applications, however,
the following Corollaries
\ref{cor:ThreeBifurcation}
and
\ref{cor:Covers}
are of the greatest importance.

\begin{corollary}
\label{cor:ThreeBifurcation}
Let \(0\le i\le 2\ell\) be a non-negative integer.

Whereas the vertices with even superscript \(i=2j\), \(j\ge 0\),
that is, \(\delta^{2j}(G)=G(-\#1;1-\#2;1)^j\),
are merely links in the distinguished path,
the vertices with odd superscript \(i=2j+1\) \(j\ge 0\),
that is, \(\delta^{2j+1}(G)=G(-\#1;1-\#2;1)^j-\#1;1\),
reveal the essential periodic bifurcations
with the following properties.

\begin{enumerate}
\item
The regular component \(\mathcal{T}^{j+2}(\delta^{2j+1}(G))\)
of the descendant tree \(\mathcal{T}(\delta^{2j+1}(G))\)
is a coclass tree which contains the mainline,
\[M_{j,k}:=\delta^{2j+1}(G)(-\#1;1)^k\text{ with }k\ge 0,\]
which entirely consists of \(\sigma\)-groups,
and three distinguished periodic sequences
whose vertices
\[V_{j,k}:=\delta^{2j+1}(G)(-\#1;1)^k-\#1;2\vert 4\vert 6\text{ resp. }4\vert 5\vert 6\text{ with }k\ge 0\]
are \(\sigma\)-groups exactly for even \(k=2k^\prime\ge 0\)
and are characterized by the following TKTs
\(\varkappa(V_{j,k})=\lbrack 1;\varkappa_1;0\rbrack\)
with layer \(\varkappa_1\) given by

\begin{equation}
\label{eqn:TKTPrdSequences}
\varkappa_1\in\lbrace(1,1,2,2),(3,1,2,2),(4,1,2,2)\rbrace\text{ resp. }
\varkappa_1\in\lbrace(2,2,3,4),(2,3,3,4),(2,4,3,4)\rbrace,
\end{equation}

\noindent
which deviate from the mainline TKT of Equation
(\ref{eqn:TKTMainline})
in a single component only.
\item
The irregular component \(\mathcal{T}^{j+3}(\delta^{2j+1}(G))\)
of the descendant tree \(\mathcal{T}(\delta^{2j+1}(G))\)
is a forest which
contains a bunch of \(3\) isolated Schur \(\sigma\)-groups
\[S_j:=\delta^{2j+1}(G)-\#2;2\vert 4\vert 6\text{ resp. }4\vert 5\vert 6,\]
which possess the same TKTs as in Equation
(\ref{eqn:TKTPrdSequences}),
and additionally contains the root of the next coclass tree
\(\mathcal{T}^{j+3}(\delta^{2j+1}(G)-\#2;1)\),
where \(\delta^{2j+1}(G)-\#2;1=\delta^{2(j+1)}(G)\),
whose mainline vertices \(\delta^{2(j+1)}(G)(-\#1;1)^k\) with \(k\ge 0\)
share the TKT in Equation
(\ref{eqn:TKTMainline}).
\end{enumerate}

\end{corollary}

The metabelian \(3\)-groups forming the three distinguished periodic sequences
\[V_{0,2k}=\delta^1(G)(-\#1;1)^{2k}-\#1;2\vert 4\vert 6\text{ resp. }4\vert 5\vert 6\text{ with }k\ge 0\]
of the pruned coclass tree \(\mathcal{T}_\ast^2(G)\)
in Corollary
\ref{cor:ThreeBifurcation},
for \(i=0\),
belong to the few groups for which all immediate descendants
with respect to the parent definition (P4) are known.
(We did not use this kind of descendants up to now.)
Since all groups in \(\mathcal{T}(G)\setminus\mathcal{T}^2(G)\)
are of derived length \(3\), the set of these descendants
can be defined in the following way.

\begin{definition}
\label{dfn:Covers}
Let \(P\) be a finite \(p\)-group,
then the set of all finite \(p\)-groups \(D\)
whose second derived quotient \(D/D^{\prime\prime}\) is isomorphic to \(P\)
is called the \textit{cover} \(\mathrm{cov}(P)\) of \(P\).
The subset \(\mathrm{cov}_\ast(P)\) consisting of
all Schur \(\sigma\)-groups in \(\mathrm{cov}(P)\)
is called the \textit{balanced cover} of \(P\).
\end{definition}

\begin{corollary}
\label{cor:Covers}
For \(0\le k\le\ell\),
the group \(V_{0,2k}\), which does not have a balanced presentation,
possesses a finite cover of cardinality \(\#\mathrm{cov}(V_{0,2k})=k+1\)
and a unique Schur \(\sigma\)-group in its balanced cover \(\#\mathrm{cov}_\ast(V_{0,2k})=1\).
More precisely, the covers are given explicitly by

\begin{equation}
\label{eqn:Covers}
\begin{aligned}
\mathrm{cov}(V_{0,2k})=\lbrace V_{j,2(k-j)}\mid 1\le j\le k\rbrace\cup\lbrace S_k\rbrace,\\
\mathrm{cov}_\ast(V_{0,2k})=\lbrace S_k\rbrace.
\end{aligned}
\end{equation}

\end{corollary}

The arrows in Figures
\ref{fig:MultiPeriodQAdmissible}
and
\ref{fig:MultiPeriodUAdmissible}
indicate the projections \(\pi\) from all members \(D\) of a cover \(\mathrm{cov}(P)\)
onto the common metabelianization \(P\),
that is, in the sense of the parent definition (P4),
from the descendants \(D\) onto the parent \(P=\pi(D)\).



\begin{proof}
(of Theorem
\ref{thm:ThreeBifurcation},
Corollary
\ref{cor:ThreeBifurcation},
Corollary
\ref{cor:Covers}
and Theorem
\ref{thm:EvenBranches})\\
The \(p\)-group generation algorithm
\cite{Nm2,Ob,HEO},
which is implemented in the computational algebra system Magma
\cite{BCP,BCFS,MAGMA},
was used for constructing the pruned descendant trees
\(\mathcal{T}_\ast(G)\) with roots \(G=\langle 243,6\vert 8\rangle\)
which were defined as the disjoint union of all pruned coclass trees
\(\mathcal{T}_\ast^{j+2}(\delta^{2j+1}(G))\) of the
descendants \(\delta^{2j+1}(G)=G(-\#1;1-\#2;1)^j-\#1;1\), \(0\le j\le 10\),
of \(G\) as roots, together with \(4\) siblings
in the irregular component \(\mathcal{T}_\ast^{j+3}(\delta^{2j+1}(G))\),
\(3\) of them Schur \(\sigma\)-groups with \(\mu=2\) and \(\nu=0\).
Using the strict periodicity
\cite{dS,EkLg}
of each pruned coclass tree \(\mathcal{T}_\ast^{j+2}(\delta^{2j+1}(G))\),
which turned out to be of length \(2\),
the vertical construction was terminated at nilpotency class \(19\),
considerably deeper than the point where periodicity sets in.
The horizontal construction was extended up to coclass \(10\),
where the consumption of CPU time became daunting.
\end{proof}



{\tiny

\begin{figure}[ht]
\caption{Periodic Bifurcations in \(\mathcal{T}_\ast(\langle 243,6\rangle)\)}
\label{fig:MultiPeriodQAdmissible}


\setlength{\unitlength}{0.8cm}
\begin{picture}(18,22)(-6,-21)

\put(-5,0.5){\makebox(0,0)[cb]{Order}}
\put(-5,0){\line(0,-1){18}}
\multiput(-5.1,0)(0,-2){10}{\line(1,0){0.2}}
\put(-5.2,0){\makebox(0,0)[rc]{\(243\)}}
\put(-4.8,0){\makebox(0,0)[lc]{\(3^5\)}}
\put(-5.2,-2){\makebox(0,0)[rc]{\(729\)}}
\put(-4.8,-2){\makebox(0,0)[lc]{\(3^6\)}}
\put(-5.2,-4){\makebox(0,0)[rc]{\(2\,187\)}}
\put(-4.8,-4){\makebox(0,0)[lc]{\(3^7\)}}
\put(-5.2,-6){\makebox(0,0)[rc]{\(6\,561\)}}
\put(-4.8,-6){\makebox(0,0)[lc]{\(3^8\)}}
\put(-5.2,-8){\makebox(0,0)[rc]{\(19\,683\)}}
\put(-4.8,-8){\makebox(0,0)[lc]{\(3^9\)}}
\put(-5.2,-10){\makebox(0,0)[rc]{\(59\,049\)}}
\put(-4.8,-10){\makebox(0,0)[lc]{\(3^{10}\)}}
\put(-5.2,-12){\makebox(0,0)[rc]{\(177\,147\)}}
\put(-4.8,-12){\makebox(0,0)[lc]{\(3^{11}\)}}
\put(-5.2,-14){\makebox(0,0)[rc]{\(531\,441\)}}
\put(-4.8,-14){\makebox(0,0)[lc]{\(3^{12}\)}}
\put(-5.2,-16){\makebox(0,0)[rc]{\(1\,594\,323\)}}
\put(-4.8,-16){\makebox(0,0)[lc]{\(3^{13}\)}}
\put(-5.2,-18){\makebox(0,0)[rc]{\(4\,782\,969\)}}
\put(-4.8,-18){\makebox(0,0)[lc]{\(3^{14}\)}}
\put(-5,-18){\vector(0,-1){2}}

\put(0.1,0.2){\makebox(0,0)[lb]{\(\langle 6\rangle\)}}
\put(0.1,-1.8){\makebox(0,0)[lb]{\(\langle 49\rangle\) (not coclass-settled)}}
\put(1.1,-2.8){\makebox(0,0)[lb]{\(1^{\text{st}}\) bifurcation}}
\put(0.1,-3.8){\makebox(0,0)[lb]{\(\langle 285\rangle\)}}
\put(0.1,-5.8){\makebox(0,0)[lb]{\(1;1\)}}
\put(0.1,-7.8){\makebox(0,0)[lb]{\(1;1\)}}
\put(0.1,-9.8){\makebox(0,0)[lb]{\(1;1\)}}
\put(0.1,-11.8){\makebox(0,0)[lb]{\(1;1\)}}
\multiput(0,0)(0,-2){7}{\circle*{0.1}}
\multiput(0,0)(0,-2){6}{\line(0,-1){2}}
\put(0,-12){\vector(0,-1){2}}
\put(-0.2,-14.2){\makebox(0,0)[rt]{\(\mathcal{T}_\ast^2(\langle 243,6\rangle)\)}}

\put(-3,-4.2){\makebox(0,0)[ct]{\(\langle 290\rangle\)}}
\put(-3,-8.2){\makebox(0,0)[ct]{\(1;6\)}}
\put(-3,-12.2){\makebox(0,0)[ct]{\(1;6\)}}
\multiput(0,-2)(0,-4){3}{\line(-3,-2){3}}
\multiput(-3,-4)(0,-4){3}{\circle*{0.1}}

\put(-2,-4.4){\makebox(0,0)[ct]{\(\langle 289\rangle\)}}
\put(-2,-8.4){\makebox(0,0)[ct]{\(1;5\)}}
\put(-2,-12.4){\makebox(0,0)[ct]{\(1;5\)}}
\multiput(0,-2)(0,-4){3}{\line(-1,-1){2}}
\multiput(-2,-4)(0,-4){3}{\circle*{0.1}}

\put(-1,-4.2){\makebox(0,0)[ct]{\(\langle 288\rangle\)}}
\put(-1,-8.2){\makebox(0,0)[ct]{\(1;4\)}}
\put(-1,-12.2){\makebox(0,0)[ct]{\(1;4\)}}
\multiput(0,-2)(0,-4){3}{\line(-1,-2){1}}
\multiput(-1,-4)(0,-4){3}{\circle*{0.1}}


\multiput(1,-6)(1,0){3}{\vector(-2,1){4}}
\multiput(5,-12)(1,0){3}{\vector(-2,1){8}}
\multiput(9,-18)(1,0){3}{\vector(-2,1){12}}

\put(0,-2){\line(1,-1){4}}

\put(4.1,-5.8){\makebox(0,0)[lb]{\(2;1\)}}
\put(4.1,-7.8){\makebox(0,0)[lb]{\(1;1\) (not coclass-settled)}}
\put(5.1,-8.8){\makebox(0,0)[lb]{\(2^{\text{nd}}\) bifurcation}}
\put(4.1,-9.8){\makebox(0,0)[lb]{\(1;2\)}}
\put(4.1,-11.8){\makebox(0,0)[lb]{\(1;1\)}}
\put(4.1,-13.8){\makebox(0,0)[lb]{\(1;1\)}}
\multiput(3.95,-6.05)(0,-2){5}{\framebox(0.1,0.1){}}
\multiput(4,-6)(0,-2){4}{\line(0,-1){2}}
\put(4,-14){\vector(0,-1){2}}
\put(3.8,-16.2){\makebox(0,0)[rt]{\(\mathcal{T}_\ast^3(\langle 729,49\rangle-\#2;1)\)}}

\put(1,-6.2){\makebox(0,0)[ct]{\(2;6\)}}
\put(1,-10.2){\makebox(0,0)[ct]{\(1;7\)}}
\put(1,-14.2){\makebox(0,0)[ct]{\(1;6\)}}
\put(0,-2){\line(1,-4){1}}
\multiput(4,-8)(0,-4){2}{\line(-3,-2){3}}
\multiput(0.95,-6.05)(0,-4){3}{\framebox(0.1,0.1){}}

\put(2,-6.4){\makebox(0,0)[ct]{\(2;5\)}}
\put(2,-10.4){\makebox(0,0)[ct]{\(1;6\)}}
\put(2,-14.4){\makebox(0,0)[ct]{\(1;5\)}}
\put(0,-2){\line(1,-2){2}}
\multiput(4,-8)(0,-4){2}{\line(-1,-1){2}}
\multiput(1.95,-6.05)(0,-4){3}{\framebox(0.1,0.1){}}

\put(3,-6.2){\makebox(0,0)[ct]{\(2;4\)}}
\put(3,-10.2){\makebox(0,0)[ct]{\(1;5\)}}
\put(3,-14.2){\makebox(0,0)[ct]{\(1;4\)}}
\put(0,-2){\line(3,-4){3}}
\multiput(4,-8)(0,-4){2}{\line(-1,-2){1}}
\multiput(2.95,-6.05)(0,-4){3}{\framebox(0.1,0.1){}}

\put(4,-8){\line(1,-1){4}}

\put(8.1,-11.8){\makebox(0,0)[lb]{\(2;1\)}}
\put(8.1,-13.8){\makebox(0,0)[lb]{\(1;1\) (not coclass-settled)}}
\put(9.1,-14.8){\makebox(0,0)[lb]{\(3^{\text{rd}}\) bifurcation}}
\put(8.1,-15.8){\makebox(0,0)[lb]{\(1;2\)}}
\multiput(7.95,-12.05)(0,-2){3}{\framebox(0.1,0.1){}}
\multiput(8,-12)(0,-2){2}{\line(0,-1){2}}
\put(8,-16){\vector(0,-1){2}}
\put(7.8,-18.2){\makebox(0,0)[rt]{\(\mathcal{T}_\ast^4(\langle 729,49\rangle-\#2;1-\#1;1-\#2;1)\)}}

\put(5,-12.2){\makebox(0,0)[ct]{\(2;6\)}}
\put(5,-16.2){\makebox(0,0)[ct]{\(1;7\)}}
\put(4,-8){\line(1,-4){1}}
\multiput(8,-14)(0,-4){1}{\line(-3,-2){3}}
\multiput(4.95,-12.05)(0,-4){2}{\framebox(0.1,0.1){}}

\put(6,-12.4){\makebox(0,0)[ct]{\(2;5\)}}
\put(6,-16.4){\makebox(0,0)[ct]{\(1;6\)}}
\put(4,-8){\line(1,-2){2}}
\multiput(8,-14)(0,-4){1}{\line(-1,-1){2}}
\multiput(5.95,-12.05)(0,-4){2}{\framebox(0.1,0.1){}}

\put(7,-12.2){\makebox(0,0)[ct]{\(2;4\)}}
\put(7,-16.2){\makebox(0,0)[ct]{\(1;5\)}}
\put(4,-8){\line(3,-4){3}}
\multiput(8,-14)(0,-4){1}{\line(-1,-2){1}}
\multiput(6.95,-12.05)(0,-4){2}{\framebox(0.1,0.1){}}

\put(8,-14){\line(1,-1){4}}

\put(12.1,-17.8){\makebox(0,0)[lb]{\(2;1\)}}
\multiput(11.95,-18.05)(0,-2){1}{\framebox(0.1,0.1){}}
\put(12,-18){\vector(0,-1){2}}
\put(11.8,-19.9){\makebox(0,0)[rt]{\(\mathcal{T}_\ast^5(\langle 729,49\rangle-\#2;1-\#1;1-\#2;1-\#1;1-\#2;1)\)}}

\put(9,-18.2){\makebox(0,0)[ct]{\(2;6\)}}
\put(8,-14){\line(1,-4){1}}
\multiput(8.95,-18.05)(0,-4){1}{\framebox(0.1,0.1){}}

\put(10,-18.4){\makebox(0,0)[ct]{\(2;5\)}}
\put(8,-14){\line(1,-2){2}}
\multiput(9.95,-18.05)(0,-4){1}{\framebox(0.1,0.1){}}

\put(11,-18.2){\makebox(0,0)[ct]{\(2;4\)}}
\put(8,-14){\line(3,-4){3}}
\multiput(10.95,-18.05)(0,-4){1}{\framebox(0.1,0.1){}}

\put(-5,-20.9){\makebox(0,0)[cc]{\textbf{TKT:}}}
\put(-3,-21){\makebox(0,0)[cc]{\(\varkappa_1\)}}
\put(-2,-21){\makebox(0,0)[cc]{\(\varkappa_2\)}}
\put(-1,-21){\makebox(0,0)[cc]{\(\varkappa_3\)}}
\put(0,-21){\makebox(0,0)[cc]{\(\varkappa_0\)}}
\put(1,-21){\makebox(0,0)[cc]{\(\varkappa_1\)}}
\put(2,-21){\makebox(0,0)[cc]{\(\varkappa_2\)}}
\put(3,-21){\makebox(0,0)[cc]{\(\varkappa_3\)}}
\put(4,-21){\makebox(0,0)[cc]{\(\varkappa_0\)}}
\put(5,-21){\makebox(0,0)[cc]{\(\varkappa_1\)}}
\put(6,-21){\makebox(0,0)[cc]{\(\varkappa_2\)}}
\put(7,-21){\makebox(0,0)[cc]{\(\varkappa_3\)}}
\put(8,-21){\makebox(0,0)[cc]{\(\varkappa_0\)}}
\put(9,-21){\makebox(0,0)[cc]{\(\varkappa_1\)}}
\put(10,-21){\makebox(0,0)[cc]{\(\varkappa_2\)}}
\put(11,-21){\makebox(0,0)[cc]{\(\varkappa_3\)}}
\put(12,-21){\makebox(0,0)[cc]{\(\varkappa_0\)}}
\put(-5.8,-21.2){\framebox(18.6,0.6){}}

\multiput(-2,-4.25)(0,-4){3}{\oval(3,1.5)}

\multiput(2,-6.25)(4,-6){3}{\oval(3,1.5)}

\end{picture}

\end{figure}

}



{\tiny

\begin{figure}[ht]
\caption{Periodic Bifurcations in \(\mathcal{T}_\ast(\langle 243,8\rangle)\)}
\label{fig:MultiPeriodUAdmissible}


\setlength{\unitlength}{0.8cm}
\begin{picture}(18,22)(-6,-21)

\put(-5,0.5){\makebox(0,0)[cb]{Order}}
\put(-5,0){\line(0,-1){18}}
\multiput(-5.1,0)(0,-2){10}{\line(1,0){0.2}}
\put(-5.2,0){\makebox(0,0)[rc]{\(243\)}}
\put(-4.8,0){\makebox(0,0)[lc]{\(3^5\)}}
\put(-5.2,-2){\makebox(0,0)[rc]{\(729\)}}
\put(-4.8,-2){\makebox(0,0)[lc]{\(3^6\)}}
\put(-5.2,-4){\makebox(0,0)[rc]{\(2\,187\)}}
\put(-4.8,-4){\makebox(0,0)[lc]{\(3^7\)}}
\put(-5.2,-6){\makebox(0,0)[rc]{\(6\,561\)}}
\put(-4.8,-6){\makebox(0,0)[lc]{\(3^8\)}}
\put(-5.2,-8){\makebox(0,0)[rc]{\(19\,683\)}}
\put(-4.8,-8){\makebox(0,0)[lc]{\(3^9\)}}
\put(-5.2,-10){\makebox(0,0)[rc]{\(59\,049\)}}
\put(-4.8,-10){\makebox(0,0)[lc]{\(3^{10}\)}}
\put(-5.2,-12){\makebox(0,0)[rc]{\(177\,147\)}}
\put(-4.8,-12){\makebox(0,0)[lc]{\(3^{11}\)}}
\put(-5.2,-14){\makebox(0,0)[rc]{\(531\,441\)}}
\put(-4.8,-14){\makebox(0,0)[lc]{\(3^{12}\)}}
\put(-5.2,-16){\makebox(0,0)[rc]{\(1\,594\,323\)}}
\put(-4.8,-16){\makebox(0,0)[lc]{\(3^{13}\)}}
\put(-5.2,-18){\makebox(0,0)[rc]{\(4\,782\,969\)}}
\put(-4.8,-18){\makebox(0,0)[lc]{\(3^{14}\)}}
\put(-5,-18){\vector(0,-1){2}}

\put(0.1,0.2){\makebox(0,0)[lb]{\(\langle 8\rangle\)}}
\put(0.1,-1.8){\makebox(0,0)[lb]{\(\langle 54\rangle\) (not coclass-settled)}}
\put(1.1,-2.8){\makebox(0,0)[lb]{\(1^{\text{st}}\) bifurcation}}
\put(0.1,-3.8){\makebox(0,0)[lb]{\(\langle 303\rangle\)}}
\put(0.1,-5.8){\makebox(0,0)[lb]{\(1;1\)}}
\put(0.1,-7.8){\makebox(0,0)[lb]{\(1;1\)}}
\put(0.1,-9.8){\makebox(0,0)[lb]{\(1;1\)}}
\put(0.1,-11.8){\makebox(0,0)[lb]{\(1;1\)}}
\multiput(0,0)(0,-2){7}{\circle*{0.1}}
\multiput(0,0)(0,-2){6}{\line(0,-1){2}}
\put(0,-12){\vector(0,-1){2}}
\put(-0.2,-14.2){\makebox(0,0)[rt]{\(\mathcal{T}_\ast^2(\langle 243,8\rangle)\)}}

\put(-3,-4.2){\makebox(0,0)[ct]{\(\langle 306\rangle\)}}
\put(-3,-8.2){\makebox(0,0)[ct]{\(1;6\)}}
\put(-3,-12.2){\makebox(0,0)[ct]{\(1;6\)}}
\multiput(0,-2)(0,-4){3}{\line(-3,-2){3}}
\multiput(-3,-4)(0,-4){3}{\circle*{0.1}}

\put(-2,-4.4){\makebox(0,0)[ct]{\(\langle 302\rangle\)}}
\put(-2,-8.4){\makebox(0,0)[ct]{\(1;4\)}}
\put(-2,-12.4){\makebox(0,0)[ct]{\(1;4\)}}
\multiput(0,-2)(0,-4){3}{\line(-1,-1){2}}
\multiput(-2,-4)(0,-4){3}{\circle*{0.1}}

\put(-1,-4.2){\makebox(0,0)[ct]{\(\langle 304\rangle\)}}
\put(-1,-8.2){\makebox(0,0)[ct]{\(1;2\)}}
\put(-1,-12.2){\makebox(0,0)[ct]{\(1;2\)}}
\multiput(0,-2)(0,-4){3}{\line(-1,-2){1}}
\multiput(-1,-4)(0,-4){3}{\circle*{0.1}}


\multiput(1,-6)(1,0){3}{\vector(-2,1){4}}
\multiput(5,-12)(1,0){3}{\vector(-2,1){8}}
\multiput(9,-18)(1,0){3}{\vector(-2,1){12}}

\put(0,-2){\line(1,-1){4}}

\put(4.1,-5.8){\makebox(0,0)[lb]{\(2;3\)}}
\put(4.1,-7.8){\makebox(0,0)[lb]{\(1;1\) (not coclass-settled)}}
\put(5.1,-8.8){\makebox(0,0)[lb]{\(2^{\text{nd}}\) bifurcation}}
\put(4.1,-9.8){\makebox(0,0)[lb]{\(1;1\)}}
\put(4.1,-11.8){\makebox(0,0)[lb]{\(1;1\)}}
\put(4.1,-13.8){\makebox(0,0)[lb]{\(1;1\)}}
\multiput(3.95,-6.05)(0,-2){5}{\framebox(0.1,0.1){}}
\multiput(4,-6)(0,-2){4}{\line(0,-1){2}}
\put(4,-14){\vector(0,-1){2}}
\put(3.8,-16.2){\makebox(0,0)[rt]{\(\mathcal{T}_\ast^3(\langle 729,54\rangle-\#2;3)\)}}

\put(1,-6.2){\makebox(0,0)[ct]{\(2;6\)}}
\put(1,-10.2){\makebox(0,0)[ct]{\(1;6\)}}
\put(1,-14.2){\makebox(0,0)[ct]{\(1;6\)}}
\put(0,-2){\line(1,-4){1}}
\multiput(4,-8)(0,-4){2}{\line(-3,-2){3}}
\multiput(0.95,-6.05)(0,-4){3}{\framebox(0.1,0.1){}}

\put(2,-6.4){\makebox(0,0)[ct]{\(2;2\)}}
\put(2,-10.4){\makebox(0,0)[ct]{\(1;4\)}}
\put(2,-14.4){\makebox(0,0)[ct]{\(1;4\)}}
\put(0,-2){\line(1,-2){2}}
\multiput(4,-8)(0,-4){2}{\line(-1,-1){2}}
\multiput(1.95,-6.05)(0,-4){3}{\framebox(0.1,0.1){}}

\put(3,-6.2){\makebox(0,0)[ct]{\(2;4\)}}
\put(3,-10.2){\makebox(0,0)[ct]{\(1;2\)}}
\put(3,-14.2){\makebox(0,0)[ct]{\(1;2\)}}
\put(0,-2){\line(3,-4){3}}
\multiput(4,-8)(0,-4){2}{\line(-1,-2){1}}
\multiput(2.95,-6.05)(0,-4){3}{\framebox(0.1,0.1){}}

\put(4,-8){\line(1,-1){4}}

\put(8.1,-11.8){\makebox(0,0)[lb]{\(2;1\)}}
\put(8.1,-13.8){\makebox(0,0)[lb]{\(1;1\) (not coclass-settled)}}
\put(9.1,-14.8){\makebox(0,0)[lb]{\(3^{\text{rd}}\) bifurcation}}
\put(8.1,-15.8){\makebox(0,0)[lb]{\(1;1\)}}
\multiput(7.95,-12.05)(0,-2){3}{\framebox(0.1,0.1){}}
\multiput(8,-12)(0,-2){2}{\line(0,-1){2}}
\put(8,-16){\vector(0,-1){2}}
\put(7.8,-18.2){\makebox(0,0)[rt]{\(\mathcal{T}_\ast^4(\langle 729,54\rangle-\#2;3-\#1;1-\#2;1)\)}}

\put(5,-12.2){\makebox(0,0)[ct]{\(2;6\)}}
\put(5,-16.2){\makebox(0,0)[ct]{\(1;6\)}}
\put(4,-8){\line(1,-4){1}}
\multiput(8,-14)(0,-4){1}{\line(-3,-2){3}}
\multiput(4.95,-12.05)(0,-4){2}{\framebox(0.1,0.1){}}

\put(6,-12.4){\makebox(0,0)[ct]{\(2;4\)}}
\put(6,-16.4){\makebox(0,0)[ct]{\(1;4\)}}
\put(4,-8){\line(1,-2){2}}
\multiput(8,-14)(0,-4){1}{\line(-1,-1){2}}
\multiput(5.95,-12.05)(0,-4){2}{\framebox(0.1,0.1){}}

\put(7,-12.2){\makebox(0,0)[ct]{\(2;2\)}}
\put(7,-16.2){\makebox(0,0)[ct]{\(1;2\)}}
\put(4,-8){\line(3,-4){3}}
\multiput(8,-14)(0,-4){1}{\line(-1,-2){1}}
\multiput(6.95,-12.05)(0,-4){2}{\framebox(0.1,0.1){}}

\put(8,-14){\line(1,-1){4}}

\put(12.1,-17.8){\makebox(0,0)[lb]{\(2;1\)}}
\multiput(11.95,-18.05)(0,-2){1}{\framebox(0.1,0.1){}}
\put(12,-18){\vector(0,-1){2}}
\put(11.8,-19.9){\makebox(0,0)[rt]{\(\mathcal{T}_\ast^5(\langle 729,54\rangle-\#2;3-\#1;1-\#2;1-\#1;1-\#2;1)\)}}

\put(9,-18.2){\makebox(0,0)[ct]{\(2;6\)}}
\put(8,-14){\line(1,-4){1}}
\multiput(8.95,-18.05)(0,-4){1}{\framebox(0.1,0.1){}}

\put(10,-18.4){\makebox(0,0)[ct]{\(2;4\)}}
\put(8,-14){\line(1,-2){2}}
\multiput(9.95,-18.05)(0,-4){1}{\framebox(0.1,0.1){}}

\put(11,-18.2){\makebox(0,0)[ct]{\(2;2\)}}
\put(8,-14){\line(3,-4){3}}
\multiput(10.95,-18.05)(0,-4){1}{\framebox(0.1,0.1){}}

\put(-5,-20.9){\makebox(0,0)[cc]{\textbf{TKT:}}}
\put(-3,-21){\makebox(0,0)[cc]{\(\varkappa_1\)}}
\put(-2,-21){\makebox(0,0)[cc]{\(\varkappa_2\)}}
\put(-1,-21){\makebox(0,0)[cc]{\(\varkappa_3\)}}
\put(0,-21){\makebox(0,0)[cc]{\(\varkappa_0\)}}
\put(1,-21){\makebox(0,0)[cc]{\(\varkappa_1\)}}
\put(2,-21){\makebox(0,0)[cc]{\(\varkappa_2\)}}
\put(3,-21){\makebox(0,0)[cc]{\(\varkappa_3\)}}
\put(4,-21){\makebox(0,0)[cc]{\(\varkappa_0\)}}
\put(5,-21){\makebox(0,0)[cc]{\(\varkappa_1\)}}
\put(6,-21){\makebox(0,0)[cc]{\(\varkappa_2\)}}
\put(7,-21){\makebox(0,0)[cc]{\(\varkappa_3\)}}
\put(8,-21){\makebox(0,0)[cc]{\(\varkappa_0\)}}
\put(9,-21){\makebox(0,0)[cc]{\(\varkappa_1\)}}
\put(10,-21){\makebox(0,0)[cc]{\(\varkappa_2\)}}
\put(11,-21){\makebox(0,0)[cc]{\(\varkappa_3\)}}
\put(12,-21){\makebox(0,0)[cc]{\(\varkappa_0\)}}
\put(-5.8,-21.2){\framebox(18.6,0.6){}}

\multiput(-2,-4.25)(0,-4){3}{\oval(3,1.5)}

\multiput(2,-6.25)(4,-6){3}{\oval(3,1.5)}

\end{picture}

\end{figure}
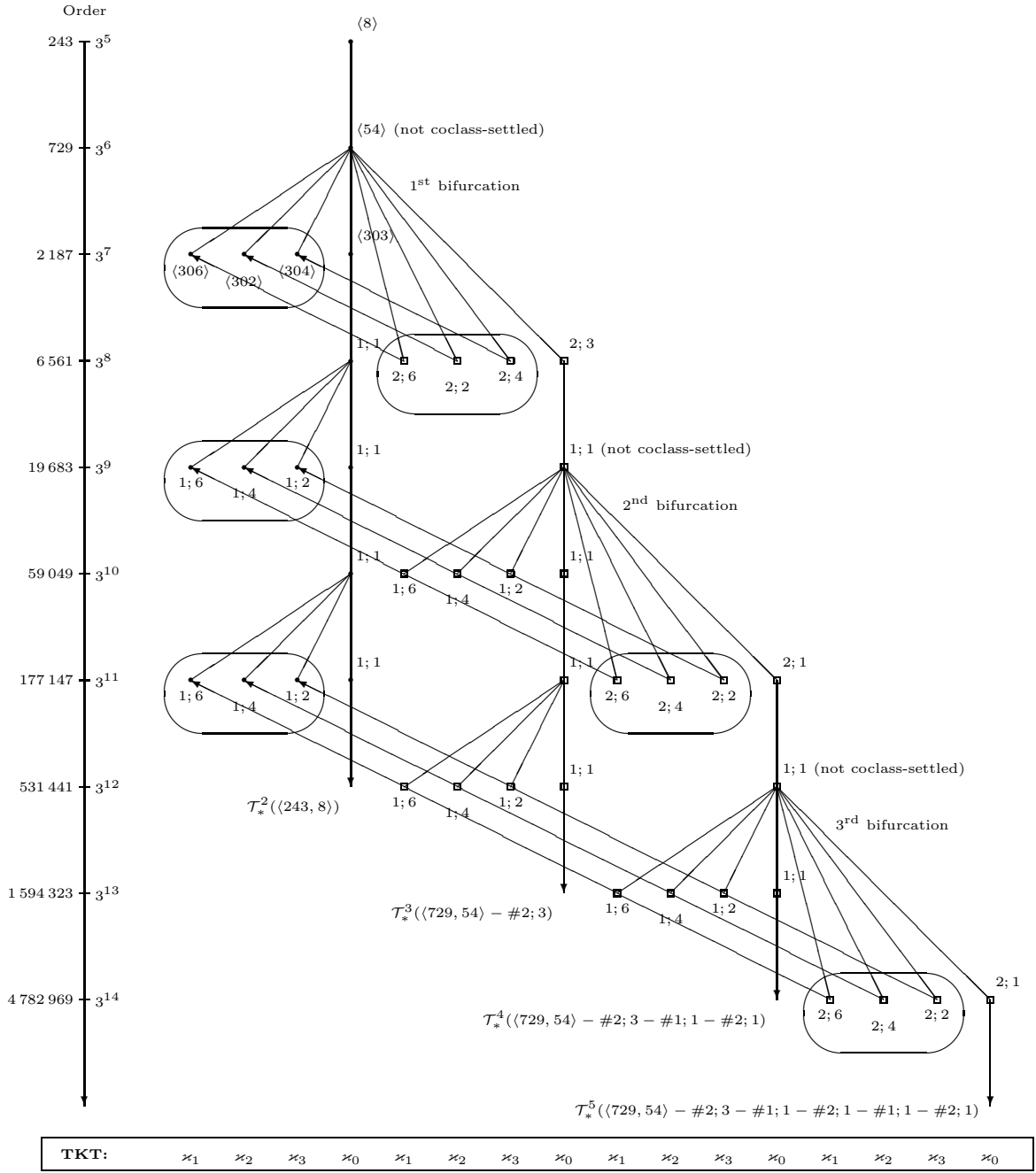

}



Within the frame of our computations, the periodicity was not
restriced to bifurcations only:
It seems that the pruned (or maybe even the entire) descendant trees
\(\mathcal{T}_\ast(\delta^{2j+1}(G))\) are all isomorphic to
\(\mathcal{T}_\ast(\delta^1(G))\) as graphs.
This is visualized impressively by the Figures
\ref{fig:MultiPeriodQAdmissible}
and
\ref{fig:MultiPeriodUAdmissible},
where the following notation (not to be confused with layers) is used
\[\varkappa_1=(4,1,2,2),\ \varkappa_2=(3,1,2,2),\ \varkappa_3=(1,1,2,2),\ \varkappa_0=(0,1,2,2),\]
resp.
\[\varkappa_1=(2,4,3,4),\ \varkappa_2=(2,3,3,4),\ \varkappa_3=(2,2,3,4),\ \varkappa_0=(2,0,3,4).\]



Similarly as in the previous section,
the extent to which we constructed the pruned descendant trees
suggests the following conjecture.

\begin{conjecture}
\label{cnj:ThreeBifurcation}
Theorem
\ref{thm:ThreeBifurcation},
Corollary
\ref{cor:ThreeBifurcation}
and Corollary
\ref{cor:Covers}
remain true for an arbitrarily large positive integer \(\ell\),
not necessarily bounded by \(8\).
\end{conjecture}



\noindent
One-parameter \textit{polycyclic pc-presentations} for the groups
in the first three pruned coclass trees of \(\mathcal{T}_\ast(\langle 243,6\rangle)\)
are given as follows.

\begin{enumerate}

\item
For the metabelian vertices of the pruned coclass tree \(\mathcal{T}_\ast^2(\delta^{0}(G))\) with class \(c\ge 5\),
that is, starting with \(\langle 2187,285\rangle\)
and excluding the root \(\delta^{0}(G)=\langle 243,6\rangle\) and its descendant \(Q=\langle 729,49\rangle\), by

\begin{equation}
\label{eqn:Scnd3ClsGrp}
\begin{aligned}
 \delta^{0}(G)(-\#1;1)^{c-3} = G_2^c(0,0),\ G_2^c(z,w) := \langle\ x,y,s_2,\ldots,s_c,t_3\ \mid \\
 s_2=\lbrack y,x\rbrack,\ s_j=\lbrack s_{j-1},x\rbrack \text{ for } 3\le j\le c,\ t_3=\lbrack s_2,y\rbrack, \\
 s_j^3=s_{j+2}^2s_{j+3} \text{ for } 2\le j\le c-3,\ s_{c-2}^3=s_c^2,\ t_3^3=1, \\
 x^3 = s_c^w,\ y^3 = s_3^2s_4s_c^z\ \rangle.
\end{aligned}
\end{equation}

\item
For the non-metabelian vertices of the pruned coclass tree \(\mathcal{T}_\ast^3(\delta^{2}(G))\) with class \(c\ge 5\),
and including the Schur \(\sigma\)-groups, which are siblings of the root, by

\begin{equation}
\label{eqn:3TwrGrp9748GS}
\begin{aligned}
 \delta^{2}(G)(-\#1;1)^{c-5} = G_3^c(0,0),\ G_3^c(z,w) := \langle\ x,y,s_2,\ldots,s_c,t_3,u_5\ \mid \\
 s_2=\lbrack y,x\rbrack,\ s_j=\lbrack s_{j-1},x\rbrack \text{ for } 3\le j\le c,\ t_3=\lbrack s_2,y\rbrack, \\
 u_5=\lbrack s_3,y\rbrack=\lbrack s_4,y\rbrack,\ \lbrack s_3,s_2\rbrack=u_5^2,\ t_3^3=u_5^2, \\
 s_2^3=s_4^2s_5u_5,\ s_j^3=s_{j+2}^2s_{j+3} \text{ for } 3\le j\le c-3,\ s_{c-2}^3=s_c^2, \\
 x^3 = s_c^w,\ y^3 = s_3^2s_4s_c^z\ \rangle.
\end{aligned}
\end{equation}

\item
For the non-metabelian vertices of the pruned coclass tree \(\mathcal{T}_\ast^4(\delta^{4}(G))\) with class \(c\ge 7\),
and including the Schur \(\sigma\)-groups, which are siblings of the root, by

\begin{equation}
\label{eqn:3TwrGrp262744ES}
\begin{aligned}
 \delta^{4}(G)(-\#1;1)^{c-7} = G_4^c(0,0),\ G_4^c(z,w) := \langle\ x,y,s_2,\ldots,s_c,t_3,u_5,u_7\ \mid \\
 s_2=\lbrack y,x\rbrack,\ s_j=\lbrack s_{j-1},x\rbrack \text{ for } 3\le j\le c,\ t_3=\lbrack s_2,y\rbrack, \\
 u_5=\lbrack s_4,y\rbrack,\ u_7=\lbrack s_6,y\rbrack,\ \lbrack s_3,s_2\rbrack=u_5^2u_7^2,\ \lbrack s_3,y\rbrack=u_5u_7^2, \\
 \lbrack s_5,y\rbrack=u_7^2,\ \lbrack s_4,s_2\rbrack=u_7^2,\ \lbrack s_5,s_2\rbrack=u_7^2,\ \lbrack s_4,s_3\rbrack=u_7, \\
 s_2^3=s_4^2s_5u_5,\ s_3^3=s_5^2s_6u_7^2,\ t_3^3=u_5^2u_7^2, u_5^3=u_7^2, \\
 s_j^3=s_{j+2}^2s_{j+3} \text{ for } 4\le j\le c-3,\ s_{c-2}^3=s_c^2, \\
 x^3 = s_c^w,\ y^3 = s_3^2s_4s_c^z\ \rangle.
\end{aligned}
\end{equation}

\end{enumerate}

The parameter \(c\) is the nilpotency class of the group,
and the parameters \(0\le w\le 1\) and \(0\le z\le 2\) determine

\begin{itemize}
\item
the location of the group on the descendant tree,
and
\item
the transfer kernel type (TKT) of the group, as follows:
\end{itemize}

\(G_r^c(0,0)\) lies on the mainline (this is the so-called \textit{mainline principle})
and has TKT c.18, \(\varkappa=(0,1,2,2)\),
whereas all the other groups belong to periodic sequences or are isolated Schur \(\sigma\)-groups:\\
\(G_r^c(0,1)\) possesses TKT E.6, \(\varkappa=(1,1,2,2)\),\\
\(G_r^c(1,0)\) and \(G_r^c(2,0)\) have TKT H.4, \(\varkappa=(2,1,2,2)\), and lie outside of the pruned tree,\\
\(G_r^c(1,1)\) and \(G_r^c(2,1)\) have TKT E.14, \(\varkappa=(3,1,2,2)\sim(4,1,2,2)\).

\noindent
In Figure
\ref{fig:3TwrGrp9748GS},
resp.
\ref{fig:3TwrGrp262744ES},
we have drawn the lattice of normal subgroups of \(G_3^5(z,w)\), resp. \(G_4^7(z,w)\).
The \textit{upper} and \textit{lower central series},
\(\zeta(G)\), \(\gamma(G)\),
of these groups form subgraphs
whose relative position justifies the names of these series,
as visualized impressively by Figures
\ref{fig:3TwrGrp9748GS}
and
\ref{fig:3TwrGrp262744ES}.



{\tiny

\begin{figure}[hb]
\caption{Normal Lattice and Central Series of \(G_3^5(z,w)\)}
\label{fig:3TwrGrp9748GS}


\setlength{\unitlength}{0.9cm}
\begin{picture}(15,21)(-5,0)

\put(-4,20.3){\makebox(0,0)[cb]{order \(3^n\)}}
\put(-4,18){\vector(0,1){2}}
\put(-4.2,18){\makebox(0,0)[rc]{\(6561\)}}
\put(-3.8,18){\makebox(0,0)[lc]{\(3^8\)}}
\put(-4.2,16){\makebox(0,0)[rc]{\(2187\)}}
\put(-3.8,16){\makebox(0,0)[lc]{\(3^7\)}}
\put(-4.2,14){\makebox(0,0)[rc]{\(729\)}}
\put(-3.8,14){\makebox(0,0)[lc]{\(3^6\)}}
\put(-4.2,12){\makebox(0,0)[rc]{\(243\)}}
\put(-3.8,12){\makebox(0,0)[lc]{\(3^5\)}}
\put(-4.2,10){\makebox(0,0)[rc]{\(81\)}}
\put(-3.8,10){\makebox(0,0)[lc]{\(3^4\)}}
\put(-4.2,8){\makebox(0,0)[rc]{\(27\)}}
\put(-3.8,8){\makebox(0,0)[lc]{\(3^3\)}}
\put(-4.2,6){\makebox(0,0)[rc]{\(9\)}}
\put(-3.8,6){\makebox(0,0)[lc]{\(3^2\)}}
\put(-4.2,4){\makebox(0,0)[rc]{\(3\)}}
\put(-4.2,2){\makebox(0,0)[rc]{\(1\)}}
\multiput(-4.1,2)(0,2){9}{\line(1,0){0.2}}
\put(-4,2){\line(0,1){16}}

\multiput(5,18)(0,-4){2}{\line(1,0){2}}
\put(6,16.5){\vector(0,1){1.5}}
\put(6,16.25){\makebox(0,0)[cc]{first}}
\put(6,15.75){\makebox(0,0)[cc]{stage}}
\put(6,15.5){\vector(0,-1){1.5}}
\put(6,9.5){\vector(0,1){4.5}}
\put(6,9.25){\makebox(0,0)[cc]{second}}
\put(6,8.75){\makebox(0,0)[cc]{stage}}
\put(6,8.5){\vector(0,-1){4.5}}
\put(6,3.5){\vector(0,1){0.5}}
\put(6,3.25){\makebox(0,0)[cc]{third}}
\put(6,2.75){\makebox(0,0)[cc]{stage}}
\put(6,2.5){\vector(0,-1){0.5}}
\multiput(5,2)(0,2){2}{\line(1,0){2}}

\put(0.3,2){\makebox(0,0)[rc]{\(\zeta_0(G)\)}}
\put(0.5,2){\circle*{0.2}}
\put(0.7,2){\makebox(0,0)[lc]{\(\gamma_6(G)=1\)}}

\put(-1.2,4){\makebox(0,0)[rb]{\(G^{\prime\prime}\)}}
\put(-1.2,3.8){\makebox(0,0)[rt]{\(u_5\)}}
\put(-1,4){\circle*{0.2}}
\multiput(0,4)(1,0){3}{\circle*{0.1}}
\put(2.2,4){\makebox(0,0)[lc]{\(s_5\)}}

\multiput(0.5,6)(-1.5,-2){2}{\line(3,-4){1.5}}
\multiput(0.5,6)(-0.5,-2){2}{\line(1,-4){0.5}}
\multiput(0.5,6)(0.5,-2){2}{\line(-1,-4){0.5}}
\multiput(0.5,6)(1.5,-2){2}{\line(-3,-4){1.5}}
\put(-2.7,6){\makebox(0,0)[rc]{\(t_3\)}}
\multiput(-2.5,6)(1,0){4}{\circle*{0.1}}
\put(0.7,6){\makebox(0,0)[lc]{\(\gamma_5(G)\)}}

\multiput(-1,8)(-1.5,-2){2}{\line(3,-4){1.5}}
\multiput(-1,8)(-0.5,-2){2}{\line(1,-4){0.5}}
\multiput(-1,8)(0.5,-2){2}{\line(-1,-4){0.5}}
\multiput(-1,8)(1.5,-2){2}{\line(-3,-4){1.5}}
\put(-1.2,8){\makebox(0,0)[rc]{\(\zeta_1(G)\)}}
\multiput(-1,8)(1,0){4}{\circle*{0.1}}
\put(2.2,8){\makebox(0,0)[lb]{\(\gamma_4(G)\)}}
\put(2.2,7.8){\makebox(0,0)[lt]{\(s_4\)}}

\multiput(0.5,10)(-1.5,-2){2}{\line(3,-4){1.5}}
\multiput(0.5,10)(-0.5,-2){2}{\line(1,-4){0.5}}
\multiput(0.5,10)(0.5,-2){2}{\line(-1,-4){0.5}}
\multiput(0.5,10)(1.5,-2){2}{\line(-3,-4){1.5}}
\put(0.3,10){\makebox(0,0)[rc]{\(\zeta_2(G)\)}}
\multiput(0.5,10)(1,0){4}{\circle*{0.1}}
\put(3.7,10){\makebox(0,0)[lc]{\(s_3\)}}

\multiput(2,12)(-1.5,-2){2}{\line(3,-4){1.5}}
\multiput(2,12)(-0.5,-2){2}{\line(1,-4){0.5}}
\multiput(2,12)(0.5,-2){2}{\line(-1,-4){0.5}}
\multiput(2,12)(1.5,-2){2}{\line(-3,-4){1.5}}
\put(1.8,12){\makebox(0,0)[rc]{\(\zeta_3(G)\)}}
\put(2,12){\circle*{0.1}}
\put(2.2,12){\makebox(0,0)[lc]{\(\gamma_3(G)\)}}

\put(2,14){\line(0,-1){2}}
\put(1.8,14){\makebox(0,0)[rc]{\(\zeta_4(G)\)}}
\put(2,14){\circle*{0.2}}
\put(2.2,14){\makebox(0,0)[lb]{\(\gamma_2(G)=G^\prime\)}}
\put(2.2,13.8){\makebox(0,0)[lt]{\(s_2\)}}

\put(0.3,16){\makebox(0,0)[rb]{\(H_1\)}}
\put(0.3,15.8){\makebox(0,0)[rt]{\(y\)}}
\put(1.3,16){\makebox(0,0)[rb]{\(H_3\)}}
\multiput(0.5,16)(1,0){4}{\circle*{0.1}}
\put(2.7,16){\makebox(0,0)[lb]{\(H_4\)}}
\put(3.7,16){\makebox(0,0)[lb]{\(H_2\)}}
\put(3.7,15.8){\makebox(0,0)[lt]{\(x\)}}

\multiput(2,18)(-1.5,-2){2}{\line(3,-4){1.5}}
\multiput(2,18)(-0.5,-2){2}{\line(1,-4){0.5}}
\multiput(2,18)(0.5,-2){2}{\line(-1,-4){0.5}}
\multiput(2,18)(1.5,-2){2}{\line(-3,-4){1.5}}
\put(1.8,18){\makebox(0,0)[rc]{\(\zeta_5(G)\)}}
\put(2,18){\circle*{0.2}}
\put(2.2,18){\makebox(0,0)[lc]{\(\gamma_1(G)=G\)}}

\end{picture}

\end{figure}
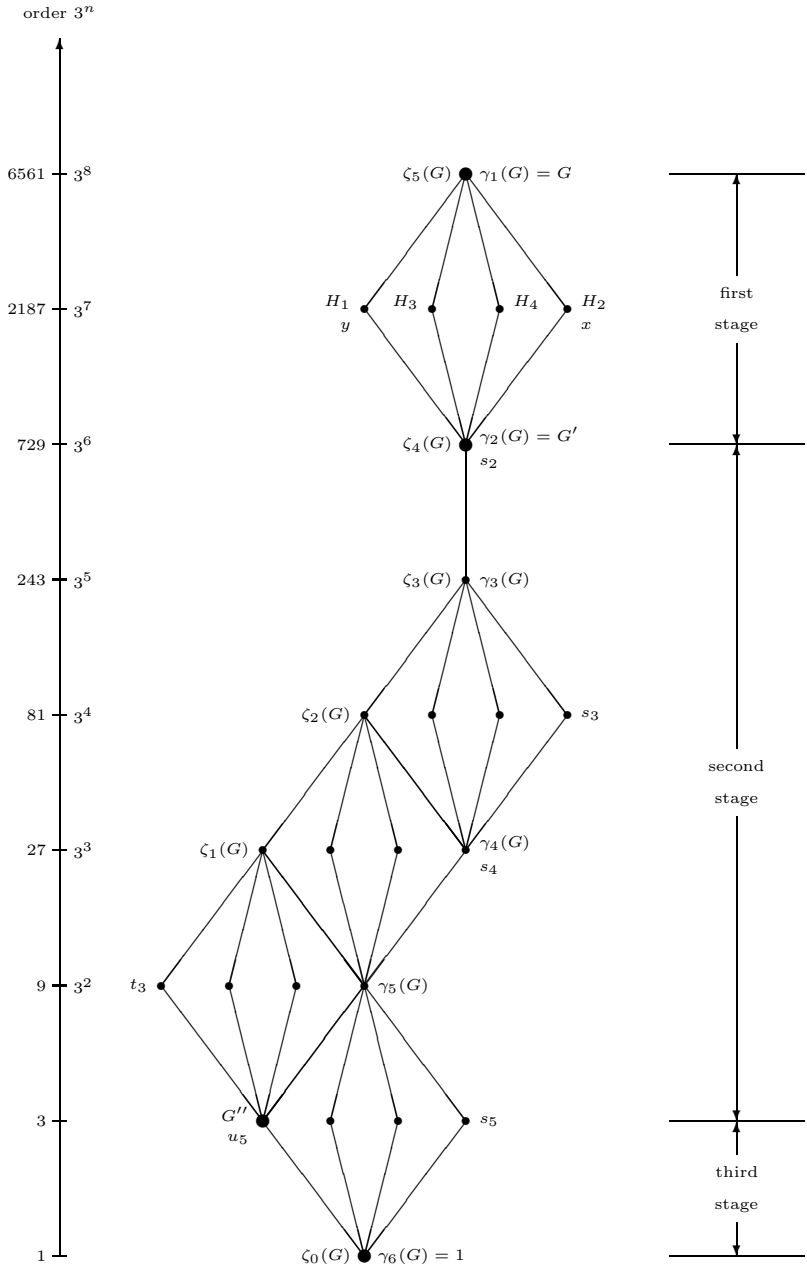

}



{\tiny

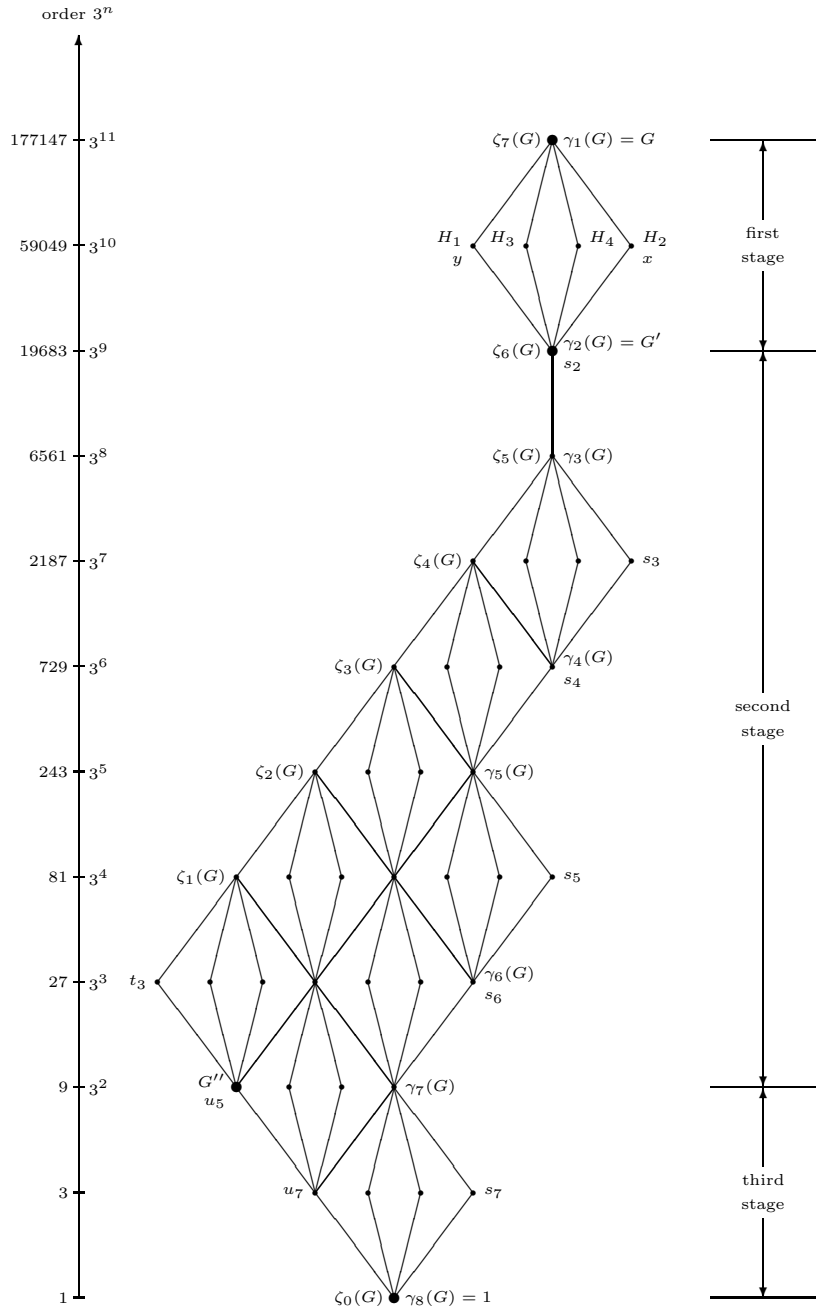
\begin{figure}[ht]
\caption{Normal Lattice and Central Series of \(G_4^7(z,w)\)}
\label{fig:3TwrGrp262744ES}


\setlength{\unitlength}{0.7cm}
\begin{picture}(15,27)(-8,0)

\put(-7,26.3){\makebox(0,0)[cb]{order \(3^n\)}}
\put(-7,24){\vector(0,1){2}}
\put(-7.2,24){\makebox(0,0)[rc]{\(177147\)}}
\put(-6.8,24){\makebox(0,0)[lc]{\(3^{11}\)}}
\put(-7.2,22){\makebox(0,0)[rc]{\(59049\)}}
\put(-6.8,22){\makebox(0,0)[lc]{\(3^{10}\)}}
\put(-7.2,20){\makebox(0,0)[rc]{\(19683\)}}
\put(-6.8,20){\makebox(0,0)[lc]{\(3^9\)}}
\put(-7.2,18){\makebox(0,0)[rc]{\(6561\)}}
\put(-6.8,18){\makebox(0,0)[lc]{\(3^8\)}}
\put(-7.2,16){\makebox(0,0)[rc]{\(2187\)}}
\put(-6.8,16){\makebox(0,0)[lc]{\(3^7\)}}
\put(-7.2,14){\makebox(0,0)[rc]{\(729\)}}
\put(-6.8,14){\makebox(0,0)[lc]{\(3^6\)}}
\put(-7.2,12){\makebox(0,0)[rc]{\(243\)}}
\put(-6.8,12){\makebox(0,0)[lc]{\(3^5\)}}
\put(-7.2,10){\makebox(0,0)[rc]{\(81\)}}
\put(-6.8,10){\makebox(0,0)[lc]{\(3^4\)}}
\put(-7.2,8){\makebox(0,0)[rc]{\(27\)}}
\put(-6.8,8){\makebox(0,0)[lc]{\(3^3\)}}
\put(-7.2,6){\makebox(0,0)[rc]{\(9\)}}
\put(-6.8,6){\makebox(0,0)[lc]{\(3^2\)}}
\put(-7.2,4){\makebox(0,0)[rc]{\(3\)}}
\put(-7.2,2){\makebox(0,0)[rc]{\(1\)}}
\multiput(-7.1,2)(0,2){12}{\line(1,0){0.2}}
\put(-7,2){\line(0,1){22}}

\multiput(5,24)(0,-4){2}{\line(1,0){2}}
\put(6,22.5){\vector(0,1){1.5}}
\put(6,22.25){\makebox(0,0)[cc]{first}}
\put(6,21.75){\makebox(0,0)[cc]{stage}}
\put(6,21.5){\vector(0,-1){1.5}}
\put(6,13.5){\vector(0,1){6.5}}
\put(6,13.25){\makebox(0,0)[cc]{second}}
\put(6,12.75){\makebox(0,0)[cc]{stage}}
\put(6,12.5){\vector(0,-1){6.5}}
\put(6,4.5){\vector(0,1){1.5}}
\put(6,4.25){\makebox(0,0)[cc]{third}}
\put(6,3.75){\makebox(0,0)[cc]{stage}}
\put(6,3.5){\vector(0,-1){1.5}}
\multiput(5,2)(0,4){2}{\line(1,0){2}}

\put(-1.2,2){\makebox(0,0)[rc]{\(\zeta_0(G)\)}}
\put(-1,2){\circle*{0.2}}
\put(-0.8,2){\makebox(0,0)[lc]{\(\gamma_8(G)=1\)}}

\put(-2.7,4){\makebox(0,0)[rc]{\(u_7\)}}
\multiput(-2.5,4)(1,0){4}{\circle*{0.1}}
\put(0.7,4){\makebox(0,0)[lc]{\(s_7\)}}

\multiput(-1,6)(-1.5,-2){2}{\line(3,-4){1.5}}
\multiput(-1,6)(-0.5,-2){2}{\line(1,-4){0.5}}
\multiput(-1,6)(0.5,-2){2}{\line(-1,-4){0.5}}
\multiput(-1,6)(1.5,-2){2}{\line(-3,-4){1.5}}
\put(-4.2,6){\makebox(0,0)[rb]{\(G^{\prime\prime}\)}}
\put(-4.2,5.8){\makebox(0,0)[rt]{\(u_5\)}}
\put(-4,6){\circle*{0.2}}
\multiput(-3,6)(1,0){3}{\circle*{0.1}}
\put(-0.8,6){\makebox(0,0)[lc]{\(\gamma_7(G)\)}}

\multiput(-2.5,8)(-1.5,-2){2}{\line(3,-4){1.5}}
\multiput(-2.5,8)(-0.5,-2){2}{\line(1,-4){0.5}}
\multiput(-2.5,8)(0.5,-2){2}{\line(-1,-4){0.5}}
\multiput(-2.5,8)(1.5,-2){2}{\line(-3,-4){1.5}}
\put(-5.7,8){\makebox(0,0)[rc]{\(t_3\)}}
\multiput(-5.5,8)(1,0){4}{\circle*{0.1}}
\multiput(-1.5,8)(1,0){3}{\circle*{0.1}}
\put(0.7,8){\makebox(0,0)[lb]{\(\gamma_6(G)\)}}
\put(0.7,7.8){\makebox(0,0)[lt]{\(s_6\)}}

\multiput(-4,10)(-1.5,-2){2}{\line(3,-4){1.5}}
\multiput(-4,10)(-0.5,-2){2}{\line(1,-4){0.5}}
\multiput(-4,10)(0.5,-2){2}{\line(-1,-4){0.5}}
\multiput(-4,10)(1.5,-2){2}{\line(-3,-4){1.5}}
\multiput(-1,10)(-1.5,-2){2}{\line(3,-4){1.5}}
\multiput(-1,10)(-0.5,-2){2}{\line(1,-4){0.5}}
\multiput(-1,10)(0.5,-2){2}{\line(-1,-4){0.5}}
\multiput(-1,10)(1.5,-2){2}{\line(-3,-4){1.5}}
\put(-4.2,10){\makebox(0,0)[rc]{\(\zeta_1(G)\)}}
\multiput(-4,10)(1,0){4}{\circle*{0.1}}
\multiput(0,10)(1,0){3}{\circle*{0.1}}
\put(2.2,10){\makebox(0,0)[lc]{\(s_5\)}}

\multiput(-2.5,12)(-1.5,-2){2}{\line(3,-4){1.5}}
\multiput(-2.5,12)(-0.5,-2){2}{\line(1,-4){0.5}}
\multiput(-2.5,12)(0.5,-2){2}{\line(-1,-4){0.5}}
\multiput(-2.5,12)(1.5,-2){2}{\line(-3,-4){1.5}}
\multiput(0.5,12)(-1.5,-2){2}{\line(3,-4){1.5}}
\multiput(0.5,12)(-0.5,-2){2}{\line(1,-4){0.5}}
\multiput(0.5,12)(0.5,-2){2}{\line(-1,-4){0.5}}
\multiput(0.5,12)(1.5,-2){2}{\line(-3,-4){1.5}}
\put(-2.7,12){\makebox(0,0)[rc]{\(\zeta_2(G)\)}}
\multiput(-2.5,12)(1,0){4}{\circle*{0.1}}
\put(0.7,12){\makebox(0,0)[lc]{\(\gamma_5(G)\)}}

\multiput(-1,14)(-1.5,-2){2}{\line(3,-4){1.5}}
\multiput(-1,14)(-0.5,-2){2}{\line(1,-4){0.5}}
\multiput(-1,14)(0.5,-2){2}{\line(-1,-4){0.5}}
\multiput(-1,14)(1.5,-2){2}{\line(-3,-4){1.5}}
\put(-1.2,14){\makebox(0,0)[rc]{\(\zeta_3(G)\)}}
\multiput(-1,14)(1,0){4}{\circle*{0.1}}
\put(2.2,14){\makebox(0,0)[lb]{\(\gamma_4(G)\)}}
\put(2.2,13.8){\makebox(0,0)[lt]{\(s_4\)}}

\multiput(0.5,16)(-1.5,-2){2}{\line(3,-4){1.5}}
\multiput(0.5,16)(-0.5,-2){2}{\line(1,-4){0.5}}
\multiput(0.5,16)(0.5,-2){2}{\line(-1,-4){0.5}}
\multiput(0.5,16)(1.5,-2){2}{\line(-3,-4){1.5}}
\put(0.3,16){\makebox(0,0)[rc]{\(\zeta_4(G)\)}}
\multiput(0.5,16)(1,0){4}{\circle*{0.1}}
\put(3.7,16){\makebox(0,0)[lc]{\(s_3\)}}

\multiput(2,18)(-1.5,-2){2}{\line(3,-4){1.5}}
\multiput(2,18)(-0.5,-2){2}{\line(1,-4){0.5}}
\multiput(2,18)(0.5,-2){2}{\line(-1,-4){0.5}}
\multiput(2,18)(1.5,-2){2}{\line(-3,-4){1.5}}
\put(1.8,18){\makebox(0,0)[rc]{\(\zeta_5(G)\)}}
\put(2,18){\circle*{0.1}}
\put(2.2,18){\makebox(0,0)[lc]{\(\gamma_3(G)\)}}

\put(2,20){\line(0,-1){2}}
\put(1.8,20){\makebox(0,0)[rc]{\(\zeta_6(G)\)}}
\put(2,20){\circle*{0.2}}
\put(2.2,20){\makebox(0,0)[lb]{\(\gamma_2(G)=G^\prime\)}}
\put(2.2,19.8){\makebox(0,0)[lt]{\(s_2\)}}

\put(0.3,22){\makebox(0,0)[rb]{\(H_1\)}}
\put(0.3,21.8){\makebox(0,0)[rt]{\(y\)}}
\put(1.3,22){\makebox(0,0)[rb]{\(H_3\)}}
\multiput(0.5,22)(1,0){4}{\circle*{0.1}}
\put(2.7,22){\makebox(0,0)[lb]{\(H_4\)}}
\put(3.7,22){\makebox(0,0)[lb]{\(H_2\)}}
\put(3.7,21.8){\makebox(0,0)[lt]{\(x\)}}

\multiput(2,24)(-1.5,-2){2}{\line(3,-4){1.5}}
\multiput(2,24)(-0.5,-2){2}{\line(1,-4){0.5}}
\multiput(2,24)(0.5,-2){2}{\line(-1,-4){0.5}}
\multiput(2,24)(1.5,-2){2}{\line(-3,-4){1.5}}
\put(1.8,24){\makebox(0,0)[rc]{\(\zeta_7(G)\)}}
\put(2,24){\circle*{0.2}}
\put(2.2,24){\makebox(0,0)[lc]{\(\gamma_1(G)=G\)}}

\end{picture}

\end{figure}

}



\textit{Generators}
\(x,y\in G\setminus G^\prime\),
\(s_2,s_3,t_3,\ldots\in G^\prime\setminus G^{\prime\prime}\), and
\(u_5,u_7\in G^{\prime\prime}\),
are carefully selected independently from individual isomorphism types
and placed in locations which illustrate the structure of the groups.
Furthermore, the \textit{normal lattice of the metabelianization} \(G/G^{\prime\prime}\)
is also included as a subgraph simply by putting \(u_5=1\).

We conclude with a theorem concerning the central series
and some fundamental properties of the Schur \(\sigma\)-groups
which we encountered among all the groups under investigation.

\begin{theorem}
\label{thm:EvenBranches}

Let \(0\le j\le 7\) be an integer.
There exist exactly \(6\) pairwise non-isomorphic groups \(G\)
of order \(3^{3j+8}\), class \(2j+5\), coclass \(j+3\),
having fixed derived length \(3\),
such that

\begin{enumerate}

\item
the factors of their upper central series are given by

\[
\zeta_{j+1}(G)/\zeta_j(G)\simeq
\begin{cases}
(3,3) & \text{ for } j=2j+4, \\
(3) & \text{ for } 1\le j\le 2j+3, \\
(3,3^{j+2}) & \text{ for } j=0, 
\end{cases}
\]

\item
their second derived group \(G^{\prime\prime}<\zeta_1(G)\) is central and cyclic of order \(3^{j+1}\).

\end{enumerate}

\noindent
Furthermore,

\begin{itemize}

\item
they are Schur \(\sigma\)-groups with automorphism group
\(\mathrm{Aut}(G)\) of order \(2\cdot 3^{4j+10}\),

\item
the factors of their lower central series are given by

\[
\gamma_j(G)/\gamma_{j+1}(G)\simeq
\begin{cases}
(3,3) & \text{ for odd } 1\le j\le 2j+5, \\
(3) & \text{ for even } 2\le j\le 2j+4, 
\end{cases}
\]

\item
their metabelianization \(G/G^{\prime\prime}\)
is of order \(3^{2j+7}\), class \(2j+5\) and of fixed coclass \(2\),

\item
their biggest metabelian generalized predecessor,
that is the \((2j+1)\)th generalized parent,
is given by either \(\langle 729,49\rangle\) or \(\langle 729,54\rangle\).

\end{itemize}

\end{theorem}



\section{Conclusion}
\label{s:Conclusion}

We emphasize that the results of section
\ref{ss:3Groups33}
provide the background for considerably stronger assertions
than those made in
\cite{BuMa}
(which were, however, sufficient already to disprove erroneous claims in
\cite{SoTa,HeSm}).
Firstly, since they concern four TKTs E.6, E.14, E.8, E.9 instead of just TKT E.9,
and secondly, since they apply to varying odd nilpotency class \(5\le\mathrm{cl}(G)\le 19\)
instead of just class \(5\).



\section{Acknowledgements}
\label{s:Acknowledgements}

We gratefully acknowledge that our research is supported by the
Austrian Science Fund (FWF): P 26008-N25.
We are indebted to the anonymous referees for valuable suggestions improving
the exposition and readability.




\end{document}